\documentclass[reqno]{amsart}
\usepackage{amsmath,amssymb,amsthm}
\usepackage{mathtools}
\usepackage{graphicx}
\usepackage[unicode,hidelinks]{hyperref}
\usepackage{color}

\newcommand{\eps}{\varepsilon}
\newcommand{\la}{\lambda}
\newcommand{\rh}{\varrho}

\newcommand{\bs}{\boldsymbol}
\newcommand{\ve}{\bs v}
\newcommand{\we}{\bs w}
\newcommand{\uu}{\bs u}
\newcommand{\n}{\bs n}
\newcommand{\x}{\bs x}
\newcommand{\z}{\bs z}
\newcommand{\fe}{\bs g}
\newcommand{\fit}{\bs\varphi}

\newcommand{\je}{\bs j}

\newcommand{\bb}{\mathbb}
\newcommand{\R}{\bb{R}}
\newcommand{\N}{\bb{N}}
\newcommand{\Wb}{\bb{W}}
\newcommand{\D}{\bb{D}}
\newcommand{\B}{\bb{B}}
\newcommand{\A}{\bb{A}}
\newcommand{\PP}{\bb{P}}
\newcommand{\T}{\bb{T}}
\newcommand{\I}{\bb{I}}
\newcommand{\Sb}{\bb{S}}

\newcommand{\CO}{\mathcal{C}}
\newcommand{\OO}{\mathcal{O}}
\newcommand{\pr}{{\rm p}}

\newcommand{\f}[2]{\frac{#1}{#2}}
\newcommand{\PD}{\R^{d\times d}_{>0}}
\newcommand{\Sym}{\R^{d\times d}_{\rm sym}}
\newcommand{\Q}{Q}
\newcommand{\SigT}{\Sigma}
\newcommand{\Dv}{\D\ve}
\newcommand{\Wv}{\bb{W}\ve}
\newcommand{\dd}[1]{\,\mathrm{d}{#1}}
\newcommand{\ii}{\int_{\Omega}}
\newcommand{\nn}{\nabla}
\newcommand{\wc}{\rightharpoonup}
\newcommand{\wcs}{\overset{*}{\wc}}
\newcommand{\embl}{\hookrightarrow}
\newcommand{\thlim}{\theta_{0}^*}

\usepackage{accents}
\renewcommand*{\dot}[1]{\accentset{\bullet}{#1}}
\newcommand{\der}[1]{\accentset{\bs\circ}{#1}}

\DeclareMathOperator{\di}{div}
\DeclareMathOperator{\tr}{tr}
\DeclareMathOperator{\spa}{span}
\DeclareMathOperator*{\esssup}{ess\,sup}

\DeclarePairedDelimiter{\norm}{\lVert}{\rVert}
\DeclarePairedDelimiter{\scal}{\langle}{\rangle}

\newtheorem{definition}{Definition}[section]
\newtheorem{theorem}{Theorem}[section]
\newtheorem*{theorem*}{Main result}
\newtheorem{lemma}{Lemma}[section]
\numberwithin{equation}{section}

\title[Analysis of thermoviscoelastic fluids]
{Coupling the Navier-Stokes-Fourier equations with the Johnson-Segalman stress-diffusive viscoelastic model: global-in-time and large-data analysis}
\author[M. Bathory]{Michal Bathory}
\address{Michal Bathory\\
	University of Vienna\\
	Faculty of Mathematics\\
	Oskar-Morgenstern-Platz 1\\
	1090 Wien\\
	Austria}
\email{michal.bathory@univie.ac.at}
\author[M. Bul\'i\v cek]{Miroslav Bul\'i\v cek}
\address{Miroslav Bul\'i\v cek\\
	Charles University\\
	Faculty of Mathematics and Physics\\
	Mathematical Institute\\
	Sokolovsk\'a 83\\186 75 Praha 8\\Czech Republic}
\email{mbul8060@karlin.mff.cuni.cz}
\author[J. M\'{a}lek]{Josef M\'{a}lek}
\address{Josef M\'{a}lek\\
	Charles University\\
	Faculty of Mathematics and Physics\\
	Mathematical Institute\\
	Sokolovsk\'a 83\\186 75 Praha 8\\Czech Republic}
\email{malek@karlin.mff.cuni.cz}
\thanks{Michal Bathory is recipient of an APART-MINT Fellowship (No.~11976) of the~Austrian Academy of Sciences (\"OAW) and also acknowledges support from the~project No.\ 1652119 financed by the~Charles University Grant Agency (GAUK). Michal Bathory,  Miroslav Bul\'i\v cek and Josef M\'{a}lek thank to the~project No. 20-11027X financed by the~Czech Science foundation (GA\v{C}R).  Miroslav Bul\'{\i}\v{c}ek and Josef M\'{a}lek are members of the~Ne\v{c}as Center for Mathematical Modelling.}
\subjclass{35A23, 76A10, 76D03}
\keywords{viscoelastic heat-conducting fluids, Johnson-Segalman, weak solution}

\begin{document}
	
	\begin{abstract}
		We prove that there exists a~large-data and global-in-time weak solution to a~system of partial differential equations describing an unsteady flow of an incompressible heat-conducting rate-type viscoelastic stress-diffusive fluid filling up a~mechanically and thermally isolated container of any dimension. To overcome the~principle difficulties connected with ill-posedness of the~diffusive Oldroyd-B model in three dimensions, we assume that the~fluid admits a~strengthened dissipation mechanism, at least for excessive elastic deformations. All the~relevant material coefficients are allowed to depend continuously on the~temperature, whose evolution is captured by a~thermodynamically consistent equation. In fact, the~studied model is derived from scratch using only the~balance equations for linear momentum and energy, the~formulation of the~second law of thermodynamics and the~constitutive equation for the~internal energy. The~latter is assumed to be a~linear function of temperature, which simplifies the~model. The~concept of our weak solution incorporates both the~temperature and entropy inequalities, and also the~local balance of total energy provided that the~pressure function exists.
	\end{abstract}
	
	\maketitle
	
	\section{Introduction}
	
	Material properties of both synthetic and organic viscoelastic materials are very sensitive to temperature changes. Reliable predictions of corresponding processes by computational tools requires
	to incorporate complex thermal/mechanical effects into the~description of the~model. The~understanding how thermal and mechanical processes are coupled and what is the~structure of the~complete temperature equation has been considered as an open issue till recently (see \cite{Tanner2009, Hron2017}). A~methodology that can be used to develop such a~complete model (system of partial differential equations -- PDEs) and that is followed in this study has its origin in \cite{RS_2000,RS_2004}. A~complete (i.e. including elastic contribution to the~internal energy) thermodynamicly consistent models for viscoelastic rate type fluids is developed in \cite{Hron2017} where also further references to earlier studies including in particular \cite{Leonov1976, Wapperom1998, Dressler1999, Ireka2013} are given. Incorporation of additional stress diffusive phenomena into this thermodynamic framework is then developed in \cite{Malek2018}. 
	
	The aim of this study is to establish mathematical foundation for a~robust class of heat-conducting viscoelastic rate-type fluids with stress diffusion. In particular, we identify reasonable conditions on material functions/coefficients that are sufficient to prove long-time and large-data existence of weak solution. To develop analysis for complete thermal/mechanical systems of PDEs is considerably harder than to studying merely mechanical systems. To our best knowledge, there is only one existing analytical work dealing with such a~problem, see \cite{Bulicek2018}, where however the~elastic response is drastically reduced to a~spherical stress governed by a~scalar quantity. In our work, we do not make such an assumption and we work with the~full $d$-dimensional elastic tensor. On the~other hand, we assume that there is a~linear relation between the~internal energy and temperature. The~main purpose for this assumption is to simplify the~(already very technical) presentation of the~existence analysis. Additionally, the~linear relationship between the~internal energy and temperature is used in applications involving viscoelastic fluids, such as polymer melts, see \cite{Rao2002}. Taking this aside, our model contains no further simplification. A~complete physical derivation of the~model studied in this paper and a~more detailed description of the~participating physical quantities are given in Section~\ref{Sphys}. This opening section continues below with an informal formulation of the~main result, a~brief description of the~PDE system and a~basic overview of the~relevant literature. In~Section~\ref{Smat}, we introduce necessary notation, derive informally a~priori estimates that naturally leads to the~definition of function spaces in which the~existence theory is established. This section also contains precise definition of the~solution to studied problem and the~formulation of the~main result. Its proof represents the~content of remaining sections of the~paper, see Sections~\ref{SS1}--\ref{SS7}. Their more detailed description is given at the~end of Section~\ref{Smat}.
	
		
		\subsection*{Formulation of the~problem}
		
		We consider an incompressible fluid with the~constant density set to be one, for simplicity. The~fluid is flowing inside an open bounded connected set $\Omega\subset\R^d$ with a~Lipschitz boundary $\partial\Omega$. For an arbitrary (but fixed) time interval $(0,T)$, $T>0$, we set $\Q\coloneqq(0,T)\times\Omega$ and $\SigT\coloneqq(0,T)\times\partial\Omega$. Our main objective, in this study, is to develop a~long-time and large-data existence theory for the~following initial- and boundary-value problem.  
		
		For given 
		\begin{itemize}
			\item right-hand side $\fe:\Q\to\R^d$, 
			\item initial data $\ve_0:\Omega\to\R^d$, $\B_0:\Omega\to\Sym$ being positive definite and $\theta_0:\Omega\to(0,\infty)$, 
			\item constants $a\in\R$, $\alpha\geq0$, $\mu>0$ and $c_v>0$, 
			\item continuous functions $\nu,\la,\kappa:(0,\infty)\to(0,\infty)$ and $\PP:(0,\infty)\times\Sym\to\Sym$, 
		\end{itemize} 
		we look for functions $\ve:\Q\to\R^d$, $\pr,\theta,e,E,\eta, \xi:\Q\to\R$ and $\B,\Sb:\Q\to\Sym$ fulfilling the~(physical) restrictions
		{\allowdisplaybreaks\begin{align}
				&\theta>0,\hspace{9cm}\label{s1}\\
				&\B\x\cdot\x>0\quad\text{for all }\x\in\R^d\setminus\{0\},\label{s2} \\
				&\Sb=2\nu(\theta)\Dv+2a\mu\theta\B,\label{s22}\\
				&e=c_v\theta,\label{s3} \\
				&E=\tfrac12|\ve|^2+e,\label{s30}\\
				&\eta=c_v\ln \theta-f(\B),\quad\text{where}\quad f(\B)\coloneqq \mu(\tr\B-d-\ln\det\B),\label{s4} \\
				&\xi=\frac{2\nu(\theta)}{\theta}|\Dv|^2+\kappa(\theta)|\nn\ln\theta|^2+  \PP(\theta,\B)\cdot f'(\B)+\la(\theta)\nn\B\cdot\nn f'(\B),\label{s5}
		\end{align}}
		and solving (in a~suitable sense) the~following system of PDEs in $\Q$
		{\allowdisplaybreaks\begin{align}
				&\di\ve=0,\label{s6d}\\
				&\partial_t\ve+\ve\cdot\nn\ve+\nn\pr-\di\Sb=\fe,\label{s7}\\
				&\partial_t\B+\ve\cdot\nn\B+\PP(\theta,\B)-\di(\la(\theta)\nn\B)=\Wv\B-\B\Wv+a(\Dv\B+\B\Dv),\label{s8}\\
				&\partial_t e+\ve\cdot\nn e-\di(\kappa(\theta)\nn\theta)=\Sb\cdot\Dv,\label{s9}\\
				&\partial_tE+\ve\cdot\nn E-\di(\kappa(\theta)\nn\theta)=\di(-\pr\ve+\Sb\ve)+\fe\cdot\ve,\label{s10}\\
				&\partial_t\eta+\ve\cdot\nn\eta-\di(\kappa(\theta)\nn\ln\theta)+\di(\la(\theta)\nn f(\B))=\xi\label{s11eta}
		\end{align}}
		completed by the~boundary conditions on $\SigT$
		\begin{align}
			\ve\cdot\n&=0,\qquad(\Sb\n+\alpha\ve)_{\tau}=0,\label{s12bc}\\
			\n\cdot\nn\B&=0,\label{s13bc}\\
			\n\cdot\nn\theta&=0,\label{s14bc}
		\end{align}
		and by the~initial conditions fulfilled in $\Omega$
		\begin{equation}
			\ve(0, \cdot)=\ve_0,\qquad\B(0, \cdot)=\B_0,\qquad\theta(0, \cdot)=\theta_0.\label{s15ic}
		\end{equation}
		
		The~physical meaning of the~above unknowns is the~following: $\ve$ is the~flow velocity, $\pr$ is the~pressure, $\B$ is the~extra stress tensor (arising due to the~elastic properties of the~fluid), $\theta$ is the~temperature, $e$ is the~internal energy, $E$ is the~total energy and $\eta$ is the~entropy. We shall now state informally our main result.
		
		\begin{theorem*}
			If the~material coefficients $\kappa(\theta)$ and $P(\cdot,\B)$ grow sufficiently fast as $\theta\to\infty$ and $|\B|\to\infty$, respectively {\rm(}with the~other coefficents being merely bounded and positive{\rm)}, then there exists a~generalized global-in-time solution of the~system \eqref{s1}--\eqref{s15ic} for any initial data with finite total energy and entropy.
		\end{theorem*}

		In order to explain the~equations above, let us first clarify some notation, see also the~beginning of Section~\ref{Smat}. The~symbol $\ve\cdot\nn\ve$ denotes a~vector with the~$i$-component 
		$(\ve\cdot\nn\ve)_i = \sum_{k=1}^d v_k \partial_{x_k}v_i$. Similarly, $\ve\cdot\nn\B$ is a~tensor with the~$ij$-component 
		$(\ve\cdot\nn\B)_{ij} = \sum_{k=1}^d v_k \partial_{x_k}(\B)_{ij}$.  The~first two terms of each equation~\eqref{s7}--\eqref{s11eta} represent the~material (or convective) derivative of the~respective unknown and we shall sometimes use the~abbreviation
		\begin{equation*}
			\dot{u}\coloneqq\partial_t u+\ve\cdot\nn u.
		\end{equation*}
		Further, the~symbol $\n$  denotes the~outward unit normal vector at a~given point of $\partial \Omega$ and $\z_{\tau}$ stands for the~tangential part (with respect to $\partial \Omega$) of any vector $ \z\in\R^d\cap\partial \Omega$, i.e. $\z_{\tau}\coloneqq\z-(\z\cdot\n)\n$. Furthermore, for any vector $\uu:\Omega\to\R^d$, the~symbols $\D\uu$ and $\mathbb{W}\uu$ denote the~symmetric and antisymmetric parts of a~gradient $\nn\uu=(\partial_j\uu_i)_{i,j=1}^{d}$ so that $\nn\uu=\D\uu+\mathbb{W}\uu$ with $(\D\uu)^T=\D\uu$ and $(\mathbb{W}\uu)^T=-\mathbb{W}\uu$.
		
		The first two equations \eqref{s6d} and \eqref{s7} resemble the~incompressible Navier--Stokes system for the~unknowns velocity field $\ve$ and the~pressure (constitutively undetermined part of the~Cauchy stress) $\pr$, however, with an additional term $2a\mu\di(\theta\B)$ coming from $\Sb$ and bringing to the~problem two other quantities: the~temperature $\theta$ and the~tensor $\B$ representing the~elastic response of the~fluid.  The~presence of this additional term prohibits one to use the~usual methods known in the~analysis of the~Navier-Stokes-Fourier-like systems, as there is no longer an useful form of the~balance of kinetic energy (the inner product $\B\cdot\Dv$ does not have a~sign). Instead, the~estimates on $\nn\ve$ are deduced only after taking into account the~whole thermodynamical evolution of the~system.
		
		Since the~dependence of the~material parameters (namely the~viscosity of the~fluid) on the~pressure $\pr$ is neglected, we simplify the~analysis by eliminating the~pressure from the~system completely, taking the~Leray projection of \eqref{s7} and searching for $\ve$ in divergence-free function spaces. If needed (for example if we want to preserve the~equation~\eqref{s10}), the~pressure can be reconstructed at the~last step. Then, it is known that Navier's slip boundary condition \eqref{s12bc}, or even more generally stick-slip boundary condition, allows one to prove that $\pr$ is an integrable function (if the~boundary of $\Omega$ is smooth enough so that $W^{2,r}$-regularity for the~classical Neumann problem holds), see \cite{BulicekIndiana, Bulicek2019, stick1, stick2, Blechta2019} for details. 
		Recall that the~integrability of the~pressure is not known to be true in general for no-slip boundary condition. The~integrability of $\pr$ is not only important in itself, but is also useful for the~validity of the~weak formulation of \eqref{s10}.
		
		To understand equation~\eqref{s8}, it is better to define first the~objective derivative of $\B$ as
		\begin{equation}\label{objder2}
			\der{\B}\coloneqq\dot{\B}-(\Wv\B-\B\Wv)-a(\Dv\B+\B\Dv),\quad a\in\R.
		\end{equation}
		This turns \eqref{s8} into
		\begin{equation}\label{JS2}
			\der{\B}+\PP(\theta,\B)-\di(\la(\theta)\nn\B)=0,
		\end{equation}
		which is a~mathematical formulation of a~generalized (due to an implicit form of $P$) Johnson--Segalman (\cite{Johnson1977}) viscoelastic model with stress diffusion (cf.\ \cite{Olmsted_2000} and references therein) and temperature dependent material parameters. The~reason why $\der{\B}$ appears in \eqref{JS2} is that,  unlike the~material derivative, the~objective derivative $\der{\B}$ (for any $a$) transforms correctly (as a~tensor) under a~time dependent rotation of the~observer. When $a\in[-1,1]$, then $\der{\B}$ is precisely the~Gordon-Schowalter derivative (\cite{Gordon1975}). It is known (see e.g.\ \cite{Pivokonsky_2015}) that by modifying the~value of $a$, it is possible to capture a~shear-thinning behaviour of the~fluid. The~case $a=0$ leads to the~class of models with the~corrotational objective derivative (cf.\ \cite{Zaremba1903}), which has very special properties that simplify the~analysis. The~case $a=1$ in \eqref{objder2} coincides with the~upper-convected objective derivative, which is probably the~most popular choice in the~literature. One of the~main features of our analysis is that, we are able to treat \eqref{s8} with any $a\in[-1,1]$ (or even $a\in\R$). As we shall see later, if $a\neq0$, the~summability of the~nonlinear terms like $\B\Dv$ in \eqref{s8} (and especially the~related term in \eqref{s9}) becomes the~main difficulty. This is essentially the~reason, why we formulate \eqref{s8} with a~general function $\PP(\theta,\B)$. The~strategy is that if $\PP(\theta,\B)$ grows sufficiently fast as $|\B|\to\infty$, then $\B$ admits enough integrability to define a~meaningful concept of solution of the~system \eqref{s1}--\eqref{s15ic}. Moreover, as the~form of $\PP$ can be attributed to the~dissipation mechanism of the~fluid, restricting its asymptotic growth should not be seen as a~significant physical drawback of our model. Recall that, for the~ classical Oldroyd-B and Giesekus models, the~function $\PP$ takes the~form  
		\begin{equation*}
			\PP(\theta,\B) = \delta(\theta) (\B - \I) \quad \textrm{ and } \quad 
			\PP(\theta,\B) =  \delta(\theta) (\B^2 - \B),
		\end{equation*}
		respectively. While these models are not covered by the~analysis presented below, the~existence result, in three dimensions, holds for 
		\begin{equation*}
			\PP(\theta,\B) = \delta(\theta) (\B^{\alpha} - \B^{\alpha-1}),\quad\alpha>2,
		\end{equation*}
		or
		\begin{equation} 
			\PP(\theta,\B) = \delta(\theta) \max\{1, \frac{|\B-\I|^2}{K^2}\} (\B - \I),\quad K>0. \label{pepa1}
		\end{equation}
		Note that the~last model coincides with the~Oldroyd-B model as long as $|\B-\I|\le K$.
		
		Due to \eqref{s3}, the~balance of internal energy \eqref{s9} is also the~temperature equation. As we hinted above, the~term $2a\mu\theta\B\cdot\Dv$ on the~right-hand side of \eqref{s9} is the~most difficult term to control in the~whole system \eqref{s6d}--\eqref{s11eta} and it is also the~term which is occasionaly omitted in some ``naive'' approaches to thermoviscoelasticity, as pointed out in \cite[Section~3]{Hron2017}.  Note also that this term does not have a~clear sign and thus, one cannot conclude the~positivity of temperature directly from \eqref{s9}  as in the~Navier-Stokes-Fourier case.  The~equations \eqref{s10} and \eqref{s11eta} govern the~evolution of two other unknowns $E$ and $\eta$, respectively. Since these quantities together with $\theta$ are mutually connected by simple algebraic relations \eqref{s3}, \eqref{s30} and \eqref{s4}, the~equations \eqref{s9}--\eqref{s11eta} are interchangeable and each of them alone can be used as the~equation for temperature evolution. To see this, note that \eqref{s30} and \eqref{s4} imply
		\begin{align}
			\partial_tE&=\ve\cdot\partial_t\ve+\partial_te\label{op1}\\
			\partial_t\eta&=c_v\theta^{-1}\partial_t\theta- f'(\B) \cdot\partial_t\B.\label{op2}
		\end{align}
		Within the~considered system of equations (assuming that all involved operations are meaningful), one can verify that the~equations \eqref{s9}, \eqref{s10}, \eqref{s11eta} are mutually equivalent. We remark that this equivalence may no longer be in place when, on the~level of generalized solutions, the~integrability of the~solution is not sufficient to define the~critical nonlinear terms in \eqref{s9} and \eqref{s10}, that is $\theta\B\cdot\Dv$ and $\theta\B\ve$, respectively. For example, this would be the~case where the~initial datum $\B_0$ has low integrability, as then the~available apriori estimates deteriorate (cf.~\eqref{Best0} below). In such cases, one may be forced to discard \eqref{s9}, or even \eqref{s10} from the~notion of generalized solution and leave only \eqref{s11eta}, which is least restrictive but still sufficient (together with the~global version of \eqref{s10}) to keep track of the~thermal evolution of the~system. Generalized solutions relying on the~weak formulation of balance of entropy were applied, e.g., in \cite{Feireisl06}, \cite{BuFei2009}, \cite{Feireisl2012}, \cite{Feireisl2016} or in \cite{Fremond_2012} for different fluid models. See also \cite{Havrda} for similar ideas in context of certain mixtures. For brevity, in this work, we shall avoid the~low integrability case and work only in the~setting, where both \eqref{s9} and \eqref{s11eta} (and \eqref{s10} if the~pressure can be defined) hold simultaneously, but only as the inequalities. Although these become automatically equalities if the~solution is smooth enough (see \eqref{weas2} below), in general this is unknown.

		\subsection*{State of the~art}
		
		Regarding the~existence analysis of a~viscoelastic fluid model including the~full temperature evolution, there is a~recent study \cite{BulPreprint2020}, where the~authors develop a~long-time and large-data existence theory for a~rate-type incompressible viscoelastic fluid model with stress diffusion under the~simplifying assumption that $\B=b\I$. This assumption leads to annihilation of irregular terms coming from the~objective derivative and it also simplifies the~momentum equation, where the~coupling to the~rest of the~system is realized only via temperature and elastic stress dependent viscosity. Other than that, to the~authors' best knowledge, there is no existence theory in a~setting that would be of similar generality as considered here. Thus, for the~first time, we provide an existence analysis for a~viscoelastic fluid model with a~full thermal evolution and taking into account all components of the~extra stress tensor. Moreover, the~equation for the~temperature we consider is derived from fundamental thermodynamical laws (similarly as in~\cite{BulPreprint2020}, \cite{Hron2017}, \cite{Malek2018}) and consequently, the~heating originates from both the~viscous and elastic effects. Also, we would like to point out that the~all material coefficients of the~model depend on the~temperature. Although we place some restrictions on the~growth of these coefficients, these are only asymptotic and therefore unimportant from the~point of view of physical applications. Furthermore, the~model considered here has the~property that the~evolution of the~temperature cannot be decoupled from the~rest of the~model even in the~case of constant material coefficients.
		
		Even if we confine to a~much simpler class of isothermal processes, the~existence theory there is far from being complete. Although there are several relevant global-in-time existence results for large data, in most cases, they are restricted in an essential way. For example, in~\cite{Lions2000} the~authors provide an existence theory for a~model with the~corrotational Jaumann--Zaremba derivative (the case $a=0$). This case is much easier than for the~other choices of $a$ since the~corrotational part drops out upon multiplication by any matrix that commutes with $\B$. Moreover, it seems that the~physically preferred case is $a=1$, which corresponds to the~upper convected (Oldroyd) derivative (see \cite{MaPr2018}, \cite{MaRaTu15}, \cite{MaRaTu18}, \cite{RS_2000} or \cite{RS_2004}). Then, in~\cite{Masmoudi2011}, a~proof of existence of a~weak solution to FENE-P, Giesekus and PTT viscoelastic models is outlined. In fact, it is shown there that certain defect measures of the~non-linear terms are compact. A~complete proof in the~case of two-dimensional flows of a~Giesekus fluid is given in \cite{los1}. In the~case of spherical elastic response when $\B=b\I$, we refer to~\cite{Bulicek2018} (and \cite{Bulicek2019}, \cite{Lu_2020} in the~compressible case) for an analysis of such models. In the~two-dimensional case, existence and regularity results can be found in~\cite{Constantin2012}. An~existence theory for related viscoelastic models (Peterlin class) was developed, e.g., in~\cite{Lukacova-Medvidova2017}. However, for these models, the~energy storage mechanism depends only on the~spherical part of the~extra stress, which is a~major simplification compared to our case. A~notable exception is the~thesis \cite{Kreml}, where the~author obtains a~global weak solution to an Oldroyd-like diffusive model under certain growth assumptions on the~material coefficients. However, the~overall thermodynamical compatibility of the~studied model is unclear. Furthermore, there are existence results for viscoelastic models involving various approximations that improve properties of the~system, see e.g.\ \cite{BARRETT2011} or \cite{Pokorny}. 
		
		The~article \cite{Bathory_2020} develops the~existence theory for viscoelastic diffusive Oldroyd-B or Giesekus models. This result relies on a~certain physical correction of the~energy storage mechanism away from the~stress-free state resulting at $L^2$ a~priori estimates for $\nabla \B$. Interestingly, for such models, in two dimensions, uniqueness and full regularity of weak solution is available (at least in the~spatially periodic case), see \cite{casey22}. Various modifications of the~classical Oldroyd-B model are also discussed in \cite{Chupin2018}. The~article contains also existence results that are of local nature or for small (initial) data. Local-in-time existence of regular solutions to a~viscoelastic Oldroyd-B model without diffusion was shown in~\cite{Saut1990}. It is also proved there that for small data there exists a~global in time solution. For the~steady case of a~generalized Oldroyd-B model with small and regular data, see e.g.\ \cite{bio4}.

		\section{Thermodynamical compatibility of the~model}\label{Sphys}
		
		In this section, we show the~physical consistency of the~system \eqref{s1}--\eqref{s15ic} as it follows naturally from the~elementary balance equations for mass, momentum and energy and some reasonable constitutive assumptions. The~latter can be efficiently encoded in just two scalar quantities describing how the~energy is stored and dissipated in the~material, see \cite{RS_2000} and \cite{RS_2004} for the~origins of this method. Physical justification of viscoelastic fluid models similar to ours is carried out in many works, see~\cite{Dostalik2019}, \cite{Hron2017}, \cite{Malek2018}, \cite{MaRaTu15} or \cite{MaRaTu18}.
		
		For the~rest of this section, we make an implicit assumption that all functions depend smoothly on time and space position (if not specified otherwise), with the~arguments $(t,x)$ suppressed as usual.
		
		Since the~density of the~fluid is assumed constant ($\rh=1$), the~balance of mass
		\begin{equation*}
			\dot{\rh}+\rh\di\ve=0
		\end{equation*}
		is reduced to \eqref{s6d}. Next, the~general form of the~balance equations of momentum, total energy and specific entropy is
		\begin{align}
			\dot{\ve}&=\di\T,\label{ba1}\\
			\dot{E}+\di\je_e&=\di(\T\ve),\label{ba2}\\
			\dot{\eta}+\di\je_{\eta}&=\xi,\label{ba3}
		\end{align}
		where $\T$ is the~Cauchy stress tensor and $\je_e$ and $\je_{\eta}$ are energy and entropy fluxes, respectively. Tensor $\T$ is symmetric due to the~conservation of angular momenta. Furthermore, the~balance equation for the~internal energy $e\coloneqq E-\f12|\ve|^2$ is
		\begin{equation}\label{ba4}
			\dot{e}+\di\je_e=\T\cdot\Dv,
		\end{equation}
		as follows easily from \eqref{ba1} and \eqref{ba2}.
		
		Turning to thermodynamics, we assert the~following fundamental relation (cf.~\cite[(1.8)]{Callen}) between specific entropy, internal energy and positive definite tensor~$\B$:
		\begin{equation}
			\eta=S(e,\B),\quad\text{where}\quad\partial_eS>0.\label{td}
		\end{equation}
		In this case, the~temperature $\theta$ is defined as usual by
		\begin{equation}\label{thdef}
			\f1{\theta}\coloneqq\partial_eS(e,\B).
		\end{equation}
		Taking the~material time derivative of both sides of \eqref{td} then leads to
		\begin{equation*}
			\dot{\eta}=\f1{\theta}\dot{e}+\partial_{\B}S(e,\B)\cdot\dot{\B}.
		\end{equation*}
		This in turn allows us to express the~rate of entropy production in the~general form via the~balance equations \eqref{ba3} and \eqref{ba4} as follows:
		\begin{equation}\label{xigen}
			\xi=\f1{\theta}(\T\cdot\Dv-\di\je_e)+\di\je_{\eta}+\partial_{\B}S(e,\B)\cdot\dot{\B}
		\end{equation}
		In the~next step, we make special choices of $\T$, $\je_e$, $\je_{\eta}$ and $S$ that lead to \eqref{s7}, \eqref{s9}--\eqref{s11eta} and verify, using the~above formula and also \eqref{s8}, that $\xi\geq0$.
		
		The formula for specific entropy is chosen as
		\begin{equation}\label{Schoice}
			S(e,\B)\coloneqq c_v\ln e-f(\B),
		\end{equation}
		where $c_v>0$ is the~specific heat constant and 
		\begin{equation}\label{helm}
			f(\B)\coloneqq \mu(\tr\B-d-\ln\det\B),\quad\mu>0,
		\end{equation}
		is a~function that characterizes the~elastic properties of the~fluid. If $\mu=0$ or $\B=\I$, then  \eqref{Schoice} reduces to the~classical Navier-Stokes-Fourier model, where one has
		\begin{equation}\label{ethrel}
			e=c_v\theta.
		\end{equation}
		Note that as long as $\mu$ does not depend on temperature (which is the~case in this work), this property actually remains valid even with our generalized assumption \eqref{Schoice}, as is immediately obvious from \eqref{Schoice}, \eqref{thdef} and \eqref{td}.
		
		Next, comparing \eqref{ba1}, \eqref{ba4} and \eqref{ba3} with \eqref{s7}, \eqref{s9} and \eqref{s11eta}, respectively, the~constitutive choices for the~fluxes are evidently as follows:
		\begin{align}
			\T&\coloneqq-\pr\I+2\nu(\theta)\Dv+2a\mu\theta\B,\label{flux1}\\
			\je_e&\coloneqq -\kappa(\theta)\nn\theta,\label{flux2}\\
			\je_{\eta}&\coloneqq-\kappa(\theta)\nn\ln\theta+\la(\theta)\nn f(\B),\label{flux3}
		\end{align}
		where $\nu(\theta)>0$, $\kappa(\theta)>0$ and $\la(\theta)>0$ are the~kinematic viscosity, thermal conductivity and stress diffusion coefficients, respectively, and parameter $a$ arises from the~definition of the~objective tensorial time derivative \eqref{objder2}.
		
		Finally, plugging the~relations \eqref{flux1}--\eqref{flux3} and \eqref{JS2} into \eqref{xigen} and taking advantage of the~identities
		\begin{align}
			&-\pr\I\cdot\Dv=-\pr\di\ve=0,\nonumber\\
			&\partial_{\B}S(e,\B)=-f'(\B)=-\mu(\I-\B^{-1})\qquad\text{(see \eqref{IJ1} below)},\label{calc1}\\
			&(\I-\B^{-1})\cdot(\Wv\B-\B\Wv)=(\I-\B^{-1})\B\cdot\Wv-\B(\I-\B^{-1})\cdot\Wv=0,\nonumber\\
			&(\I-\B^{-1})\cdot(\Dv\B+\B\Dv)=2(\B-\I)\cdot\Dv=2\B\cdot\Dv,\nonumber\\
			&\nn(\I-\B^{-1})\cdot\nn\B=\B^{-1}\nn\B\B^{-1}\cdot\nn\B=|\B^{-\f12}\nn\B\B^{-\f12}|^2\label{calc2}
		\end{align}	
		(here we used that $\B$ is a~symmetric positive definite matrix, which follows from the~same property of $\B_0$ as we shall see later) leads to
		\begin{align*}
			\xi&=\f1{\theta}(2\nu(\theta)|\Dv|^2+2a\mu\theta\B\cdot\Dv+\di(\kappa(\theta)\nn\theta))+\di(-\kappa(\theta)\nn\ln\theta+\la(\theta)\nn f(\B))\\
			&\qquad-\mu(\I-\B^{-1})\cdot(\Wv\B-\B\Wv+a(\Dv\B+\B\Dv)-P(\theta,\B)+\di(\la(\theta)\nn\B))\\
			&=\f{2\nu(\theta)}{\theta}|\Dv|^2+\kappa(\theta)|\nn\ln\theta|^2+\mu(\I-\B^{-1})\cdot P(\theta,\B)+\mu\la(\theta)|\B^{-\f12}\nn\B\B^{-\f12}|^2,
		\end{align*}
		which validates \eqref{s5} and verifies the~physical consistency of the~model. Moreover, from the~last expression, it is evident that $\xi\geq0$ whenever $(\I-\B^{-1})\cdot P(\theta,\B)\geq0$, in which case the~second law of thermodynamics is always fulfilled.

		\section{Weak formulation \& main result}\label{Smat}
		
		In this section, we focus on mathematical properties of the~problem \eqref{s1}--\eqref{s15ic}. After we introduce the~necessary notation, we formally derive apriori estimates that clarify the~imposed restrictions on model parameters. They also indicate the~functions spaces in which the~long-time and large-data existence theory can be established. Then we provide the~definition of weak solution to \eqref{s1}--\eqref{s15ic} and formulate the~main result of the~paper.
		
		\subsection*{Notation and function spaces}		
		
		The sets of symmetric, positive definite and positive semi-definite matrices are defined as follows:
		\begin{align*}
			\Sym&\coloneqq\{\A\in\R^{d\times d}:\A=\A^T\},\\
			\PD&\coloneqq\{\A\in\Sym:\A\x\cdot\x>0\text{ for all }0\neq\x\in\R^d\},\\
			\R^{d\times d}_{\geq0}&\coloneqq\{\A\in\Sym:\A\x\cdot\x\geq0\text{ for all }\x\in\R^d\}.
		\end{align*}
		If $d=1$, we set $\R_{>0}:=\R^{1\times 1}_{>0}=(0,\infty)$ and $\R_{\geq0}\coloneqq\R^{1\times 1}_{\geq0}=[0,\infty)$.  We use the~symbol ``$\cdot$" to denote the~standard inner product in any multi-dimensional space, while the~symbol ``$\otimes$" denotes the~outer product. Further, the~symbol ``$|\cdot|$" can be applied to either scalars, vectors or matrices, meaning always the~Euclidean (or Frobenius) norm. The~functions of matrices, such as matrix real powers, matrix logarithm and matrix exponential, are understood in the~standard way, using the~spectral decomposition for symmetric matrices, for instance. For various products of matrix-valued functions, we use an intuitive index-free notation. One can follow the~rule that $\nn$ can only be contracted with another vector (or one-form), but never with columns or rows of some matrix, so for example:  $\nn\A\cdot\nn\B=\sum_{i,j,k}\partial_i\A_{jk}\partial_i\B_{jk}$ or $(\ve\otimes\A)\cdot\nn\B=\sum_{i,j,k}\ve_i\A_{jk}\partial_i\B_{jk}$ or $|\A\nn\B\mathbb{C}|^2=\sum_{i,j,k}\big(\sum_{l,m}\A_{il}\partial_k\B_{lm}\mathbb{C}_{mj}\big)^2$.

		If not stated otherwise, the~set $\Omega\subset\R^d$ is an open bounded set with a~Lipschitz boundary (i.e.\ of the~class $\mathcal{C}^{0,1}$) in the~sense of \cite[Sect.\ 2.1.1]{Necas}. Let $\OO\subset\R^m$ be an open bounded set (such as $(0,T)$, $\Omega$ or $\Q$) and $V$ be a~subset of an Euclidean space. The~symbol $(L^p(\OO;V),\norm{\cdot}_{L^p(\OO;V)})$ denotes the~Lebesgue space of functions $u:\OO\to V$.
		The~standard inner products in $L^2(\OO;V)$ and also in $L^2(\partial\Omega;V)$ are denoted as $(\cdot,\cdot)_{\OO}$ and $(\cdot,\cdot)_{\partial\Omega}$, respectively. In the~special case that $\OO=\Omega$, we write just $\norm{\cdot}_p$ instead of $\norm{\cdot}_{L^p(\Omega;V)}$ and $(\cdot,\cdot)$ instead of $(\cdot,\cdot)_{\Omega}$. 
		
		The symbol $(W^{k,p}(\Omega;V),\norm{\cdot}_{k,p})$, $1\leq p\leq \infty$, $k\in\N$, is used to denote the~Sobolev spaces with their standard norm considered over the~set $\Omega$. If $p>1$, we set $W^{-k,p}(\Omega;V)\coloneqq(W^{k,p'}(\Omega;V))^*$, where $p'\coloneqq p/(p-1)$, $k\in\N$, and the~star symbol ``$^*$'' denotes the~topological (continuous) dual space. For vector-valued functions, we introduce the~following subspaces:
		\begin{align*}
			W^{k,p}_{\n}&\coloneqq\{\uu\in W^{k,p}(\Omega;\R^d):\uu\cdot\n=0\},\quad k\in\N,\quad p<\infty,\\
			W^{k,p}_{\n,\di}&\coloneqq\{\uu\in W^{k,p}_{\n}:\di\uu=0\},\quad k\in\N,\quad p<\infty,\\
			W^{-k,2}_{\n,\di}&\coloneqq(W^{k,2}_{\n,\di})^*,\quad k\in\N,\\
			L^2_{\n,\di}&\coloneqq\overline{W^{1,2}_{\n,\di}}^{\norm{\cdot}_2}.
		\end{align*}
		The expression $\uu\cdot\n$ is understood as a~trace of a~Sobolev function, for which we do not use any special notation. The~meaning of the~duality pairing $\langle \cdot, \cdot \rangle$ is always understandable in the~given~context. 
		
		Let $X$ be a~Banach space. The~Bochner spaces $L^p(0,T;X)$ with $1\leq p \leq\infty$ consist of strongly measurable mappings $u:[0,T]\to X$ for which the~norm
		\begin{equation*}
			\norm{u}_{L^p(0,T;X)}\coloneqq\Bigg\{\begin{aligned}\left(\int_0^T\norm{u}^p_X\right)^{\f1p}&\qquad\text{if } 1\leq p<\infty,\\
				\esssup_{(0,T)}\norm{u}_X&\qquad\text{if } p=\infty,\end{aligned}
		\end{equation*}
		is finite. If $X=L^q(\Omega;V)$ or $X=W^{k,q}(\Omega;V)$, with $1\leq q\leq \infty$, $n\in\N$, we use the~abbreviations $\norm{\cdot}_{L^pL^q}$ or $\norm{\cdot}_{L^pW^{k,q}}$, respectively, for the~corresponding norms. Next, the~space of weakly continuous functions is defined as
		\begin{align*}
			\mathcal{C}_w([0,T];X)&\coloneqq\big\{u\in L^{\infty}(0,T;X):\text{the function }\langle g,u\rangle\text{ is continuous in }[0,T]\\
			&\hskip7.5cm\text{ for every }g\in X^*\big\},
		\end{align*}
		whereas the~standard space of continuous $X$-valued functions on $[0,T]$ is denoted by $\CO([0,T];X)$ and equipped with the~norm
		\begin{equation*}
			\norm{u}_{\CO([0,T];X)}\coloneqq\sup_{t\in[0,T]}\norm{u(t)}_X.
		\end{equation*}
		In addition, if $X$ is separable and reflexive, we define two more spaces. First, the~space of $X^*$-valued Radon measures on $[0,T]$ is defined as
		\begin{equation*}
			\mathcal{M}([0,T];X^*)\coloneqq (\mathcal{C}([0,T];X))^*.
		\end{equation*} 
		Then, we set
		\begin{align*}
			BV([0,T];X^*)&\coloneqq\big\{u\in L^{\infty}(0,T;X^*), \; \partial_t u \in \mathcal{M}([0,T]; X^*)\big\}
		\end{align*}
		to be the~space of functions having $X^*$-valued bounded variation with respect to the~time variable. Note that if $u\in BV([0,T];X^*)$ then it makes sense to define value from left and from right at any point $t$, i.e., there exist
		\begin{equation*}
			u(t_+)\coloneqq \lim_{\tau\to t_+}u(\tau) \, \textrm{ for any } t\in[0,T) \quad \textrm{ and } \quad 
			u(t_{-})\coloneqq\lim_{\tau \to t_{-}}u(\tau) \, \textrm {for any } t\in(0,T],
		\end{equation*}
		where the~limits are considered in the~strong topology of $X^*$. For properties of $BV$ mappings in Bochner spaces, we refer e.g. to~\cite{HePaRe19}.
		
		\subsection*{Assumptions on material coefficients}
		
		The mathematical properties of the~system \eqref{s1}--\eqref{s15ic} depend crucially on the~behaviour of the~material coefficients, which we now specify. We will require that
		\begin{equation}\label{Acont}
			\nu,\kappa,\la,P\quad\text{are continous functions in }\R,\R,\R\text{ and }\R\times\Sym\text{, respectively,}
		\end{equation}
		and there are numbers $q,r>0$, $C,C_{\alpha}>0$ and $\omega_{P}>0$, such that, for all $s\in\R$, the~following conditions hold:
		{\allowdisplaybreaks\begin{align}
				C^{-1}&\leq\nu(s)\leq C,&&\label{Anu}\\
				C^{-1}(1+s^r)&\leq \kappa(s)\leq C(1+s^r),&&\label{Akap}\\
				C^{-1}&\leq \la(s)\leq C,&&\label{Ala}\\
				\PP(s,\A)&=\PP(s,\A)^T&&\text{for all }\A\in\Sym,\label{APkom}\\
				|\PP(s,\A)|&\leq C(1+|\A|^{q+1})&&\text{for all }\A\in\Sym,\label{APbou}\\
				\PP(s,\A)\cdot\A^{\alpha}&\geq C_{\alpha}|\A|^{q+1+\alpha}-C&&\text{for all }\alpha>0\text{ and }\A\in\PD,\label{APcoer}\\
				\PP(s,\A)\cdot\I&\geq-C&&\text{for all }\A\in\PD,\label{APcoer2}\\
				\PP(s,\A)\cdot(\I-\A^{-1})&\geq 0&&\text{for all }\A\in\PD,\label{APp}\\
				\PP(s,\A+\omega_P\I)\x\cdot\x&\leq0&&\text{for all }\A\in\Sym\text{ and }\x\in\R^d\nonumber\\
				&&&\qquad\text{ such that }\A\x\cdot\x\leq0.	\label{APpd}
		\end{align}}
		Assumption \eqref{Anu} is quite standard for fluids. Restriction \eqref{Akap} means that $\kappa$ is a~bounded function near zero and has an~$r$-growth near infinity. Assumption \eqref{Ala} is chosen just for simplicity. Condition \eqref{APkom} is necessary for validity of \eqref{s8}. Assumptions \eqref{APbou} and \eqref{APcoer} mean that $\PP(\cdot,\A)$ behaves asymptotically as $\A^{q+1}$, which is a~crucial information to get sufficient a~priori estimates. Condition \eqref{APcoer2} simplifies the~analysis at one step and means basically that the~leading order term of $\PP(\cdot,\A)$ appears with the~positive sign, compare e.g.\ with the~Oldroyd-B and Giesekus model, where $\PP(\cdot,\A)=\A-\I$ and $\PP(\cdot,\A)=\A^2-\A$, respectively. Property \eqref{APp} is important for the~validity of the~second law of thermodynamics in our model. Again, both Oldroyd-B and Giesekus models fulfill this requirement. Finally, the~assumption \eqref{APpd} restricts the~behaviour of $\PP(\cdot,\A)$ when $\A$ is not positive definite or if its eigenvalues are too small. We remark that this technical condition concerns the~case $s\leq 0$ or $\A\in\R^{d\times d}\setminus\PD$ that actually never arises in the~studied problem. An explicit example of function $\PP$ satisfying \eqref{APkom}--\eqref{APpd} would be
		\begin{equation*}
			\PP(s,\A)=\delta(s)(1+|\A-\I|^{q-\beta})\A^{\beta}(\A-\I),
		\end{equation*}
		where $\delta$ is a~continuous positive real function and $\beta\in[0,q]$. Indeed, note that, for any $\A\in\PD$, we can write
		\begin{align*}
			\A^{\beta}(\A-\I)&\cdot(\I-\A^{-1})=\A^{\f{\beta}2}\A^{\f{\beta}2}(\A^{\f12}-\A^{-\f12})\A^{\f12}\cdot(\I-\A^{-1})\\
			&=\A^{\f{\beta}2}(\A^{\f12}-\A^{-\f12})\cdot\A^{\f{\beta}2}(\I-\A^{-1})\A^{\f12}=|\A^{\f{\beta}2}(\A^{\f12}-\A^{-\f12})|^2\geq0,
		\end{align*}
		implying \eqref{APp}. The~properties \eqref{APbou}, \eqref{APcoer} and \eqref{APcoer2} follow easily from \eqref{I8} in Appendix. Finally, we claim that \eqref{APpd} holds with $\omega_P=1$. Indeed, let $0\neq\x\in\R^d$ be an eigenvector of $\A\in\Sym$, for which $\la\coloneqq\A\x\cdot\x/|\x|^2\leq0$. If $\A+\I\not\in\PD$ then we can redefine $\PP(\cdot,\A+\I)$ as needed. Otherwise, we have $\A+\I\in\PD$, and thus $\la>-1$ and we can write
		\begin{align*}
			\PP(s,\A+\I)\x\cdot\x&=\delta(s)(1+|\A|^{q-\beta})(\A+\I)^{\beta}\A\x\cdot\x\\
			&=\delta(s)(1+|\A|^{q-\beta})(\la+1)^{\beta}\la|\x|^2\leq0.
		\end{align*}
		
		
		\subsection*{Conditions on \texorpdfstring{$q$}{q} and \texorpdfstring{$r$}{r}.}
		
		To make sure that the~individual terms appearing in the~weak formulation of the~governing equations (defined below) are well defined, we need to restrict the~parameters $q$ and $r$ by the~conditions
		\begin{align}\label{A0}
			r&>1-\f2d \quad \textrm{ and } \quad q>1+\frac{2}{r-1+\frac{2}{d}};
		\end{align}
		we recall that $d\geq2$ is the~dimension of the~domain $\Omega$.

		Condition \eqref{A0} is sufficient to define every term of the~system \eqref{s1}--\eqref{s15ic} in a~weak sense, with the~exception of \eqref{s10}, which needs additional technical assumptions due to the~presence of pressure (see the~second part of Theorem~\ref{LG} below). As such, condition \eqref{A0} is actually sufficient for the~existence of a~weak solution, which is the~content of our main result.
		

		%
		
		By imposing \eqref{A0}, we place some restrictions on the~coefficients of the~model which may not agree with experimental measurements. Note, however, that \eqref{Akap}, \eqref{APbou} and \eqref{APcoer} restrict only the~asymptotic behaviour of the~coefficients. For example, any continuous function $\kappa$ defined on some interval $(\theta_0,\theta_1)$, $0<\theta_0<\theta_1<\infty$, can be modified in a~neighbourhood of $0$ and $\infty$ so that \eqref{Akap} holds. The~interval $(\theta_0,\theta_1)$ may represent the~temperature range for which the~model we are considering makes sense. When the~fluid starts to freeze or boil, then we are clearly outside this range and it makes no sense to prescribe the~coefficients $\nu$, $\kappa$, $\delta$ and $\la$ there. On the~other hand, it is unclear whether one can deduce some absolute bounds for the~temperature, besides $\theta>0$, using only the~information that is encoded in the~system. Thus, purely for mathematical reasons, we have to assume that these material coefficients are defined in some way also outside $(\theta_0,\theta_1)$. A~similar remark applies also for the~other coefficients. For example, if $|\A|$ is too large, any realistic material eventually breaks down. Thus, we may set $\PP(\cdot,\A)=\A-\I$, $|\A|\in[0,M)$, where $M$ is large (to mimic the~Oldroyd-B model, for example) and then extend this function continuously so that \eqref{APcoer} holds with some large $q$, see \eqref{pepa1}. 
		
		\subsection*{A priori estimates}

		Let us now the~motivate the~definition of the~weak solution to \eqref{s1}--\eqref{s11eta} by an informal derivation of the~available a~priori estimates. This clarifies the~need for \eqref{A0} and highlights the~main idea of the~existence proof. The~starting point are the~assumptions on the~data:
		\begin{equation}\label{data}
			E_0\in L^1(\Omega;\R_{\geq0}),\;\eta_0\in L^1(\Omega;\R),\;\B_0\in L^q(\Omega;\PD),\;\fe\in L^2(\Q;\R^d).
		\end{equation}
		In addition, we may suppose that
		\begin{equation}\label{psd}
			\theta\geq0\quad\text{and}\quad\B\x\cdot\x\geq0\quad\text{for all}\quad\x\in\R^d,
		\end{equation}
		which is due to a~suitable construction of the~solution (cf.~\eqref{neup} below).	
		
		In what follows, the~basic relations \eqref{s3}--\eqref{s5} and also \eqref{helm} will be used without further reference. Moreover, the~symbol $C$ will be used to denote a~positive constant that can change from line to line and can depend only on the~data, domain $\Omega$, time $T>0$ and other constants appearing in \eqref{Anu}--\eqref{APpd}.
		
		Integrating \eqref{s10} over $\Omega$ and applying the~boundary conditions \eqref{s12bc} and \eqref{s14bc} drops the~divergence terms, which, together with Young's inequality and $\eqref{psd}_1$, leads to
		\begin{equation*}
			\frac{\dd{}}{\dd{t}}\ii E=\ii\fe\cdot\ve\leq\f12\ii|\fe|^2+\ii E.
		\end{equation*}
		Hence, using $\eqref{data}_1$, we see that $E\in L^{\infty}(0,T;L^1(\Omega;\R))$, therefore also
		\begin{equation}
			\theta\in L^{\infty}(0,T;L^1(\Omega;\R_{\geq0}))\quad\text{and}\quad\ve\in L^{\infty}(0,T;L^2(\Omega;\R^d)).\label{vthe}
		\end{equation}
		
		Next, integrating the~entropy inequality~\eqref{s11eta}, and again applying the~boundary conditions in the~divergence terms, gives
		\begin{equation*}
			\f{\dd{}}{\dd{t}}\ii\eta(t)\geq0.
		\end{equation*}
		Applying $\eqref{data}_2$ and $\eqref{psd}_2$ ($\tr\B\geq0$, to be precise), the~last inequality yields
		\begin{equation*}
			\ii(\ln\theta(t)+\ln\det\B(t))>-C,
		\end{equation*}
		which is a~very important inequality as it ensures that $\theta>0$ and $\B$ is positive definite almost everywhere. Although one also gets $\xi\in L^1(0,T;L^1(\Omega))$ after integrating \eqref{s11eta} and using $\eqref{vthe}_1$, this information turns out to be too weak. Instead, we can get better estimates directly from \eqref{s8} and \eqref{s9}. 
		
		Due to the~positive definiteness of $\B$, the~equation~\eqref{s8} can be tested by the~matrix power $\B^{q-1}$. (Though here one can also use $|\B|^{q-2}\B$ since $q\geq1$ and the~stress diffusion term is actually not important for the~estimate itself.)  Then, using \eqref{Ala}, \eqref{APcoer}, Young's inequality and Lemma~\ref{PDalg} below, we eventually get
		\begin{equation}\label{Best0}
			\f{\dd{}}{\dd{t}}\ii\tr\B^{q}+\ii|\B|^{2q}+\ii|\nn\B^{\f{q}2}|^2\leq C\ii|\B|^{q}|\Dv|+C\leq C\ii|\Dv|^{2}+C.
		\end{equation}
		Hence, integrating over $(0,T)$ and thanks to $\eqref{data}_3$, we have
		\begin{equation}\label{Best}
			\norm{\B}_{L^{2q}L^{2q}}\leq C\norm{\Dv}_{L^2L^2}^{\f1q}+C.
		\end{equation}
		Clearly, we need control over $\Dv$, but it has to be obtained differently than for the~Navier--Stokes--Fourier systems, as we pointed out in the~introduction. 
		
		
		Thanks to $\theta>0$, we may test \eqref{s9}  by the~function $-\theta^{-\beta}$ with~$\beta\ge 0$. Eventually, applying $\eqref{vthe}_1$, \eqref{Anu} and \eqref{Akap}, this leads to the~estimate
		\begin{equation}\label{keys}
			\beta\int_{\Q}\theta^{r-\beta-1}|\nn\theta|^2+\int_{\Q}|\Dv|^2\leq C\int_{\Q} \theta |\B||\Dv|+C.
		\end{equation}
		Using \eqref{vthe}, \eqref{Best} and the~H\"{o}lder inequality, the~above inequality gives
		\begin{equation}\label{keys2}
			\begin{split}
				\beta \|\theta^{\frac{r-\beta+1}{2}}\|^2_{L^2W^{1,2}}+\|\Dv\|_{L^2L^2}^2&\leq C\|\theta\|_{L^{2q'}L^{2q'}} \|\B\|_{L^{2q}L^{2q}}\|\Dv\|_{L^2L^2}+C\\
				&\le C\|\theta\|^{2q'}_{L^{2q'}L^{2q'}} + \f12\|\Dv\|^2_{L^2L^2}+C.
			\end{split}
		\end{equation}
		The last term is absorbed by the~left-hand side and for the~first term we use the~interpolation inequality
		\begin{equation*}
			\|\theta\|^{2q'}_{2q'}\le \|\theta\|_{1}^{2q'-\frac{d(r-\beta+1)(2q' -1)}{d(r-\beta)+2}} \|\theta\|_{\frac{d(r-\beta+1)}{d-2}}^{\frac{d(r-\beta+1)(2q' -1)}{d(r-\beta)+2}}
		\end{equation*}
		and \eqref{vthe} to deduce
		\begin{equation}\label{keys3}
			\begin{split}
				&\beta \|\theta^{\frac{r-\beta+1}{2}}\|^2_{L^2W^{1,2}}+\|\Dv\|_{L^2L^2}^2\le C\int_0^T \|\theta\|_{\frac{d(r-\beta+1)}{d-2}}^{\frac{d(r-\beta+1)(2q' -1)}{d(r-\beta)+2}}+C\\
				&= C\int_0^T \|\theta^{\frac{r-\beta+1}{2}}\|_{\frac{2d}{d-2}}^{\frac{2d(2q' -1)}{d(r-\beta)+2}}+C\le  C\int_0^T \|\theta^{\frac{r-\beta+1}{2}}\|_{1,2}^{\frac{2d(2q' -1)}{d(r-\beta)+2}}+C.
			\end{split}
		\end{equation}
		Hence, if
		\begin{equation}
			\label{beta11}
			\frac{2d(2q' -1)}{d(r-\beta)+2}<2,
		\end{equation}
		the first term on the~right-hand side can be absorbed by the~left-hand side and thus, we get
		\begin{equation}\label{keys4}
			\begin{split}
				\beta \|\theta^{\frac{r-\beta+1}{2}}\|^2_{L^2W^{1,2}}+\|\Dv\|_{L^2L^2}^2+\|\B\|^{2q}_{L^{2q}L^{2q}}&\leq C.
			\end{split}
		\end{equation}
		Finally, the~inequality~\eqref{beta11} can be made true by choosing $\beta>0$ sufficiently small if and only if $q$ and $r$ satisfy \eqref{A0}. Note that, in this case, we were able to estimate the~right-hand side of \eqref{keys}, i.e., the~``critical'' term $\theta\B\cdot\Dv$ appearing in \eqref{s9}. It is easy to verify, using estimates \eqref{vthe} and \eqref{keys4} that all the~other nonlinear terms appearing in the~system \eqref{s1}--\eqref{s11eta} are integrable as well.
		
		\subsection*{Definition of weak solution}
		
		Motivated by the~above estimates, we now deliver the~exact definition of a~weak solution to \eqref{s1}--\eqref{s15ic}.
		
		\begin{definition}\label{Dws}
			Let $T>0$ and let $\Omega\subset\R^d$, $d\geq2$, be a~Lipschitz domain. Assume that the~constants $a\in\R$, $\alpha\geq0$, $c_v,\mu>0$ and the~functions $\nu,\kappa,\la,\PP$ fulfil the~assumptions \eqref{Acont}--\eqref{APp} with the~parameters $q$ and $r$ satisfying \eqref{A0}  and let $m\coloneqq\min\{2,\f{4q}{q+2}\}$.  Suppose that the~initial data satisfy
			\begin{align}\label{init}
				&\ve_0\in L^2_{\n,\di}(\Omega;\R^d),\quad\B_0\in L^{q}(\Omega;\R^{d\times d}_{>0}),\quad \theta_0\in L^1(\Omega;\R_{>0}),\\
				&\eta_0\coloneqq c_v\ln\theta_0-f(\B_0)\in L^1(\Omega;\R),\label{initln}
			\end{align}
			where $f$ is given by \eqref{helm},  and that
			\begin{equation}
				\fe\in L^2(\Q;\R^d).
			\end{equation}
			
			Then, we say that the~sextuplet $(\ve,\B,\theta,e,E,\eta):\Q\to\R^d\times\PD\times\R_{>0}\times\R_{>0}\times\R_{>0}\times\R$ is a~weak solution of the~initial-boundary value problem \eqref{s1}--\eqref{s15ic} if all of the~following conditions {\scshape (i)--(iv)} are satisfied:
			
			{\scshape(i)} The~functions $\ve$, $\B$, $\theta$ and $\eta$ fulfil the~properties
			{\allowdisplaybreaks\begin{align}
					\ve&\in L^{2}(0,T;W^{1,2}_{\n,\di})\cap \CO_w([0,T];L^2(\Omega;\R^d)),\label{DR1}\\
					\partial_t\ve&\in L^{\f{d+2}{d}}(0,T;W^{-1,\f{d+2}{d}}_{\n,\di}),\label{DRdv}\\
					\B&\in L^{ m}(0,T;W^{1, m}(\Omega;\PD))\cap \CO_w([0,T];L^{q}(\Omega;\PD)),\label{DRB}\\
					\B&\in L^{2q}(\Q;\PD),\label{DRBint}\\
					\B^{\frac{q}2}&\in L^2(0,T;W^{1,2}(\Omega;\PD)),\label{DRB2}\\
					\partial_t\B&\in \big(L^{2q'}(0,T;W^{1,2q'}(\Omega;\mathbb{R}^{d \times d}))
					\big)^*,\label{DRdtB}\\
					\B^{-\f12}\nn\B\B^{-\f12}&\in L^2(\Q;\R^d\times\Sym),\label{DRnnB}\\
					\ln\det\B&\in L^2(0,T;W^{1,2}(\Omega;\R))\cap L^{\infty}(0,T;L^1(\Omega;\R)),\label{DRnnlnB}\\
					\theta&\in L^{\infty}(0,T;L^1(\Omega;\R_{>0}))\cap L^{r+\f2d+1-\eps}(\Q;\R_{>0}),\label{DRth}\\
					\theta^{\frac{r+1-\eps}{2}}&\in L^2(0,T;W^{1,2}(\Omega;\R_{>0})),\label{DRthW}\\
					\ln\theta&\in L^2(0,T;W^{1,2}(\Omega;\R))\cap L^{\infty}(0,T;L^1(\Omega;\R)),\label{DRlnt}\\
					\eta&\in L^{ m}(0,T;W^{1, m}(\Omega;\R))\cap L^{\infty}(0,T;L^1(\Omega;\R))\label{DReta}
			\end{align}}
			for every $\eps\in (0,1)$. 
			
			{\scshape(ii)} The~ relations \eqref{s1}--\eqref{s5}  hold almost everywhere in $\Q$.
			
			{\scshape(iii)} Equations \eqref{s7}--\eqref{s11eta} are satisfied in the~following sense:
			{\allowdisplaybreaks\begin{align}
					&\begin{aligned}\label{Qv}				&\scal{\partial_t\ve,\fit}-(\ve\otimes\ve,\nn\fit)_{\Q}+(\Sb,\nn\fit)_{\Q}
						+(\alpha\ve_{\tau},\fit_{\tau})_{\SigT}=(\fe,\fit)_{\Q}\\
						&\text{for all}\;\;\fit\in L^{\f{d}2+1}(0,T;W_{\n,\di}^{1,\f{d}2+1}),
					\end{aligned}\\
					&\begin{aligned}\label{QB}
						&\scal{\partial_t\B,\A}-(\B\otimes\ve,\nn\A)_{\Q}+(\PP(\theta,\B),\A)_{\Q}+(\la(\theta)\nn\B,\nn\A)_{\Q}\\
						&\quad=((a\Dv+\Wv)\B,\A+\A^T)_{\Q}\\
						&\text{for all}\;\;\A\in L^{2q'}(0,T;W^{1,2q'}(\Omega;\R^{d\times d})),  
					\end{aligned}\\		&\begin{aligned}\label{Qthe}
						&-(c_v\theta_0,\phi\varphi(0))-(c_v\theta,\phi\partial_t\varphi)_{\Q}-(c_v\theta\ve,\nn\phi\varphi)_{\Q}
						+(\kappa(\theta)\nn\theta,\nn\phi\varphi)_{\Q}\\
						&\quad\geq(\Sb\cdot\Dv,\phi\varphi)_{\Q}\\
						&\text{for all}\;\;\varphi\in W^{1,\infty}((0,T);\R_{\geq0}),\,\varphi(T)=0,\;\;\text{and all}\;\;\phi\in W^{1,\infty}(\Omega;\R_{\geq0}),
					\end{aligned}\\	
					&\begin{aligned}
						&-(\eta_0,\phi\varphi(0))-(\eta,\phi\partial_t\varphi)_{\Q}-(\eta\ve,\nn\phi\varphi)_{\Q}\\ &\quad+\big(\kappa(\theta)\nn\ln\theta-\la(\theta)\nn f(\B),\nn\phi\varphi\big)_{\Q}\geq(\xi,\phi\varphi)_{\Q}\\
						&\text{for all}\;\;\varphi\in W^{1,\infty}((0,T);\R_{\geq0}),\,\varphi(T)=0,\;\;\text{and all}\;\;\phi\in W^{1,\infty}(\Omega;\R_{\geq0}),\label{Qent}
					\end{aligned}\\		
					&\f{\dd{}}{\dd{t}}\ii E + \alpha \int_{\partial \Omega} |\ve|^2=\ii\fe\cdot\ve\;\;\text{a.e.\ in }[0,T].\label{QE}
			\end{align}}
			
			{\scshape(iv)} The~initial data are attained in the~following way:
			\begin{align}
				&\lim_{t\to0+}\norm{\ve(t)-\ve_0}_2=0,\label{Qicv}\\
				&\lim_{t\to0+}\norm{\B(t)-\B_0}_{q-\eps}=0\quad\text{for every}\quad\eps\in(0,q-1],\label{QicB}\\
				&\lim_{t\to0+}\norm{\theta(t)-\theta_0}_1=0,\label{Qict}\\
				&\liminf_{t\to0+}\,(\eta(t),\phi)\geq(\eta_0,\phi)\quad\text{for all}\quad 0\leq\phi\in W^{1,\infty}(\Omega).\label{Qicet}
			\end{align}
		\end{definition}
		
		
		
		With this definition in hand, we now formulate our main result. 
		
		\begin{theorem}\label{LG}
			Suppose that all the~assumptions of Definition~\ref{Dws} are fulfilled. Then, there exists a~weak solution of the~system \eqref{s1}--\eqref{s15ic} in the~sense of Definition~\ref{Dws}. 
			
			In addition, if $d\leq 3$ and $\Omega\in\mathcal{C}^{1,1}$, then there is a~pressure $\pr\in L^{\f{d+2}{d}}(\Q;\R)$ such that the~local balance of total energy \eqref{s10} holds in the~sense:
			\begin{equation}\label{QEloc}\begin{aligned}
					&-(\tfrac12|\ve_0|^2+c_v\theta_0,\phi\varphi(0))-(E,\phi\partial_t\varphi)_{\Q}+(\alpha|\ve_{\tau}|^2,\phi\varphi)_{\SigT}+(\kappa(\theta)\nn\theta,\nn\phi\varphi)_{\Q}\\
					&\hskip4cm=(E\ve+\pr\ve-\Sb\ve,\nn\phi\varphi)_{\Q}\\
					&\text{for all}\;\;\varphi\in W^{1,\infty}((0,T);\R),\;\;\varphi(T)=0,\;\;\text{and every}\;\;\phi\in W^{1,\infty}(\Omega;\R)
			\end{aligned}\end{equation}
			and also \eqref{Qv} can be generalized to
			\begin{equation}\begin{aligned}\label{Qvp}				&\scal{\partial_t\ve,\fit}-(\ve\otimes\ve,\nn\fit)_{\Q}+(-p\I+\Sb,\nn\fit)_{\Q}
					+(\alpha\ve_{\tau},\fit_{\tau})_{\SigT}=(\fe,\fit)_{\Q}\\
					&\hskip2cm\text{for all}\quad\fit\in L^{\f{d}2+1}(0,T;W_{\n}^{1,\f{d}2+1}).
			\end{aligned}\end{equation}
		\end{theorem}
		
		We remark that if a~weak solution admits enough regularity so that \eqref{Qv} can be tested by $\ve$ and \eqref{Qthe} can be localized in space, then \eqref{Qthe} holds as an equality. Indeed, the~localized version of \eqref{Qthe} reads
		\begin{equation}\label{weas2}
			c_v\partial_t\theta+c_v\ve\cdot\nn\theta-\di(\kappa(\theta)\nn\theta)-\Sb\cdot\Dv\geq0.
		\end{equation}
		On the~other hand, subtracting \eqref{Qv} tested by $\ve$ from \eqref{QE} yields
		\begin{equation*}
			\ii(c_v\partial_t\theta-\Sb\cdot\Dv)=0.
		\end{equation*}
		Since also
		\begin{equation*}
			\ii (c_v\ve\cdot\nn\theta-\di(\kappa(\theta)\nn\theta))=0
		\end{equation*}
		due to the~boundary conditions $\ve\cdot\n=0$ and $\nn\theta\cdot\n=0$ on $\partial\Omega$, we conclude from the~above that \eqref{weas2} must be an equality. Consequently, the~entropy inequality~\eqref{Qent} also becomes an equality, provided that one is able to justify $\B^{-1}$ and $\theta^{-1}$ as tests in \eqref{QB} and \eqref{Qthe}. These considerations imply that a~weak solution that admits sufficient regularity is also a~solution of \eqref{s1}--\eqref{s15ic} in the~classical sense.

		The~existence~proof below is done only for $d\geq 3$ (the~case $d=2$ is simpler). Also, it is clearly enough to focus on the~case $\alpha>0$. In the~simpler case $\alpha=0$ (corresponding to the~free-slip boundary condition), one just has to use a~different Korn--Poincar\'{e} inequality in case $\Omega$ is axially symmetric.

		The general strategy of the~proof is to approximate the~system \eqref{s7}, \eqref{s8}, \eqref{s9} using several parameters to obtain a~proper Galerkin approximation generated by a~smooth basis of eigenvectors and to show that the~resulting (ODE) system has a~solution. After that, our aim is to derive the~entropy equation. At this point, possibly irregular terms containing $\theta$ and $\B$ are cut-off and $\ve$ is smooth, hence we easily obtain uniform estimates for the~Galerkin approximations of $\B$ and $\theta$, which might not be positive definite or positive, respectively. However, after taking the~limit with these approximations and then proving certain maximum principles, we prove invertibility of $\theta$ and $\B$, which, in turn, enables us to derive the~entropy equation. From this we read that the~positivity of $\det\B$ and $\theta$ is preserved uniformly, which then enables us to remove the~cut-off from the~system.  The~proof of this is presented in Section~\ref{SS1}. Note that at this point, the~velocity is still kept in a~finite $\ell$-dimensional space. To the~equation for the~internal energy we add the~regularization $-\omega \Delta_{r+2}\theta$ (the so-called $(r+2)$-Laplacian) in order to avoid weighted Sobolev spaces, where the~density of smooth functions is not available in general.
		
		Next, in Section~\ref{SS5}, we first improve the~uniform estimates by considering appropriate test functions in the~equations for $\theta$ and $\B$. At this point such a~procedure is rigorous. Finally, we let $\omega \to 0$ and $\ell\to \infty$ and  we pass to the~final limit, identify the~non-linear terms and initial conditions, hereby obtaining a~solution of the~original problem. Finally, in Section~\ref{SS7}, we prove the~validity of the~local energy equality provided $d\le 3$.			
		
		\section{Existence of a~weak solution: the~approximative problem} \label{SS1}
		First we introduce a~truncation, which is essential for the~proof. We also prepare some simple estimates corresponding to this truncation that are used later in the~proof. Recalling that $\omega_P$ is introduced in \eqref{APpd}, we define, for any $\omega\in(0,\omega_P)$, the~``cut-off'' function $g_{\omega}$ in the~following way:
		\begin{equation*}
			g_{\omega}(\A,\tau)\coloneqq\f{\max\{0,\Lambda(\A)-\omega\}\max\{0,\tau-\omega\}}{(|\Lambda(\A)|+\omega)(1+\omega|\A|^2)(|\tau|+\omega)(1+\omega\tau^2)},\quad\A\in\R^{d\times d}_{\rm sym},\quad\tau\in\R,
		\end{equation*}
		where $\Lambda(\A)$ denotes the~smallest eigenvalue of $\A$, i.e., 
		\begin{equation*}
			\Lambda(\A)\coloneqq\min\{\la:\det(\A-\la\I)=0\}.
		\end{equation*}
		Note that $g_{\omega}$ is a~continuous function in $\R^{d\times d}_{\rm sym}\times\R$ and satisfies $0\leq g_{\omega}(\A,\tau)<1$ for every $(\A,\tau)\in\R^{d\times d}_{\rm sym}\times\R$. Moreover, if $\Lambda(\A)\leq\omega$ or $\tau\leq\omega$, then $g_{\omega}(\A,\tau)=0$, whereas if $\Lambda(\A)>0$ and $\tau>0$, then $g_{\omega}(\A,\tau)\to1$ as $\omega\to0+$.
		Furthermore, we remark that
		\begin{equation}\label{godhad}
			g_{\omega}(\A,\tau)(1+|\A|+|\A|^2)(1+\tau+\tau^2)\leq C(\omega).
		\end{equation}
		The function $g_{\omega}$ is used below in the~system \eqref{galv}--\eqref{galt} to control irregular terms of the~original problem. We also truncate the~initial functions $\B_0$ and $\theta_0$ and set 
		\begin{align}
			&\begin{aligned}\label{Bodef}
				\B_0^{\omega}(x)&\coloneqq\Big\{\begin{matrix}
					\;\B_0(x)&\text{if }\Lambda(\B_0(x))>\omega\;\text{ and }\;|\B_0(x)|<\sqrt{d}\,\omega^{-1},\\
					\;\I &\text{elsewhere};\end{matrix}
			\end{aligned}\\
			&\begin{aligned}\label{todef}
				\theta_0^{\omega}(x)&\coloneqq\Big\{\begin{matrix}
					\;\theta_0(x)&\text{if }\omega<\theta_0(x)<\omega^{-1},\\
					\;1 &\text{elsewhere}.
				\end{matrix}
			\end{aligned}
		\end{align}
		With such definitions, these functions satisfy (a.e. in $\Omega$)
		\begin{align}\label{0zesp}
			\Lambda(\B_0^{\omega}) &>\omega, && \quad\theta_0^{\omega} >\omega \\
			\label{0sh}
			|\B_0^{\omega}|&<\sqrt{d}\,\omega^{-1}, && \,\, |\theta_0^{\omega}|<\omega^{-1} \\
			\label{entOO}
			|\B_0^{\omega}| &\leq \sqrt{d}+|\B_0|, &&\quad\theta_0^{\omega}\leq 1+\theta_0,
		\end{align}
		and, since $\ln 1=0$,
		\begin{align}\label{entO}
			|\ln\det\B_0^{\omega}|\leq|\ln\det\B_0|,\qquad \quad|\ln\theta_0^{\omega}|\leq|\ln\theta_0|.
		\end{align}
		Since $\B_0\in L^{q}(\Omega;\PD)$, we also observe that the~Lebesgue measure of the~sets $\{\Lambda(\B_0)\leq \omega\}$ and $\{|\B_0|\geq\omega^{-1}\}$ tends to zero as $\omega\to0+$, and thus
		\begin{equation}\label{B0con}			\norm{\B_0^{\omega}-\B_0}_{q}^{q}=\int_{\Lambda(\B_0)\leq\omega}|\I-\B_0|^{q}+\int_{|\B_0|\geq\omega^{-1}}|\I-\B_0|^{q}\to0.
		\end{equation}
		Analogously, relying on $\theta_0\in L^1(\Omega;\R_{>0})$, we also obtain
		\begin{equation}\label{t0con}
			\norm{\theta_0^{\omega}-\theta_0}_1\to0,\quad\omega\to0+.
		\end{equation}	
		
		Next, we discretize the~$\omega$-truncated system in space by the~Galerkin method.\footnote{ With this approach, we do not need the~positive definiteness of the~basis functions for $\B$. } Let $\{\we_i\}_{i=1}^{\infty}$, $\{\Wb_j\}_{j=1}^{\infty}$ and $\{w_k\}_{k=1}^{\infty}$ be bases of $W^{N,2}(\Omega;\R^d)\cap W^{1,2}_{\n,\di}$, $W^{N,2}(\Omega;\Sym)$ and $W^{N,2}(\Omega;\R)$, respectively, with the~following properties:
		\begin{itemize}
			\item The~bases are $L^2$-orthonormal and $W^{N,2}$-orthogonal.
			\item The~number $N\in\N$ is chosen so large that the~elements of the~bases are Lipschitz (due to embeddings of Sobolev spaces).
			\item $w_1=|\Omega|^{-\f12}$.
			\item For any $\ell,n\in\N$, there exist $L^2$-orthogonal projections
			\begin{align*}
				P_{\ell}:L^2(\Omega;\R^d)&\to \spa\{\we_i\}_{i=1}^{\ell},\\
				Q_n:L^2(\Omega;\R^{d\times d})&\to \spa\{\Wb_j\}_{j=1}^n,\\
				R_n:L^2(\Omega;\R)&\to \spa\{w_k\}_{k=1}^n
			\end{align*}
			\item\label{projc}$
			P_{\ell},Q_n,R_n\text{ are }L^2\text{- and }W^{N,2}\text{-bounded, uniformly w.r.t. }\ell,n.$
		\end{itemize}
		Existence of these bases and corresponding projections follows from standard results (see Appendix~4 in \cite{Malek1996}) using the~eigenvectors of the~generalized Laplace or Stokes operators.
		
		We fix $\ell,n\in\N$ and consider the~problem of finding the~functions $\alpha_{\ell n}^i$, $\beta_{\ell n}^j$, $\gamma_{\ell n}^k$ of time, where $i=1,\ldots,\ell$ and $j,k=1,\ldots, n$, such that the~functions $\ve_{\ell n}$, $\B_{\ell n}$, $\theta_{\ell n}$ and $\Sb_{\ell n}^{\omega}$ defined as
		\begin{equation*}
			\ve_{\ell n}(t,x)=\sum_{i=1}^{\ell}\alpha^i_{\ell n}(t)\we_i(x),\;
			\B_{\ell n}(t,x)=\sum_{j=1}^n\beta^j_{\ell n}(t)\Wb_j(x),\; \theta_{\ell n}=\sum_{k=1}^n\gamma^k_{\ell n}(t)w_k(x)
		\end{equation*}
		and
		\begin{equation}\label{Som}
			\Sb_{\ell n}^{\omega}\coloneqq2\nu(\theta_{\ell n})\Dv_{\ell n}+2a\mu g_{\omega}(\B_{\ell n},\theta_{\ell n})\theta_{\ell n}\B_{\ell n}
		\end{equation}
		satisfy the~following equations a.e.\ in $(0,T_0)$, $T_0>0$:
		\begin{align}
			&\begin{aligned}
				&(\partial_t\ve_{\ell n},\we_i)-(\ve_{\ell n}\otimes\ve_{\ell n},\nn\we_i)+(\Sb_{\ell n}^{\omega},\nn\we_i)+\alpha(\ve_{\ell\n},\fit)_{\partial\Omega}=(\fe,\we_i),\label{galv}
			\end{aligned}\\[0.1cm]
			&\begin{aligned}
				&(\partial_t\B_{\ell n},\Wb_j)-(\B_{\ell n}\otimes\ve_{\ell},\nn\Wb_j)+(\PP(\theta_{\ell n},\B_{\ell n}),\Wb_j)+(\la(\theta_{\ell n})\nn\B_{\ell n},\nn\Wb_j)\\
				&\hskip4.5cm=(2g_{\omega}(\B_{\ell n},\theta_{\ell n})(a\Dv_{\ell n}+\Wv_{\ell n})\B_{\ell n},\Wb_j),\label{galB}
			\end{aligned}\\[0.1cm]
			&\begin{aligned}
				&(c_v\partial_t\theta_{\ell n},w_k)-(c_v\theta_{\ell}\ve_{\ell n},\nn w_k)+((\kappa(\theta_{\ell n})+\omega|\nn\theta_{\ell n}|^r)\nn\theta_{\ell n},\nn w_k)\\
				&\qquad\qquad\qquad=(\Sb_{\ell n}^{\omega}\cdot\Dv_{\ell n},w_k)\label{galt},
			\end{aligned}
		\end{align}
		for all $1\leq i\leq\ell$, $1\leq j,k\leq n$ and with the~initial conditions
		\begin{equation}\label{galic}
			\ve_{\ell n}(0)=P_{\ell}\ve_0,\quad \B_{\ell n}(0)=Q_{n}\B_0^{\omega},\quad \theta_{\ell n}(0)=R_n\theta_0^{\omega}\quad\text{in }\Omega.
		\end{equation}
		By the~$L^2$-orthonormality of the~bases, we have 
		\begin{equation*}
			(\partial_t\ve_{\ell n},\we_i)=\sum_{m=1}^{\ell}\partial_t\alpha_{\ell n}^m(\we_m,\we_i)=(\alpha_{\ell n}^i)'
		\end{equation*}
		and similarly
		\begin{equation*}
			(\partial_t\B_{\ell n},\Wb_j)=(\beta^j_{\ell n})'\quad \textrm{ and } \quad(\partial_t\theta_{\ell n},w_k)=(\gamma^k_{\ell n})'.
		\end{equation*}
		Thus, \eqref{galv}--\eqref{galt} is a~system of $\ell+2n$ ordinary differential equations of the~form
		\begin{equation}\left.\hskip2cm\begin{aligned}\label{carat}
				(\alpha_{\ell n}^i)'&=F_1(t,\alpha_{\ell n}^1,\ldots,\alpha_{\ell n}^{\ell}),&&i=1,\ldots,\ell,\\
				(\beta_{\ell n}^j)'&=F_2(\beta_{\ell n}^1,\ldots,\beta_{\ell n}^n),&&j=1,\ldots,n,\\
				(\gamma_{\ell n}^k)'&=F_3(\gamma_{\ell n}^1,\ldots,\gamma_{\ell n}^n),&&k=1,\ldots,n.
			\end{aligned}\hskip2cm\right\}\end{equation}
		It is easy to see, using \eqref{Acont}, that $F_1,F_2$ and $F_3$ are continuous with respect to the~variables $\alpha_{\ell n}^i$, $\beta_{\ell n}^j$ and $\gamma_{\ell n}^k$ and measurable with respect to $t$, respectively. Moreover, the~explicit dependence of $F_1$ on time is controlled by
		\begin{equation*}
			|(\fe,\we_i)|\leq\norm{\fe}_2\norm{\we_i}_2\in L^2(0,T;\R).
		\end{equation*}
		Thus, we can apply the~Carath\'{e}odory existence theorem (see \cite[Chapter~2, Theorem~1]{Coddington1955} or \cite[Chapter~30]{Zeidler1990}) and hereby obtain absolutely continuous functions $\alpha^i_{\ell n}$, $\beta^j_{\ell n}$, $\gamma_{\ell n}^k$, $1\leq i\leq\ell$, $1\leq j,k\leq n$, solving \eqref{carat} on $(0,T_0)$, where $T_0<T$ is the~time of the~first blow-up. In view of the~a~priori estimates derived below (see e.g.\ \eqref{alest}), we are able to prove that
		\begin{equation*}
			\sup_{t\in(0,T_0)}\Big(\sum_{i=1}^{\ell}(\alpha_{\ell n}^i(t))^2+\sum_{j=1}^n(\beta_{\ell n}^j(t))^2+\sum_{k=1}^n(\gamma_{\ell n}^k(t))^2\Big)<\infty,
		\end{equation*}
		hence, there can be no blow-up and the~functions $\ve_{kl},\B_{kl},\theta_{kl}$ are defined on an arbitrary time interval, in particular on $[0,T]$.
		
		\subsection*{Estimates uniform with respect to \texorpdfstring{$n$}{n}}
		By multiplying the~$i$-th equation in~\eqref{galv} by $\alpha^i_{\ell n}$, summing the~result over all $i=1,\ldots,\ell$, integrating by parts and using the~facts that the~basis functions satisfy $\ve_{\ell n}\cdot \n = 0$ on $\partial \Omega$ and $\di \ve_{\ell n} = 0$ in $\Omega$ (hence the~convective term vanishes), we obtain
		(a.e. in $(0,T)$)
		\begin{align}\begin{aligned}\label{vest}
				&\f12\f{\dd{}}{\dd{t}}\norm{\ve_{\ell n}}_2^2+\norm{\sqrt{2\nu(\theta_{\ell n})}\Dv_{\ell n}}_2^2+\alpha\norm{\ve_{\ell n}}^2_{L^2(\partial\Omega;\R^d)}\\
				&\hskip3cm=-(2a\mu g_{\omega}(\B_{\ell n},\theta_{\ell n})\theta_{\ell n}\B_{\ell n},\D\ve_{\ell n})+(\fe,\ve_{\ell n}).
			\end{aligned}
		\end{align}
		Then we use \eqref{Anu}, \eqref{godhad}, Korn's and Young's inequality, and deduce 
		\begin{align*}
			\f{\dd{}}{\dd{t}}\norm{\ve_{\ell n}}_2^2+\norm{\nn\ve_{\ell n}}_2^2+\alpha\norm{\ve_{\ell n}}^2_{L^2(\partial\Omega;\R^d)}&\leq C(\omega)\ii|\Dv_{\ell n}|+C\norm{\fe}_{2}\norm{\nn\ve_{\ell n}}_{2}\\
			&\leq C(\omega)+C\norm{\fe}_{2}^{2}+\f12\norm{\nn\ve_{\ell n}}_2^2.
		\end{align*}
		Integration with respect to time and the~use of \eqref{galic} and \eqref{init} directly leads to
		\begin{equation}\label{est1}
			\sup_{t\in(0,T)}\norm{\ve_{\ell n}(t)}_2^2+\int_0^T\norm{\nn\ve_{\ell n}}_2^2+\alpha\int_0^T\norm{\ve_{\ell n}}^2_{L^2(\partial\Omega;\R^d)}\leq C(\omega).
		\end{equation}
		(the~dependence of the~constant $C$ on the~data is omitted as $\fe$, $\ve_0$, $\theta_0$, or $\B_0$ are fixed functions in our setting). Utilizing the~$L^2$-orthonormality of the~basis vectors $\{\we_i\}_{i=1}^{\ell}$, estimate \eqref{est1} yields
		\begin{equation}\label{alest}
			\sup_{t\in(0,T)}\sum_{i=1}^{\ell}(\alpha^i_{\ell n}(t))^2=\sup_{t\in(0,T)}\norm{\ve_{\ell n}(t)}_2^2\leq C(\omega).
		\end{equation}
		Hence, recalling $\we_i\in W^{1,\infty}(\Omega;\R^d)$, $i=1,\ldots,\ell$ and then also the~definition \eqref{Som} and the~estimate \eqref{godhad}, we obtain
		\begin{equation}\label{vLip}
			\norm{\ve_{\ell n}}_{L^{\infty}W^{1,\infty}}+\norm{\Sb_{\ell n}^{\omega}}_{L^{\infty}L^{\infty}}\leq C(\omega,\ell).
		\end{equation}
		Using \eqref{vLip} in~\eqref{galv}, we see that
		\begin{align}\label{dert}
			&\norm{(\alpha^i_{\ell n})'}_{L^2(0,T;\R)}=\norm{(\partial_t\ve_{\ell n},\we_i)}_{L^2(0,T;\R)}\nonumber\\
			&\quad=\norm{(\ve_{\ell n}\otimes\ve_{\ell n}-\Sb_{\ell n}^{\omega},\nn\we_i)-\alpha(\ve_{\ell n},\we_i)_{\partial\Omega}+(\fe,\we_i)}_{L^2(0,T;\R)}\\
			&\quad\leq C(\omega,\ell)+C(\ell)\norm{\fe}_{L^{2}L^2}.\nonumber
		\end{align}
		Thus, we get
		\begin{equation}\label{dvLip}
			\norm{\partial_t\ve_{\ell n}}_{L^{2}W^{1,\infty}}=\norm{{\textstyle\sum}_{i=1}^{\ell}(\alpha^i_{\ell n})'\we_i}_{L^{2}W^{1,\infty}}\leq C(\omega,\ell)
		\end{equation}
		and, using the~fundamental theorem of calculus and H\"older's inequality, also that
		\begin{equation}\label{unifc}
			|\alpha^i_{\ell n}(t)-\alpha_{\ell n}^i(s)|\leq\int_s^t|(\alpha^i_{\ell n})'|\leq C(\omega,\ell)|t-s|^{\f12}\quad\text{for every }t,s\in[0,T]
		\end{equation}
		and any $i=1,\ldots,\ell$.
		
		Next, we multiply the~$j$-th equation in~\eqref{galB} by $\beta^j_{\ell n}$ and sum the~result over $j=1,\ldots,n$. Note that the~convective term vanishes after integration by parts and use of $\eqref{s12bc}_1$ and \eqref{s6d}. Also the~term including $\Wv_{\ell n}$ vanishes due to symmetry of $\B_{\ell n}^2$. Thus, we obtain
		\begin{align}\begin{aligned}\label{r1}
				\f12\f{\dd{}}{\dd{t}}\norm{\B_{\ell n}}_2^2&+(\PP(\theta_{\ell n},\B_{\ell n}),\B_{\ell n})+\norm{\sqrt{\la(\theta_{\ell n})}\nn\B_{\ell n}}_2^2\\
				&\hskip1cm=(2ag_{\omega}(\B_{\ell n},\theta_{\ell n})\Dv_{\ell n}\B_{\ell n},\B_{\ell n}) \quad \textrm{ a.e. in } (0,T).
		\end{aligned}\end{align}
		Then using \eqref{galic}, \eqref{APcoer}, \eqref{Ala} and \eqref{godhad} we obtain, after integration over $(0,t)$, $t\in(0,T)$, that
		\begin{equation*}
			\norm{\B_{\ell n}(t)}_2^2+\int_0^t\norm{\B_{\ell n}}_{2+q}^{2+q}+\int_0^t\norm{\nn\B_{\ell n}}_2^2\leq\norm{Q_n\B_0^{\omega}}_{2}^2+ C(\omega,\ell).
		\end{equation*}
		From this, using properties of $Q_n$ and \eqref{0sh}, we easily read that
		\begin{equation}\label{best}
			\norm{\B_{\ell n}}_{L^{\infty}L^2}+\norm{\B_{\ell n}}_{L^{2+q}L^{2+q}}+\norm{\nn\B_{\ell n}}_{L^2L^2}\leq C(\omega,\ell). 
		\end{equation}
		To estimate the~time derivative of $\B_{\ell n}$, we take $\A\in L^{q+2}(0,T;W^{N,2}(\Omega))$ with $\norm{\A}_{L^{q+2}W^{N,2}}\leq 1$ and use \eqref{galB}, H\"older's inequality, \eqref{best}, \eqref{vLip}, \eqref{Ala}, \eqref{APbou}, \eqref{godhad}, properties of $Q_n$ and $(\min\{2,\f{q+2}{q+1}\})'=q+2$ to get 
		\begin{align*}
			\scal{\partial_t\B_{\ell n},\A}
			&=(\partial_t\B_{\ell n},Q_n\A)_{\Q}\\
			&\quad=(\B_{\ell n}\otimes\ve_{\ell n}-\la(\theta_{\ell n})\nn\B_{\ell n},\nn Q_n\A)_{\Q}-(\PP(\theta_{\ell n},\B_{\ell n}),Q_n\A)_{\Q}\\
			&\quad\quad+(2g_{\omega}(\B_{\ell n},\theta_{\ell n})(a\Dv_{\ell n}+\Wv_{\ell n})\B_{\ell n},Q_n\A)_{\Q}\\
			&\quad\leq C(\omega,\ell)\int_{\Q}\big((|\B_{\ell n}|+|\nn\B_{\ell n}|)|\nn Q_n\A|+(|\B_{\ell n}|^{q+1}+1)|Q_n\A|\big)\\
			&\quad\leq C(\omega,\ell)\int_0^T(\norm{\nn\B_{\ell n}}_1+\norm{\B_{\ell n}}_{q+1}^{q+1}+1)\norm{Q_n\A}_{1,\infty}\\
			&\quad\leq C(\omega,\ell)\int_0^T(\norm{\nn\B_{\ell n}}_2+\norm{\B_{\ell n}}_{q+2}^{q+1}+1)\norm{Q_n\A}_{N,2}\\
			&\quad\leq C(\omega,\ell)\norm{\A}_{L^{q+2}W^{N,2}}\leq C(\omega,\ell).
		\end{align*}
		Hence, we can conclude
		\begin{equation}\label{dtB}
			\norm{\partial_t\B_{\ell n}}_{L^{\f{q+2}{q+1}}W^{-N,2}}\leq C(\omega,\ell).
		\end{equation}
		
		Next, we multiply the~$k$-th equation in~\eqref{galt} by $\gamma^k_{\ell n}$, sum the~result over $k=1,\ldots,n$, use $\eqref{s12bc}_1$, $\eqref{s6d}$ and integration by parts in the~convective term to get
		\begin{equation}\begin{aligned}\label{potest}
				&\f{c_v}2\f{\dd{}}{\dd{t}}\norm{\theta_{\ell n}}_2^2+\norm{\sqrt{\kappa(\theta_{\ell n})}\nn \theta_{\ell n}}_2^2+\omega\norm{\nn\theta_{\ell n}}_{r+2}^{r+2}=(\Sb_{\ell n}^{\omega}\cdot\Dv_{\ell n},\theta_{\ell n})
		\end{aligned}\end{equation}
		a.e.~in $(0,T)$.  Thus, integrating this inequality over time, using \eqref{vLip} and Young's, Gronwall's and Poincar\'e's inequalities, properties of $R_n$ and \eqref{0sh}, we deduce
		\begin{equation}\label{test}
			\norm{\theta_{\ell n}(t)}_{L^{\infty}L^2}+\norm{\sqrt{\kappa(\theta_{\ell n})}\nn \theta_{\ell n}}_{L^2L^2}+\norm{\theta_{\ell n}}_{L^{r+2}W^{1,r+2}}\leq C(\omega,\ell).
		\end{equation}
		Furthermore, taking $\tau\in L^{r+2}(0,T;W^{N,2}(\Omega))$ with $\norm{\tau}_{L^{r+2}W^{N,2}}\leq1$ and using \eqref{galt}, Young's inequality, H\"older's inequality, \eqref{Akap}, \eqref{vLip}, \eqref{test} and properties of $R_n$, we obtain 
		\begin{equation*}\begin{aligned}
				&\scal{\partial_t\theta_{\ell n},\tau}=(\partial_t\theta_{\ell n},R_n\tau)_{\Q}\\
				&\quad=(c_v\theta_{\ell n}\ve_{\ell n}-\kappa(\theta_{\ell n})\nn\theta_{\ell n}-\omega|\nn\theta_{\ell n}|^r\nn\theta_{\ell n},\nn R_n\tau)_{\Q}+(\Sb_{\ell n}^{\omega}\cdot\Dv_{\ell n},R_n\tau)_{\Q}\\
				&\quad\leq C(\omega,\ell)\int_{\Q}\Big(\big(|\theta_{\ell n}|+|\theta_{\ell n}|^{\f r2}\big|\sqrt{\kappa(\theta_{\ell n})}\nn\theta_{\ell n}\big|+|\nn\theta_{\ell n}|^{r+1}\big)|\nn R_n\tau|+|R_n\tau|\Big)\\
				&\quad\leq C(\omega,\ell)\int_0^T\ii\Big(|\theta_{\ell n}|^{r+1}+\left|\sqrt{\kappa(\theta_{\ell n})}\nn\theta_{\ell n}\right|^{\f{2r+2}{r+2}}+|\nn\theta_{\ell n}|^{r+1}+1\Big)\norm{R_n\tau}_{1,\infty}\\
				&\quad\leq C(\omega,\ell)\int_0^T\Big(\norm{\theta_{\ell n}}_{r+2}^{r+1}+\norm{\sqrt{\kappa(\theta_{\ell n})}\nn\theta_{\ell n}}_2^{\f{2r+2}{r+2}}+\norm{\nn\theta_{\ell n}}_{r+2}^{r+1}+1\Big)\norm{R_n \tau}_{N,2}\\
				&\quad\leq C(\omega,\ell)\norm{\tau}_{L^{r+2}W^{N,2}}\leq C(\omega,\ell),
		\end{aligned}\end{equation*}
		hence
		\begin{equation}\label{dt}
			\norm{\partial_t\theta_{\ell n}}_{L^{\f{r+2}{r+1}}W^{-N,2}}\leq C(\omega,\ell).
		\end{equation}
		
		\subsection*{The~limit \texorpdfstring{$n\to\infty$}{n}}
		
		For every $i=1,\ldots,\ell$, the~sequence $\{\alpha^i_{\ell n}\}_{n=1}^{\infty}\subset\CO([0,T];\R)$ is bounded due to~\eqref{alest} and uniformly equicontinuous by \eqref{unifc}. Hence, using the~Arzel\`a-Ascoli theorem, for every $i=1,\ldots,\ell$, we obtain $\alpha^i_{\ell}\in\CO([0,T];\R)$ and a~subsequence (not relabelled) such that
		\begin{equation}\label{stroalf}
			\alpha^i_{\ell n}\to\alpha^i_{\ell}\quad\text{strongly in }\CO([0,T];\R)
		\end{equation}
		as $n\to\infty$. Then, we define
		\begin{equation*}
			\ve_{\ell}\coloneqq\sum_{i=1}^{\ell}\alpha^i_{\ell}\we_i\in \CO([0,T];W^{1,\infty}(\Omega;\R^d)\cap W^{1,2}_{\n,\di})
		\end{equation*}
		and note that
		\begin{equation}\label{strov}
			\ve_{\ell n}\to\ve_{\ell}\quad\text{strongly in }\CO([0,T];W^{1,\infty}(\Omega;\R^d)).
		\end{equation}
		According to estimates \eqref{dvLip}, \eqref{best}, \eqref{dtB}, \eqref{test}, \eqref{dt} and using reflexivity of the~underlying spaces and the~Aubin--Lions lemma, there exist subsequences $\{\ve_{\ell n}\}_{n=1}^{\infty}$, $\{\B_{\ell n}\}_{n=1}^{\infty}$, $\{\theta_{\ell n}\}_{n=1}^{\infty}$ and their limits $\ve_{\ell}$, $\B_{\ell}$, $\theta_{\ell}$, such that
		{\allowdisplaybreaks\begin{align}
				\partial_t\ve_{\ell n}&\wcs\partial_t\ve_{\ell}&&\text{weakly* in }L^{2}(0,T;W^{1,\infty}(\Omega;\R^d)),\label{dvl}\\
				\ve_{\ell n}&\wc\ve_{\ell}&&\text{weakly in }L^2(0,T;L^2(\partial\Omega;\R^d)),\label{tracen}\\
				\B_{\ell n}&\wc\B_{\ell}&&\text{weakly in }L^2(0,T,W^{1,2}(\Omega;\Sym)),\label{Bw}\\
				\B_{\ell n}&\to\B_{\ell}&&\text{strongly in }L^{2+q)}(\Q;\Sym)\text{ and a.e.\ in }\Q,\label{bstr}\\[-0.1cm]
				\partial_t\B_{\ell n}&\wc\partial_t\B_{\ell}&&\text{weakly in }L^{\f{q+2}{q+1}}(0,T;W^{-N,2}(\Omega;\Sym)),\label{gBd}\\
				\theta_{\ell n}&\wc\theta_{\ell}&&\text{weakly in }L^{r+2}(0,T,W^{1,r+2}(\Omega;\R))\label{thw},\\
				\theta_{\ell n}&\to\theta_{\ell}&&\text{strongly in }L^{r+2+\f4d)}(\Q;\Sym)\text{ and a.e.\ in }\Q,\label{strt}\\
				\partial_t\theta_{\ell n}&\wc\partial_t\theta_{\ell}&&\text{weakly in }L^{\f{r+2}{r+1}}(0,T;W^{-N,2}(\Omega;\R))\label{tep}.
		\end{align}}
		Now, we explain how to take the~limit in the~non-linear terms appearing in~\eqref{galv}, \eqref{galB} and \eqref{galt}. To handle most of the~terms, namely
		\begin{align*}
			&\ve_{\ell n}\otimes\ve_{\ell n},\quad\nu(\theta_{\ell n})\Dv_{\ell n},\quad \Sb_{\ell n}^{\omega},\quad \PP(\theta_{\ell n},\B_{\ell n}),\quad\la(\theta_{\ell n})\nn\B_{\ell n},\\
			&\quad g_{\omega}(\B_{\ell n},\theta_{\ell n})(a\Dv_{\ell n}+\Wv_{\ell n})\B_{\ell n},\quad\ve_{\ell n}\cdot\nn\theta_{\ell n},\quad \Sb_{\ell n}^{\omega}\cdot\Dv_{\ell n},
		\end{align*}
		we use the~following standard argument: all these terms can be seen as a~product of a~weakly converging sequence with a~strongly converging sequence, obtained via Vitali's theorem, \eqref{Acont}, continuity of $g_{\omega}$ and pointwise convergence of $\ve_{\ell n}$, $\B_{\ell n}$ and $\theta_{\ell n}$.
		This argument is sufficient to take the~limit $n\to\infty$ in the~equations \eqref{galv} and \eqref{galB}. In \eqref{galB}, we first multiply the~equation by a~function $\varphi\in\CO^1([0,T];\R)$, integrate over $(0,T)$, then take the~limit and finally use the~density of functions of the~form $\varphi\A$, $\A\in\spa\{\Wb_j\}_{j=1}^{\infty}$, in the~space $L^{(q+2)'}(0,T;W^{N,2}(\Omega;\Sym))$. This way, defining also
		\begin{equation*}
			\Sb_{\ell}^{\omega}\coloneqq 2\nu(\theta_{\ell})|\Dv_{\ell}|^2+2a\mu g_{\omega}(\B_{\ell},\theta_{\ell})\theta_{\ell}\B_{\ell},
		\end{equation*}
		we obtain
		\begin{align}\begin{aligned}
				&(\partial_t\ve_{\ell},\we_i)-(\ve_{\ell}\otimes\ve_{\ell},\nn\we_i)+(\Sb_{\ell}^{\omega},\nn\we_i)+\alpha(\ve_{\ell},\we_i)_{\partial\Omega}=(\fe,\we_i)\\ &\hspace{2cm} \text{ for every }i=1,\ldots,\ell, \textrm{ and a.e. in } (0,T),\label{sys1v}
		\end{aligned}\end{align}
		and
		\begin{align}\begin{aligned}
				&\scal{\partial_t\B_{\ell},\A}-(\B_{\ell}\otimes\ve_{\ell},\nn\A)_{\Q}+(\PP(\theta_{\ell},\B_{\ell}),\A)_{\Q}+(\la(\theta_{\ell})\nn\B_{\ell},\nn\A)_{\Q}\\
				&\quad=(2g_{\omega}(\B_{\ell},\theta_{\ell})(a\Dv_{\ell}+\Wv_{\ell})\B_{\ell},\A)_{\Q}\\
				&\hspace{3cm}\text{for all }\A\in L^{q+2}(0,T;W^{N,2}(\Omega;\Sym)).\label{sys1B}
		\end{aligned}\end{align}
		However, the~space of test functions in~\eqref{sys1B} can be enlarged using a~standard density argument. Indeed, using H\"older's inequality, it is easy to see that every term of \eqref{sys1B} (taking aside the~time derivative) is well defined provided that
		\begin{equation*}
			\A\in L^2(0,T;W^{1,2}(\Omega;\Sym))\cap L^{q+2}(\Q;\Sym)
		\end{equation*}
		and thus, we can read from \eqref{sys1B} that
		\begin{equation*}
			\partial_t\B_{\ell}\in \left(L^2(0,T;W^{1,2}(\Omega;\Sym))\cap L^{q+2}(\Q;\Sym)\right)^*.
		\end{equation*}
		Since we also have that
		\begin{equation}\label{Bellreg}
			\B_{\ell}\in L^2(0,T;W^{1,2}(\Omega;\Sym))\cap L^{q+2}(\Q;\Sym),
		\end{equation}
		it follows from Lemma~\ref{Temb} below that
		\begin{equation}\label{BCo}
			\B_{\ell}\in \CO([0,T];L^2(\Omega;\Sym)).
		\end{equation}
		The~value of $\B_{\ell}(0)$ can be identified by a~standard argument, which we briefly outline here.  Using $\A(t,x)=\psi(t)\PP(x)$ in~\eqref{sys1B}, where $\psi\in \CO^1([0,T];\R)$, $\psi(0)=1$, $\psi(T)=0$, and $\PP\in W^{N,2}(\Omega;\Sym)$, one gets, after integration by parts, that
		\begin{align}\begin{aligned}\label{id0}
				(\B_{\ell}(0),\PP)&=-(\B_{\ell},\PP\partial_t\psi)_{\Q}+(\ve_{\ell}\cdot\nn\B_{\ell},\PP\psi)_{\Q}+(\PP(\theta_{\ell},\B_{\ell}),\PP\psi)_{\Q}\\
				&\qquad-(\la(\theta_{\ell})\nn\B_{\ell},\nn\PP\psi)_{\Q}-(2g_{\omega}(\B_{\ell},\theta_{\ell})(a\Dv_{\ell}+\Wv_{\ell})\B_{\ell},\PP\psi)_{\Q}.\end{aligned}
		\end{align}
		On the~other hand, exactly the~same expression can be obtained also for $(\B_0^{\omega},\PP)$ if one multiplies \eqref{galB} by $\psi$, integrate over $(0,T)$ and by parts in the~time derivative using \eqref{galic}
		and uses completeness of $\{\Wb_j\}_{j=1}^{\infty}$ in $W^{N,2}(\Omega;\Sym)$ and the~same arguments as before to take the~limit $n\to\infty$. But since $\PP$ was arbitrary and $W^{N,2}(\Omega;\Sym)$ is dense in $L^2(\Omega;\Sym)$, we conclude 
		\begin{equation}\label{BB}
			\B_{\ell}(0)=\B_0^{\omega}.
		\end{equation}
		We can use an~analogous procedure to identify $\ve_{\ell}(0)$, but here the~situation is simpler since \eqref{strov} directly implies $\ve_{\ell}\in \CO([0,T];W^{1,\infty}(\Omega;\R^d))$ and we obtain
		\begin{equation}\label{vv}
			\ve_{\ell}(0)=P_{\ell}\ve_0.
		\end{equation}
		
		Our aim is now to take the~limit in equation~\eqref{galt}, where we need to justify the~limit in the~terms $\kappa(\theta_{\ell n})\nn\theta_{\ell n}$ and $|\nn\theta_{\ell n}|^r\nn\theta_{\ell n}$  (the term $2\nu(\theta_{\ell n})|\Dv_{\ell n}|^2$ is easy due to \eqref{strov}).  For the~first one, we use \eqref{Akap}, \eqref{strt} and Vitali's theorem to get
		\begin{align}\label{kap}
			\sqrt{\kappa(\theta_{\ell n})}&\to\sqrt{\kappa(\theta_{\ell})}\quad\text{strongly in }L^{2+\f4r}(\Q;\R)
		\end{align}
		and then we combine this with \eqref{thw}, to obtain
		\begin{align}\label{wctok}
			\sqrt{\kappa(\theta_{\ell n})}\nn\theta_{\ell n}&\wc\sqrt{\kappa(\theta_{\ell})}\nn\theta_{\ell}\quad\text{weakly in }L^1(\Q;\R^d).
		\end{align}
		However, by \eqref{test} we know that \eqref{wctok} is valid also in $L^2(\Q;\R^d)$ up to a~subsequence, and hence, using again \eqref{kap}, we obtain
		\begin{equation}\label{kapf}
			\kappa(\theta_{\ell n})\nn\theta_{\ell n}=\sqrt{\kappa(\theta_{\ell n})}\sqrt{\kappa(\theta_{\ell n})}\nn\theta_{\ell n}\wc\sqrt{\kappa(\theta_{\ell})}\sqrt{\kappa(\theta_{\ell})}\nn\theta_{\ell}=\kappa(\theta_{\ell})\nn\theta_{\ell}
		\end{equation}
		weakly in $L^{\f{r+2}{r+1}}(\Q;\R^d)$.
		
		
		Finally, due to~\eqref{test}, there exists $K\in L^{(r+2)'}(\Q;\R^d)$ such that
		\begin{equation}\label{posuk0}
			|\nn\theta_{\ell n}|^r\nn\theta_{\ell n}\wc K\quad\text{weakly in }L^{(r+2)'}(\Q;\R^d).
		\end{equation}
		Then, using also \eqref{kapf} and previous convergence results, we can take the~limit in~\eqref{galt} and obtain, for all $\tau\in L^{r+2}(0,T;W^{N,2}(\Omega;\R))$, that
		\begin{equation}\label{tep0}
			\scal{c_v\partial_t\theta_{\ell},\tau}-(c_v\theta_{\ell}\ve_{\ell},\nn\tau)_{\Q}+(\kappa(\theta_{\ell})\nn\theta_{\ell},\nn\tau)_{\Q}+\omega(K,\nn\tau)_{\Q}=(\Sb_{\ell}^{\omega}\cdot\Dv_{\ell},\tau)_{\Q}.
		\end{equation}
		Recalling \eqref{tep}, \eqref{kapf} and \eqref{posuk0}, we easily conclude, using a~density argument, that \eqref{tep0} is valid for all $\tau\in L^{r+2}(0,T;W^{1,r+2}(\Omega;\R))$ and that the~time derivative extends to the~functional $\partial_t\theta_{\ell}\in L^{(r+2)'}(0,T;W^{-1,(r+2)'}(\Omega;\R))$. Thus, using Lemma~\ref{Temb} below, we also see that
		\begin{equation}\label{thC}
			\theta_{\ell}\in \CO([0,T];L^2(\Omega;\R)).
		\end{equation}
		Furthermore, choosing $\tau=\theta_{\ell}$ in~\eqref{tep0}, rewriting the~time derivative term and integrating by parts in the~convective term leads to
		\begin{equation}\label{neco1}
			\omega(K,\nn\theta_{\ell})_{\Q}=\f{c_v}2(\norm{\theta_{\ell}(0)}_2^2-\norm{\theta_{\ell}(T)}_2^2)-\int_{\Q}\kappa(\theta_{\ell})|\nn\theta_{\ell}|^2+(\Sb_{\ell}^{\omega}\cdot\Dv_{\ell},\theta_{\ell})_{\Q}.
		\end{equation}
		We use this information to identify $K$ as follows. We note that weak lower semi-continuity and \eqref{wctok} (which is valid in $L^2(\Q;\R^d)$) imply
		\begin{equation}\label{posuk}
			\int_{\Q}\kappa(\theta_{\ell})|\nn\theta_{\ell}|^2\leq \liminf_{n\to\infty}\int_{\Q}\kappa(\theta_{\ell n})|\nn\theta_{\ell n}|^2.
		\end{equation}
		Thus, if we integrate \eqref{potest} over $(0,T)$ and use \eqref{posuk}, \eqref{strov}, weak lower semi-continuity of $\norm{\cdot}_{2}$ and the~convergence results above to take the~limes superior $n\to\infty$ and then apply \eqref{neco1}, we get
		\begin{equation}\begin{aligned}\label{identif}
				&\omega\limsup_{n\to\infty}\int_{\Q}|\nn\theta_{\ell n}|^{r+2}\\
				&\quad=-\liminf_{n\to\infty}\f{c_v}2\norm{\theta_{\ell n}(T)}_2^2+\f{c_v}2\norm{\theta_0^{\omega}}_2^2-\liminf_{n\to\infty}\int_{\Q}\kappa(\theta_{\ell n})|\nn \theta_{\ell n}|^2\\
				&\quad\quad+\lim_{n\to\infty}(\Sb_{\ell n}^{\omega}\cdot\Dv_{\ell n},\theta_{\ell n})_{\Q}\\
				&\quad\leq-\f{c_v}2\norm{\theta_{\ell}(T)}_2^2+\f{c_v}2\norm{\theta_0^{\omega}}_2^2-\int_{\Q}\kappa(\theta_{\ell})|\nn \theta_{\ell}|^2+(\Sb_{\ell}^{\omega}\cdot\Dv_{\ell},\theta_{\ell})_{\Q}\\
				&\quad=\f{c_v}2\norm{\theta_0^{\omega}}_2^2-\f{c_v}2\norm{\theta_{\ell}(0)}_2^2+\omega(K,\nn\theta_{\ell})_{\Q}.
		\end{aligned}\end{equation}
		To identify the~initial condition for $\theta_{\ell}(0)$, it is enough to show that
		\begin{equation}\label{enough}
			\theta_{\ell}(t)\wc\theta_0^{\omega}\quad\text{weakly in }L^2(\Omega;\R)
		\end{equation}
		as $t\to0+$ since then we can use \eqref{thC} to conclude
		\begin{equation}\label{t0}
			\theta_{\ell}(0)=\theta_0^{\omega}\quad\text{a.e.\ in }\Omega
		\end{equation}
		by the~uniqueness of a~(weak) limit. To prove \eqref{enough}, we return to~\eqref{galt}, which we multiply by $\varphi\in W^{1,\infty}(0,T;\R)$ fulfilling $\varphi(0)=1$, $\varphi(T)=0$ and integrate the~result over $(0,T)$ to get
		\begin{equation}\begin{aligned}\label{alie}
				&-(c_v\theta_0^{\omega},w_k)-\int_0^T(c_v\theta_{\ell n},w_k)\partial_t\varphi=\int_0^Th_n\varphi.
		\end{aligned}\end{equation}
		for all $k=1,\ldots,n$, where we integrated by parts and abbreviated 
		\begin{equation*}
			h_n=(c_v\theta_{\ell n}\ve_{\ell n},\nn w_k)-(\kappa(\theta_{\ell n})\nn\theta_{\ell n}+\omega|\nn\theta_{\ell n}|^r\nn\theta_{\ell n},\nn w_k)+(\Sb_{\ell n}^{\omega}\cdot\Dv_{\ell n},w_k).
		\end{equation*}
		It follows from the~results above (cf.\ the~derivation of \eqref{tep0}) that
		\begin{equation*}
			h_n\wc h\quad\text{weakly in }L^{(r+2)'}(0,T;\R),
		\end{equation*}
		where
		\begin{equation*}
			\begin{aligned}
				h&=(c_v\theta_{\ell}\ve_{\ell},\nn w_k)-(\kappa(\theta_{\ell})\nn\theta_{\ell},\nn w_k)-\omega(K,\nn w_k)+(\Sb_{\ell}^{\omega}\cdot\Dv_{\ell},w_k).
			\end{aligned}
		\end{equation*}
		Thus, by taking the~limit $n\to\infty$ in~\eqref{alie}, we arrive at
		\begin{equation*}
			\begin{aligned}
				&-(c_v\theta_0^{\omega},w_k)-\int_0^T(c_v\theta_{\ell},w_k)\partial_t\varphi=\int_0^Th\varphi.
			\end{aligned}
		\end{equation*}
		Making now a~special choice
		\begin{equation*}
			\varphi_{\eps}(s)=\left\{\begin{matrix}1&s\leq t,\\1-\f{s-t}{\eps}&s\in(t,t+\eps),\\0&s\geq t+\eps,\end{matrix}\right.
		\end{equation*}
		where $t\in(0,T)$ and $0<\eps<T-t$, leads to
		\begin{equation*}
			\begin{aligned}
				&-(c_v\theta_0^{\omega},w_k)+\f1{\eps}\int_t^{t+\eps}(c_v\theta_{\ell},w_k)=\int_0^{t+\eps}h\varphi_{\eps}.
			\end{aligned}
		\end{equation*}
		Furthermore, we can take the~limit $\eps\to0+$ in this equation using \eqref{thC} on the~left hand side and absolute continuity of integral on the~right hand side to get
		\begin{equation*}
			-(c_v\theta_0^{\omega},w_k)+(c_v\theta_{\ell}(t),w_k)=\int_0^tf.
		\end{equation*}
		Finally, taking the~limit $t\to0+$ yields
		\begin{equation*}
			\lim_{t\to0+}(\theta_{\ell}(t),w_k)=(\theta_0^{\omega},w_k),
		\end{equation*}
		for all $k=1,\ldots,n$, from which \eqref{enough} follows by exploiting the~density of the~set $\spa\{w_k\}_{k=1}^{\infty}$ in $L^2(\Omega;\R)$. Hence, the~identity~\eqref{t0} is proved and \eqref{identif} hereby simplifies to
		\begin{equation}\label{limsup}
			\limsup_{n\to\infty}\int_{\Q}|\nn\theta_{\ell n}|^{r+2}\leq\int_{\Q}K\cdot\nn\theta_{\ell}.
		\end{equation}
		Since the~operator $\uu\mapsto |\uu|^r\uu$ is monotone and continuous, it is standard to show, using \eqref{limsup} and the~Minty method, that
		\begin{equation*}
			K=|\nn\theta_{\ell}|^r\nn\theta_{\ell}\quad \text{a.e.\ in }\Q.
		\end{equation*}
		Hence, we proved that
		\begin{equation}\label{sys1t}
			\scal{c_v\partial_t\theta_{\ell},\tau}-(c_v\theta_{\ell}\ve_{\ell}\,\nn\tau)_{\Q}+(\kappa(\theta_{\ell})\nn\theta_{\ell}+\omega|\nn\theta_{\ell}|^r\nn\theta_{\ell},\nn\tau)_{\Q}=(\Sb_{\ell}^{\omega}\cdot\Dv_{\ell},\tau)_Q
		\end{equation}
		for all $\tau\in L^{r+2}(0,T;W^{1,r+2}(\Omega;\R))$.
		
		\subsection*{Positive definiteness of \texorpdfstring{$\B_{\ell}$}{B} and positivity of \texorpdfstring{$\theta_{\ell}$}{t}}
		
		Here we follow the~method developed in~\cite{Bathory_2020}.  We shall use the~notation
		\begin{equation*}
			h_+=\max\{0,h\},\qquad h_-=\min\{0,h\}.
		\end{equation*}
		We choose a~fixed vector $\x\in\R^d$ with $|\x|=1$, and $t\in(0,T)$. The~idea is to use
		\begin{equation*}
			\A_{\x}=\chi_{(0,t)}(b-\omega)_-\,\x\otimes\x,\quad\text{where}\quad b\coloneqq \B_{\ell}\x\cdot\x.
		\end{equation*}
		in \eqref{sys1B}. The~function $\A_{\x}$ belongs to $L^2(0,T;W^{1,2}(\Omega;\Sym))\cap L^{q+2}(\Q;\Sym)$ and is thus a~valid test function in~\eqref{sys1B}. The~key property of $\A_{\x}$ is that it vanishes whenever the~smallest eigenvalue of $\B_{\ell}$ is greater than $\omega$. Thus, we have
		\begin{equation*}
			(\Lambda(\B_{\ell})-\omega)_+(b-\omega)_-=0,
		\end{equation*}
		which implies
		\begin{equation}\label{vanish}
			g_{\omega}(\B_{\ell},\theta_{\ell})\A_{\x}=0\quad\text{a.e. in }\Q.
		\end{equation}
		Let us now evaluate separately the~terms arising from the~choice $\A=\A_{\x}$ in \eqref{sys1B}.  For the~time derivative, we write
		\begin{equation}\label{op}
			\scal{\partial_t\B_{\ell},\A_{\x}}=\int_0^t\scal{\partial_t(b-\omega),(b-\omega)_-}=\tfrac12\norm{(b-\omega)_-(t)}_2^2,
		\end{equation}
		where we applied Lemma~\ref{Ldual} below for the~Lipschitz function $s\mapsto s_-$ and also \eqref{BB} and \eqref{0zesp}  to eliminate the~value at $t=0$. 
		Furthermore, using integration by parts, $\ve_{\ell}\cdot\n=0$ and $\di\ve_{\ell}=0$, we get
		\begin{equation*}	(\B_{\ell}\otimes\ve_{\ell},\nn\A_{\x})_{\Q}=\int_0^t((b-\omega)\ve_{\ell},\nn(b-\omega)_-)=\f12\int_0^t\int_{\partial\Omega}((b-\omega)_-)^2\ve_{\ell}\cdot\n=0
		\end{equation*}
		and also
		\begin{align*}
			(\la(\theta_{\ell})\nn\B_{\ell},\nn\A_{\x})_{\Q}=\int_0^t\norm{\sqrt{\la(\theta_{\ell})}\nn(b-\omega)_-}_2^2\geq 0.
		\end{align*}
		Moreover, we have $b-\omega_P<b-\omega$ and thus, the~assumption \eqref{APpd} yields
		\begin{align*}
			(\PP(\theta_{\ell},\B_{\ell}),\A_{\x})_{\Q}&=\int_0^t\ii(b-\omega)_-\,\PP(\theta_{\ell},\B_{\ell})\x\cdot\x\\
			&=\int_0^t\int\limits_{\{b<\omega\}}(b-\omega)\PP(\theta_{\ell},(\B_{\ell}-\omega_P\I)+\omega_P\I)\x\cdot\x\geq0.
		\end{align*}
		In addition, the~right hand side of \eqref{sys1B} vanishes due to~\eqref{vanish}.
		Thus, using the~above computation in~\eqref{sys1B}, we obtain
		\begin{equation*}
			\norm{(b-\omega)_-(t)}_2^2\leq0
		\end{equation*}
		for all $t\in(0,T)$ (recall \eqref{BCo}), whence
		\begin{equation}\label{pd}
			\B_{\ell}(t)\x\cdot\x\geq\omega|\x|^2\quad\text{a.e.\ in }\Omega,\text{ for all }t\in(0,T)\text{ and for every }\x\in\R^d.
		\end{equation}
		Note that this immediately yields $\B_{\ell}\in\PD$, $\B_{\ell}^{-1}\in\PD$ a.e.\ in $\Q$, and thus
		\begin{equation}\label{Binvom}
			|\B_{\ell}^{-1}|=|\B_{\ell}^{-\f12}\B_{\ell}^{-\f12}|\leq|\B_{\ell}^{-\f12}|^2=\tr\B_{\ell}^{-1}\leq\f d{\omega}.
		\end{equation}
		Also, using the~identity
		\begin{equation*}
			\nn\B_{\ell}^{-1}=-\B_{\ell}^{-1}\nn\B_{\ell}\B_{\ell}^{-1},
		\end{equation*}
		(which is standard for continuously differentiable functions and in general we can approximate $\B_{\ell}$ by smooth mappings and pass to the~limit) and \eqref{best} we conclude that $\B_{\ell}^{-1}$ exists a.e.\ in $\Q$ and satisfies
		\begin{equation}\label{Binv}
			\B_{\ell}^{-1}\in L^{\infty}(0,T;L^{\infty}(\Omega;\PD))\cap L^2(0,T;W^{1,2}(\Omega;\PD)).
		\end{equation}
		Moreover, recalling $f$ from \eqref{helm} and using the~simple inequalities
		\begin{equation*}
			\det\B_{\ell}\geq\omega^d\quad\text{and}\quad |\ln x|\leq x+\f1x,\;x>0,
		\end{equation*}
		it is easy to see that also
		\begin{equation*}
			f(\B_{\ell})\in L^2(0,T;W^{1,2}(\Omega;\R_{\geq0}))\cap L^{q+2}(\Q;\R_{\geq0}).
		\end{equation*}
		
		Next, we prove positivity of $\theta_{\ell}$. Since $\theta_{\ell}\in L^{r+2}(0,T;W^{1,r+2}(\Omega;\R))$, we can use the~analogous method as before. Indeed, we start by choosing
		\begin{equation*}
			\tau=\chi_{(0,t)}(\theta_{\ell}-\omega)_-\in L^{r+2}(0,T;W^{1,r+2}(\Omega;\R))
		\end{equation*}
		as a~test function in~\eqref{sys1t} to get (using $\di \ve_{\ell} =0$)
		\begin{align}\begin{aligned}
				&\f{c_v}2\norm{(\theta_{\ell}-\omega)_-(t)}_2^2-\f{c_v}2\norm{(\theta_{\ell}-\omega)_-(0)}_2^2\\
				&\quad+\int_0^t\norm{\sqrt{\kappa(\theta_{\ell})}\nn(\theta_{\ell}-\omega)_-}_2^2+\int_0^t\norm{\nn(\theta_{\ell}-\omega)_-}_{r+2}^{r+2}\\
				&\qquad=\int_0^t\big(\Sb_{\ell}^{\omega}\cdot\Dv_{\ell},(\theta_{\ell}-\omega)_-\big)\leq0.\label{ofh}\end{aligned}
		\end{align}
		Hence, using $\theta_{\ell}(0)=\theta_0^{\omega}\geq\omega$ in $\Omega$ and \eqref{thC}, we obtain that $\norm{(\theta_{\ell}(t)-\omega)_-}_2=0$ for all $t\in(0,T)$,
		which means
		\begin{equation}\label{pt}
			\theta_{\ell}(t)\geq\omega\quad\text{a.e.\ in }\Omega\text{ and for all }t\in(0,T).
		\end{equation}
		Consequently, since $\nn\theta_{\ell}^{-1}=\theta^{-2}_{\ell}\nn\theta_{\ell}$, we also obtain
		\begin{equation}\label{tinv}
			\theta_{\ell}^{-1}\in L^{\infty}(0,T;L^{\infty}(\Omega;\R_{>0}))\cap L^{r+2}(0,T;W^{1,r+2}(\Omega;\R_{>0})).
		\end{equation}
		From these findings we also easily read that
		\begin{equation*}
			|\ln\theta_{\ell}|\leq\theta_{\ell}+\f1{\theta_{\ell}}\leq \theta_{\ell}+\f1{\omega}\quad\text{and}\quad|\nn\ln\theta_{\ell}|=\f{|\nn\theta_{\ell}|}{\theta_{\ell}}\leq\f1{\omega}|\nn\theta_{\ell}|,
		\end{equation*}
		hence also
		\begin{equation*}
			\ln\theta_{\ell}\in L^{r+2}(0,T;W^{1,r+2}(\Omega;\R)).
		\end{equation*}
		
		\subsection*{Entropy equation}
		In order to take the~remaining limits $\ell\to\infty$ and $\omega\to0+$, we need to derive the~entropy (in)equality from which  we then deduce that $\det\B_{\ell}$ and $\theta_{\ell}$ remain {\it strictly} positive a.e.\ in $\Q$. First, we rewrite \eqref{sys1t} in the~form
		\begin{equation}\label{sys1t2}
			\scal{c_v\partial_t\theta_{\ell},\tau}+(c_v\ve_{\ell}\cdot\nn\theta_{\ell},\tau)+(\kappa(\theta_{\ell})\nn\theta_{\ell}+\omega|\nn\theta_{\ell}|^r\nn\theta_{\ell},\nn\tau)=(\Sb_{\ell}^{\omega}\cdot\Dv_{\ell},\tau)
		\end{equation}
		for all $\tau\in W^{1,r+2}(\Omega;\R)$ and a.e.\ in $(0,T)$.
		Then, we take $\phi\in W^{1,\infty}(\Omega;\R)$ and note that $\tau=\theta_{\ell}^{-1}\phi$ can be used as a~test function in~\eqref{sys1t2} thanks to~\eqref{tinv}. This way, we get
		\begin{align}\begin{aligned}\label{1test}
				&\langle c_v\partial_t\theta_{\ell},\theta_{\ell}^{-1}\phi\rangle+(c_v\ve_{\ell}\cdot\nn\ln\theta_{\ell},\phi)+(\kappa(\theta_{\ell})\nn\ln\theta_{\ell},\nn\phi)-(\kappa(\theta_{\ell})|\nn\ln\theta_{\ell}|^2,\phi)\\
				&\quad+\omega(|\nn\theta_{\ell}|^r\nn\ln\theta_{\ell},\nn\phi)-\omega(|\nn\theta_{\ell}|^r|\nn\ln\theta_{\ell}|^2,\phi)\\
				&\qquad=(2\nu(\theta_{\ell})\theta_{\ell}^{-1}|\Dv_{\ell}|^2+2a\mu g_{\omega}(\B_{\ell},\theta_{\ell})\B_{\ell}\cdot\Dv_{\ell},\phi)\end{aligned}
		\end{align}
		a.e.\ in $(0,T)$. Similarly, we observe that $f'(\B_{\ell})\phi=\mu(\I-\B_{\ell}^{-1})\phi$  (recall \eqref{calc1}, \eqref{calc2})  is a~valid test function in \eqref{sys1B} due to~\eqref{Binv}. Thus, we obtain
		\begin{equation}\begin{aligned}\label{2test}
				&\scal{\partial_t\B_{\ell},f'(\B_{\ell})\phi}+(\ve_{\ell}\cdot\nn f(\B_{\ell}),\phi)\\[-0.1cm]
				&\quad+(\mu\PP(\theta_{\ell},\B_{\ell})\cdot(\I-\B_{\ell}^{-1}),\phi)+(\mu\la(\theta_{\ell})|\B_{\ell}^{-\f12}\nn\B_{\ell}\B_{\ell}^{-\f12}|^2,\phi)\hskip0.2cm\\
				&\quad\quad=-(\la(\theta_{\ell})\nn f(\B_{\ell}),\nn\phi)+(2a\mu g_{\omega}(\B_{\ell},\theta_{\ell})\B_{\ell}\cdot\Dv_{\ell},\phi)
		\end{aligned}\end{equation}
		a.e.\ in $(0,T)$. If we define
		\begin{equation}\label{etadef}
			\eta_{\ell}\coloneqq c_v\ln\theta_{\ell}- f(\B_{\ell})
		\end{equation}
		and
		\begin{align}\begin{aligned}\label{xil}
				\xi_{\ell}&\coloneqq2\nu(\theta_{\ell})\theta_{\ell}^{-1}|\Dv_{\ell}|^2+\kappa(\theta_{\ell})|\nn\ln\theta_{\ell}|^2+\omega|\nn\theta_{\ell}|^r|\nn\ln\theta_{\ell}|^2\\
				&\qquad\qquad\qquad+\mu \PP(\theta_{\ell},\B_{\ell})\cdot(\I-\B_{\ell}^{-1})+\mu\la(\theta_{\ell})|\B_{\ell}^{-\f12}\nn\B_{\ell}\B^{-\f12}_{\ell}|^2
		\end{aligned}\end{align}
		and subtract \eqref{2test} from \eqref{1test}, we get
		\begin{equation}\begin{aligned}\label{enteq}
				&\langle c_v\partial_t\theta_{\ell},\theta_{\ell}^{-1}\phi\rangle-\scal{\partial_t\B_{\ell},f'(\B_{\ell})\phi}+(\ve_{\ell}\cdot\nn\eta_{\ell},\phi)\\
				&\quad+\big((\kappa(\theta_{\ell})+\omega|\nn\theta_{\ell}|^r)\nn\ln\theta_{\ell}-\la(\theta_{\ell})\nn f(\B_{\ell}),\nn\phi\big)=(\xi_{\ell},\phi)\end{aligned}
		\end{equation}
		a.e.\ in $(0,T)$ and for all $\phi\in W^{1,\infty}(\Omega;\R)$. It remains to rewrite the~time derivative accordingly. Concerning the~term containing $\partial_t\theta_{\ell}$, note that $\psi(s)=\max\{|s|,\omega\}^{-1}$, $s\in\R$, is a~bounded Lipschitz function. Since $\theta_{\ell}\geq\omega$ a.e.\ in $\Q$ by \eqref{pt}, we get
		\begin{equation*}
			\int_1^{\theta_{\ell}}\psi(s)\dd{s}=\int_1^{\theta_{\ell}}\f1s\dd{s}=\ln\theta_{\ell}.
		\end{equation*}
		Thus, Lemma~\ref{Ldual} below yields
		\begin{equation*}
			\langle c_v\partial_t\theta_{\ell},\theta_{\ell}^{-1}\phi\rangle=\f{\dd{}}{\dd{t}}(c_v\ln\theta_{\ell},\phi).
		\end{equation*}
		If we multiply this by $\varphi\in W^{1,\infty}((0,T);\R)$ with $\varphi(T)=0$, integrate over $(0,T)$ and by parts, we are led to
		\begin{equation}\label{tide}
			\langle c_v\partial_t\theta_{\ell},\theta_{\ell}^{-1}\phi\varphi\rangle=-( c_v\ln\theta_{\ell},\phi\partial_t\varphi)_{\Q}-(c_v\ln\theta_{0}^{\omega},\phi\varphi(0)),
		\end{equation}
		where we also used \eqref{t0}. Analogous ideas can be used to rewrite the~second term of \eqref{enteq}. However, since the~duality $\scal{\partial_t\B_{\ell},f'(\B_{\ell})\phi}$ cannot be interpreted entry-wise, let us proceed more carefully. We apply Lemma~\ref{Temb} below to obtain functions $\B_{\ell}^{\eps}\in \CO^1([0,T];W^{1,2}(\Omega;\PD)\cap L^{q+2}(\Omega;\PD))$, $\eps>0$, such that
		\begin{equation}\label{strr}
			\norm{\B_{\ell}^{\eps}-\B_{\ell}}_{L^2W^{1,2}\cap L^{q+2}L^{q+2}}+\norm{\partial_t\B_{\ell}^{\eps}-\partial_t\B_{\ell}}_{(L^2W^{1,2}\cap L^{q+2}L^{q+2})^*}\to0
		\end{equation}
		as $\eps\to0+$ and also $\Lambda(\B_{\ell}^{\eps})\geq\omega$ a.e.\ in $\Q$.  For such regularization, we have
		\begin{equation}\label{prel}
			\scal{\partial_t\B_{\ell}^{\eps},f'(\B_{\ell}^{\eps})\phi\varphi}
			=-(f(\B_{\ell}^{\eps}(0)),\phi\varphi(0))-(f(\B_{\ell}^{\eps}),\phi\partial_t\varphi)_{\Q}
		\end{equation}
		by the~standard calculus and it remains to justify the~limit $\eps\to0+$ on both sides of \eqref{strr}.  Since $\B_{\ell}\in\CO([0,T];L^2(\Omega))$ (cf. \eqref{BCo}), we know that
		\begin{equation}\label{uniB}
			\norm{\B_{\ell}^{\eps}-\B_{\ell}}_2\rightrightarrows0\quad\text{uniformly in }[0,T].
		\end{equation}
		Now it is important to observe that since we have $\Lambda(\B_s)\geq\omega$ for all $s\in[0,1]$, where
		\begin{equation*}
			\B_s\coloneqq(1-s)\B_{\ell}+s\B_{\ell}^{\eps},
		\end{equation*}
		the convergence \eqref{uniB} actually also implies
		\begin{equation}\label{Uni}
			\norm{f(\B_{\ell}^{\eps})-f(\B_{\ell})}_2+\norm{(\B_{\ell}^{\eps})^{-1}-\B_{\ell}^{-1}}_2\rightrightarrows0\quad\text{uniformly in }[0,T].
		\end{equation}
		Indeed, this is a~simple consequence of the~identities
		\begin{align}
			f(\B_{\ell}^{\eps})- f(\B_{\ell})&=\int_0^1\f{\dd{}}{\dd{s}}f(\B_s)\dd{s}=\int_0^1\mu(\I-\B_s^{-1})\cdot(\B_{\ell}^{\eps}-\B_{\ell})\dd{s},\label{comp1}\\
			(\B_{\ell}^{\eps})^{-1}- \B_{\ell}^{-1}&=\int_0^1\f{\dd{}}{\dd{s}}\B_s^{-1}\dd{s}=-\int_0^1\B_s^{-1}(\B_{\ell}^{\eps}-\B_{\ell})\B_s^{-1}\dd{s},\nonumber
		\end{align}
		\eqref{uniB} and the~estimate 
		\begin{equation*}
			|\B_s^{-1}|\leq\tr\B_s^{-1}\leq  \f{d}{\Lambda(\B_s)}\leq\f{d}{\omega}. 
		\end{equation*}
		Using the~same scheme as in \eqref{comp1}, we also deduce from \eqref{BCo} and \eqref{BB} that 
		\begin{equation}\label{psi00}
			f(\B_{\ell})\in\CO(0,T;L^2(\Omega;\R)),\quad  f(\B_{\ell}(0))= f(\B_0^{\omega}).
		\end{equation}
		This and $\eqref{Uni}_1$ allow us to pass to the~desired limit on the~right-hand side of \eqref{prel}.  Next, using \eqref{Binvom}, we can estimate, for any $\phi\in W^{1,\infty}(\Omega;\R)$, that
		\begin{align*}
			|\nn(f'(\B_{\ell}^{\eps})\phi)|&=|(\B_{\ell}^{\eps})^{-1}\nn\B_{\ell}^{\eps}(\B_{\ell}^{\eps})^{-1}\phi+(\I-(\B_{\ell}^{\eps})^{-1})\nn\phi|\\&\leq C\omega^{-2}|\nn\B_{\ell}^{\eps}||\phi|+(1+C\omega^{-1})|\nn\phi|.
		\end{align*}
		Using the~second line of this estimate to show boundedness and the~first line to identify the~weak $\eps$-limit using $\eqref{Uni}_2$ and \eqref{strr},  we eventually obtain
		\begin{equation*}
			f'(\B_{\ell}^{\eps})\phi\wc f'(\B_{\ell})\phi\text{ weakly in }L^2(0,T;W^{1,2}(\Omega;\Sym))\cap L^{q+2}(\Q;\Sym).
		\end{equation*}
		If we apply this with \eqref{strr}, we get, for all $\varphi\in W^{1,\infty}((0,T);\R)$, $\varphi(T)=0$, that
		\begin{align*}
			&\left|\scal{\partial_t\B_{\ell}^{\eps},f'(\B_{\ell}^{\eps})\phi\varphi}-\scal{\partial_t\B_{\ell},f'(\B_{\ell})\phi\varphi}\right|\\
			&\qquad\leq\left|\scal{\partial_t\B_{\ell}^{\eps}-\partial_t\B_{\ell},f'(\B_{\ell}^{\eps})\phi\varphi}\right|+\left|\scal{\partial_t\B_{\ell}\varphi,f'(\B_{\ell}^{\eps})\phi-f'(\B_{\ell})\phi}\right|\to0
		\end{align*}
		as $\eps\to0+$.  This validates the~limit on the~left-hand side of \eqref{prel}, and thus 
		\begin{equation}\label{BIde}
			\scal{\partial_t\B_{\ell},f'(\B_{\ell})\phi\varphi}=-(f(\B_0^{\omega}),\phi\varphi(0))-(f(\B_{\ell}),\phi\partial_t\varphi)_{\Q}
		\end{equation}
		for all $\varphi\in W^{1,\infty}(\Omega;\R)$, $\varphi(T)=0$, and every $\phi\in W^{1,\infty}(\Omega;\R)$.  Therefore, after application of \eqref{tide} and \eqref{BIde}, entropy equation~\eqref{enteq} becomes 				
		\begin{align}\begin{aligned}\label{ENTl}
				&-(\eta_{\ell},\phi\partial_t\varphi)_{\Q}-(\eta_0^{\omega},\phi)\varphi(0)-(\ve_{\ell}\eta_{\ell},\nn\phi\varphi)_{\Q}\\
				&\quad+\big((\kappa(\theta_{\ell})+\omega|\nn\theta_{\ell}|^r)\nn\ln\theta_{\ell}-\la(\theta_{\ell})\nn f(\B_{\ell}),\nn\phi\varphi\big)_{\Q}=(\xi_{\ell},\phi\varphi)_{\Q}
		\end{aligned}\end{align}
		for all $\varphi\in W^{1,\infty}(0,T;\R)$, $\varphi(T)=0$, and $\phi\in W^{1,\infty}(\Omega;\R)$, where
		\begin{equation*}
			\eta_0^{\omega}\coloneqq c_v\ln\theta_0^{\omega}- f(\B_0^{\omega}).
		\end{equation*}
		Moreover, since $\ln\theta_{\ell}\in\CO([0,T];L^2(\Omega;\R))$ and \eqref{psi00} hold, we easily read
		\begin{equation}\label{etaC}
			\eta_{\ell}\in \CO([0,T];L^2(\Omega;\R)),\quad \eta_{\ell}(0)=\eta_0^{\omega}.
		\end{equation}
		
		\subsection*{Total energy equality}
		
		The integrated version of the~total energy equality is important in the~derivation of the~a~priori estimates below. We multiply the~$i$-th equation in~\eqref{sys1v} by $(\ve_{\ell},\we_i)$, sum up the~result over $i=1,\ldots,\ell$ and then we add \eqref{sys1t} with $\tau=1$. This way, after several cancellations using also $\eqref{s12bc}_1$, we obtain
		\begin{equation}\label{Ener0}
			\f{\dd{}}{\dd{t}}\ii E_{\ell}+\alpha \int_{\partial \Omega}|\ve_{\ell}|^2=(\fe,\ve_{\ell})\quad\text{a.e.\ in }(0,T),
		\end{equation}
		where $E_{\ell}\coloneqq\f12|\ve_{\ell}|^2+c_v\theta_{\ell}$.
		
		\section{Existence of a~weak solution: limits \texorpdfstring{$\omega\to0$}{w}, \texorpdfstring{$\ell\to\infty$}{l}}\label{SS5}		
		
		This is the~most essential part of the~paper. Here, we first rigorously derive the~estimates independent of $\omega$ and $\ell$ and then let $\omega \to 0+$ and $\ell \to \infty$  (in fact, we take these two limits simultaneously by setting $\omega=\f1{\ell}$).  Due to the~linearity of the~leading differential operators, the~limit passage is then relatively straightforward. On the~other hand, to obtain the~attainment of the~initial condition in the~strong topology, we need to develop a~new technique based on the~combination of the~entropy inequality and the~global energy inequality.

		\subsection*{Estimates independent of \texorpdfstring{$\ell,\omega$}{l} based on global energy and entropy}
		
		
		Let us first show that the~total energy of the~fluid remains bounded. In \eqref{Ener0}, we apply Young's inequality, \eqref{init} and $\theta_{\ell}>0$, to estimate
		\begin{equation*}
			\f{\dd{}}{\dd{t}}\ii E_{\ell}\leq \f12\ii|\ve_{\ell}|^2+ \f12\ii|\fe|^2\leq \ii E_{\ell}+\f12\ii|\fe|^2
		\end{equation*}
		a.e.\ in $(0,T)$. Hence, by the~Gronwall inequality, we get
		\begin{equation*}
			\ii E_{\ell}(t)\leq e^t\left(\ii E_{\ell}(0)+\f12\int_0^t\norm{\fe}_2^2\right)\quad\text{for all }t\in[0,T].
		\end{equation*}
		Then, we apply \eqref{vv}, \eqref{t0} to identify that
		\begin{equation*}
			E_{\ell}(0)=\f12|P_{\ell}\ve_0|^2+c_v\theta_0^{\omega}
		\end{equation*}
		and if we use properties of $P_{\ell}$, \eqref{entOO} and \eqref{init}, we arrive at
		\begin{equation}\label{th}
			\norm{\theta_{\ell}}_{L^{\infty}L^1}+\norm{\ve_{\ell}}_{L^{\infty}L^2}\leq C\norm{E_{\ell}}_{L^{\infty}L^1}\leq C.
		\end{equation}
		
		Now we turn our attention to~\eqref{ENTl}, which we localize in time by choosing\footnote{Strictly speaking, as $\chi_{(0,t)}$ is not Lipschitz, we can not use it directly in \eqref{ENTl}. However, a~standard argument using a~piecewise linear approximation of $\chi_{(0,t)}$ with the~Lebesgue differentiation theorem and absolute continuity of integral shows that $\chi_{(0,t)}$ is a~reasonable test function.} $\varphi=\chi_{(0,t)}$, leading to
		\begin{equation}\label{i}
			\ii\eta_{\ell}(t)\phi+\int_0^{t}\ii\je_{\ell}\cdot\nn\phi= \ii\eta^{\omega}_0\phi+\int_0^t\ii\xi_{\ell}\phi\quad\text{for all }\phi\in W^{1,\infty}(\Omega;\R)
		\end{equation}
		and all $t\in(0,T)$ (in fact, for all $t\in[0,T]$ due to continuity), where
		\begin{equation*}
			\je_{\ell}\coloneqq-\ve_{\ell}\eta_{\ell}+(\kappa(\theta_{\ell})+\omega|\nn\theta_{\ell}|^r)\nn\ln\theta_{\ell}-\la(\theta_{\ell})\nn f(\B_{\ell})\in L^1(\Q;\R^d).
		\end{equation*}
		In particular, taking $\phi=1$, we deduce, using $\xi_{\ell}\geq0$, that the~function $t\mapsto\ii\eta_{\ell}(t)$ is non-decreasing, and thus
		\begin{equation}\label{gg}
			\int_{\Q}\xi_{\ell}=\max_{t\in[0,T]}\int_0^t\ii\xi_{\ell}=\max_{t\in[0,T]}\ii\eta_{\ell}(t)-\ii\eta_0^{\omega}=\ii\eta_{\ell}(T)-\ii\eta_0^{\omega}.
		\end{equation}
		Then, using \eqref{etadef}, the~inequalities
		\begin{equation}\label{lnx}
			\ln x\leq x-1\quad \text{for all }x>0\quad\text{and}\quad
			f(\B_{\ell})\geq0,
		\end{equation}
		assumption \eqref{initln} and \eqref{th} (recall also \eqref{thC}), we obtain
		\begin{align}\begin{aligned}\label{upe}
				\int_{\Q}\xi_{\ell}&\leq\ii(c_v\ln\theta_{\ell}(T)- f(\B_{\ell}(T)))+C
				\leq C\ii\left(\theta_{\ell}(T)-1\right)+C\leq C,
		\end{aligned}\end{align}
		hence
		\begin{equation}\label{xiest}
			\norm{\xi_{\ell}}_{L^1L^1}\leq C.
		\end{equation}
		Also, it is easy to see using \eqref{init}, \eqref{initln}, \eqref{entOO}, \eqref{entO} and \eqref{gg} that
		\begin{equation}\label{etaint}
			\norm{\eta_{\ell}}_{L^{\infty}L^1}\leq C.
		\end{equation}
		Estimate \eqref{xiest} implies, using \eqref{Anu} and \eqref{APp}, that
		\begin{align}
			\begin{aligned}\label{EST}			\norm{\theta_{\ell}^{-\f12}\Dv_{\ell}}_{L^2L^2}+\norm{\sqrt{\kappa(\theta_{\ell})}\nn\ln\theta_{\ell}}_{L^2L^2}&+\sqrt{\omega}\norm{|\nn\theta_{\ell}|^{\f r2}\nn\ln\theta_{\ell}}_{L^2L^2}\\
				&+\norm{\B_{\ell}^{-\f12}\nn\B_{\ell}\B_{\ell}^{-\f12}}_{L^2L^2}\leq C.
			\end{aligned}
		\end{align}
		
		\subsection*{Improved \texorpdfstring{$\ell,\omega$}{l} estimates}
	
	In what follows, we improve the~uniform estimate \eqref{EST} considerably by choosing appropriate test functions in~\eqref{sys1B} and \eqref{sys1t} and then using \eqref{A0}. In fact, we repeat the~scheme of estimates presented in \eqref{Best}--\eqref{keys4}, but now, we prove it fully rigorously.
	
	Our aim is to set $\A:=\B_{\ell}^{q-1}$ in~\eqref{sys1B}. To verify that this is a~valid test function, we show first that $\B_{\ell}$ is actually essentially bounded. Indeed, setting first $\A=\chi_{(0,t)}\phi\I$, $t\in(0,T)$, $\phi\in L^{q+2}(0,T;L^{q+2}(\Omega;\R))\cap L^2(0,T;W^{1,2}(\Omega;\R))$, in \eqref{sys1B} yields
	\begin{align*}		&\int_0^t\scal{\partial_t\tr\B_{\ell},\phi}+\int_0^t(\ve\cdot\nn\tr\B_{\ell},\phi)+\int_0^t(\PP(\theta_{\ell},\B_{\ell})\cdot\I,\phi)+\int_0^t(\la(\theta_{\ell})\nn\tr\B_{\ell},\nn\phi)\\
		&\qquad=\int_0^t(2ag_{\omega}(\B_{\ell},\theta_{\ell})\B_{\ell}\cdot\Dv_{\ell},\phi).
	\end{align*}
	Hence, recalling \eqref{APcoer2} to bound the~third term on the~left hand side and using \eqref{godhad} and \eqref{strov} to estimate the~right hand side, we see that  there exists a~constant $C(\ell,\omega)>0$, such that
	\begin{align*}
		&\int_0^t\scal{\partial_t\tr\B_{\ell},\phi}+\int_0^t(\ve\cdot\nn\tr\B_{\ell},\phi)+\int_0^t(\la(\theta_{\ell})\nn\tr\B_{\ell},\nn\phi)\leq C(\ell,\omega)\int_0^t\int_{\Omega}|\phi|.
	\end{align*}
	Substituting $u(\x,t)\coloneqq\tr\B_{\ell}(\x,t)-C(\ell,\omega)t$ leads to
	\begin{align*}
		\int_0^t\scal{\partial_tu,\phi}+\int_0^t(\ve\cdot\nn u,\phi)+\int_0^t(\la(\theta_{\ell})\nn u,\nn\phi)\leq C(\ell,\omega)\int_0^t\int_{\Omega}(|\phi|-\phi).
	\end{align*}
	If we choose $\phi=(u-K)_+$ and use \eqref{s6d}, $\eqref{s12bc}_1$ to eliminate the~convective term, we obtain
	\begin{align*}
		&\frac12\norm{(u(t)-K)_+}_2^2+\int_0^t\norm{\sqrt{\la(\theta_{\ell})}\nn(u-K)_+}_2^2\leq\frac12\norm{(u(0)-K)_+}_2^2.
	\end{align*}
	If we let $K\coloneqq\f{d}{\omega}$, then $\eqref{BB}$ and \eqref{0sh} imply
	\begin{equation*}
		0\le (u(0)-K)_+=(\tr\B_0^{\omega}-\tfrac{d}{\omega})_+\leq(\sqrt{d}|\B_0^{\omega}|-\tfrac d{\omega})_+=0 \quad \textrm{ in } \Omega.
	\end{equation*}
	Thus, we get $\norm{(u(t)-\tfrac{d}{\omega})_+}_2^2=0$, hence
	\begin{equation*}
		|\B_{\ell}|\leq\tr\B_{\ell}\leq\tfrac{d}{\omega}+C_0t\leq\tfrac{d}{\omega}+C_0T
	\end{equation*}
	and we see that indeed
	\begin{equation}\label{Bsigg}
		\B_{\ell}\in L^{\infty}(0,T;L^{\infty}(\Omega;\PD))\cap L^2(0,T;W^{1,2}(\Omega;\PD)).
	\end{equation}
	Due to the~fact that $\B_{\ell}$ is strictly positive definite, we can use the~above property to show that the~same holds also for $\B_{\ell}^{q-1}$, which is essential for showing that $\A:=\B_{\ell}^{q-1}$ can be used in \eqref{sys1B} as a~test function. Indeed, the~boundedness of $\B_{\ell}^{q-1}$ is a~direct  consequence of \eqref{Bsigg} and the~spectral decomposition. To show that gradient of $\B_{\ell}^{q-1}$ is square integrable, we recall the~identity
	\begin{equation*}	
		\begin{aligned}
			&\frac{\nn\B_{\ell}^{q-1}}{q-1}= \int_0^1\!\!\int_0^1\B_{\ell}^{(1-s)(q-1)}((1\!-\!t)\I+t\B_{\ell})^{-1}\nn\B_{\ell}((1\!-\!t)\I+t\B_{\ell})^{-1}\B_{\ell}^{s(q-1)}\dd{s}\dd{t},
		\end{aligned}
	\end{equation*}
	which is a~consequence of the~well known identities for $\nn\exp\A$ and $\nn\log\A$, see e.g. \cite{Wilcox_1967,Wouk_1965,BathoryOdh} and references therein for details. Then  using also \eqref{pd} to estimate 
	\begin{equation*}
		|((1-t)\I+t\B_{\ell})^{-1}|\leq \f{\sqrt{d}}{\Lambda((1-t)\I+t\B_{\ell})}\leq\f{\sqrt{d}}{\omega}
	\end{equation*}
	and also \eqref{Bsigg}, we see that  $\nn\B_{\ell}^{q-1}\in L^2(0,T;L^2(\R\times\Sym))$ and consequently
	\begin{equation*}
		\B_{\ell}^{q-1}\in L^{\infty}(0,T;L^{\infty}(\Omega;\PD))\cap L^2(0,T;W^{1,2}(\Omega;\PD)).
	\end{equation*}
	Hence, setting $\A:=\chi_{[0,t]}\B_{\ell}^{q-1}$ in \eqref{sys1B}, using \eqref{APcoer}
	and the~identities\footnote{To interpret the~duality pairing in the~first identity, one has to approximate $\B_{\ell}$ similarly as before when dealing with $\scal{\partial_t\B_{\ell},(\I-\B^{-1}_{\ell})\phi}$.}
	\begin{align*}
		\scal{\partial_t\B_{\ell},\B^{q-1}_{\ell}}&=\f1{q}\ii\partial_t\tr\B_{\ell}^{q},\\
		(\ve\cdot\nn\B_{\ell},\B_{\ell}^{q-1})&=\f1{q}\ii\ve\cdot\nn\tr\B_{\ell}^{q}=0\\
		\intertext{and the~estimate (see (iv) and (v) in Lemma~\ref{PDalg} below)}
		\nn\B_{\ell}\cdot\nn\B_{\ell}^{q-1}& \geq\f{4(q-1)}{q^2}|\nn\B_{\ell}^{\f{q}2}|^2,
	\end{align*}
	we get
	\begin{align*}		
		&\frac1{q}\ii(\tr\B_{\ell}^{q}(t)-\tr\B_{\ell}^{q}(0))+C_{q-1}\int_0^t\ii|\B|^{2q}
		+\frac{4(q-1)}{q^2}\int_0^t\ii\lambda(\theta_{\ell})|\nn\B_{\ell}^{\frac{q}2}|^2\\
		&\quad\le 2a\int_0^t\ii g(\B_{\ell},\theta_{\ell})\Dv_{\ell}\cdot\B_{\ell}^{q}+ C.
	\end{align*}
	If we apply \eqref{BB}, \eqref{Ala}, $g_{\omega}\leq1$ and $|\B_{\ell}^{q}|\leq \max\{1,d^{\frac{1-q}2}\}|\B_{\ell}|^{q}$ (see \cite{BathoryOdh}), we deduce 
	\begin{equation}\label{estBB}
		\begin{split}
			&\ii\tr\B_{\ell}^{q}(t)+\int_0^t\ii|\B_{\ell}|^{2q}+\int_0^t\ii|\nn\B_{\ell}^{\frac{q}2}|^2\\
			&\quad \leq \ii(\tr\B_0^{\omega})^{q}+ C\left(\int_0^t\ii|\Dv_{\ell}||\B_{\ell}|^{q}+1\right).
		\end{split}
	\end{equation}
	Then, to estimate the~term with $\tr\B_0^{\omega}$, we use \eqref{entOO} and \eqref{init}. On the~last term on the~right hand side, we apply Young's inequality, leading to
	\begin{equation}\label{integ}			\norm{\B_{\ell}}^{q}_{L^{\infty}L^{q}}+\norm{\B_{\ell}}_{L^{2q}L^{2q}}^{2q}+\norm{\nn\B_{\ell}^{\frac{q}2}}_{L^2L^2}^2\leq C\left(1+ \norm{\Dv_{\ell}}_{L^{2}L^{2}}^{2}\right),
	\end{equation}
	where the~right hand side is finite due to \eqref{strov}, but we do not have a~uniform bound yet. To obtain it, we 
	combine the~estimate \eqref{integ} with the~temperature equation~\eqref{sys1t} and improve the~information about $\theta_{\ell}$ and $\Dv_{\ell}$.
	
	Let $\beta\in [0,\f12]$ be arbitrary. We define
	\begin{equation*}
		\tau_{\beta}\coloneqq-\theta_{\ell}^{-\beta}.
	\end{equation*}
	Using Lemma~\ref{Ldual} with $\psi(s)=-\max(s,\omega)^{-\beta}$ to rewrite the~time derivative, the~a~priori bound~\eqref{th} with Young's inequality, $\eqref{s12bc}_1$ and \eqref{Akap}, we obtain the~estimate
	\begin{equation}
		\begin{aligned}\label{vypoc}			&\scal{c_v\partial_t\theta_{\ell},\tau_{\beta}}+(\kappa(\theta_{\ell})\nn\theta_{\ell},\nn\tau_{\beta})_{\Q}
			+\omega(|\nn\theta_{\ell}|^r\nn\theta_{\ell},\nn\tau_{\beta})_{\Q}\\	&\quad\geq\f{-c_v}{1-\beta}\ii\theta_{\ell}^{1-\beta}(T)+\beta\int_{\Q}\theta_{\ell}^{-1-\beta}
			\kappa(\theta_{\ell})|\nn\theta_{\ell}|^2+\omega\beta\int_{\Q}\theta_{\ell}^{-1-\beta}|\nabla \theta_{\ell}|^{r+2}\\
			&\quad\geq C\beta\int_{\Q}\Big|\nn\theta_{\ell}^{\f{r+1-\beta}2}\Big|^2 + \omega\beta \int_{\Q}\theta_{\ell}^{-1-\beta}|\nabla \theta_{\ell}|^{r+2}-C.
		\end{aligned}
	\end{equation}
	The function $\tau_{\beta}$ evidently satisfies $\tau_{\beta}\in L^{r+2}(0,T;W^{1,r+2}(\Omega))\cap L^{\infty}(0,T;L^{\infty}(\Omega))$ (cf.~\eqref{tinv}), and is thus an~admissible test function in~\eqref{sys1t}. This way, noting that the~convective term $(c_v\ve_{\ell}\cdot\nn\theta_{\ell},\tau_{\beta})_Q$ disappears since $\di \ve_{\ell}=0$, and having the~estimate \eqref{vypoc} and using also $g_{\omega}\leq 1$, we deduce
	\begin{equation*}
		\beta\int_{\Q}\Big(\Big|\nn\theta^{\f{r+1-\beta}2}_{\ell}\Big|^2+\omega\theta_{\ell}^{-1-\beta}|\nabla \theta_{\ell}|^{r+2}\Big)+\int_{\Q}\theta^{-\beta}_{\ell}|\Dv_{\ell}|^2\leq C\Big(\int_{\Q}\theta^{1-\beta}_{\ell}|\B_{\ell}||\Dv_{\ell}|+1\Big).
	\end{equation*}
	Note that since $\beta \in [0,\f12]$ is arbitrary, we can reduce the~above inequality to
	\begin{equation*}
		\beta\int_{\Q}\Big(\Big|\nn\theta^{\f{r+1-\beta}2}_{\ell}\Big|^2+\omega\theta_{\ell}^{-1-\beta}|\nabla \theta_{\ell}|^{r+2}\Big)+\int_{\Q}|\Dv_{\ell}|^2\leq C\Big(\int_{\Q}(\theta_{\ell}+1)|\B_{\ell}||\Dv_{\ell}|+1\Big),
	\end{equation*}
	which is very much similar to \eqref{keys}, while the~estimate \eqref{integ} mimics \eqref{Best}. Hence, applying the~Young and  the~H\"{o}lder inequality, and using \eqref{integ}, we deduce similarly as in \eqref{keys2} that
	\begin{equation}
		\beta\int_{\Q}\Big(\Big|\nn\theta^{\f{r+1-\beta}2}_{\ell}\Big|^2+\omega\theta_{\ell}^{-1-\beta}|\nabla \theta_{\ell}|^{r+2}\Big)+\int_{\Q}|\Dv_{\ell}|^2\leq C\Big(1+\int_{\Q}|\theta_{\ell}|^{2q'}\Big). \label{BBest}
	\end{equation}
	Next, we continue as after \eqref{keys2}. We recall the~interpolation inequality
	\begin{equation}\label{interp1}
		\|\theta_{\ell}\|^{2q'}_{2q'}\le  C\|\theta_{\ell}\|_{1}^{2q'-\frac{d(r-\beta+1)(2q' -1)}{d(r-\beta)+2}} \|\theta_{\ell}^{\frac{r+1-\beta}{2}}\|_{1,2}^{\frac{2d(2q' -1)}{d(r-\beta)+2}}.
	\end{equation}
	Thus, using the~uniform bound \eqref{th}, the~estimate \eqref{BBest} and the~interpolation inequality~\eqref{interp1}, we deduce
	\begin{equation}\label{dost}
		\begin{aligned}
			\beta \int_0^T \|\theta_{\ell}^{\frac{r+1-\beta}{2}}\|_{1,2}^2 &\le C\beta \int_0^T\big( \|\nabla \theta_{\ell}^{\frac{r+1-\beta}{2}}\|^2_{2} + \|\theta_{\ell}\|_{1}^{r+1-\beta}\big)\\
			&\le C\Big(1+\int_{\Q}|\theta_{\ell}|^{2q'}\Big)\le C+C\int_0^T\|\theta_{\ell}^{\frac{r+1-\beta}{2}}\|_{1,2}^{\frac{2d(2q' -1)}{d(r-\beta)+2}}.
		\end{aligned}
	\end{equation}
	Finally, thanks to \eqref{A0}, we can find $\beta_0>0$ such that for all $\beta\in (0,\beta_0)$ we have
	\begin{equation*}
		\frac{2d(2q' -1)}{d(r-\beta)+2}<2.
	\end{equation*}
	Consequently, we can use the~Young inequality in \eqref{dost} and conclude that
	\begin{equation}\label{kaniec}
		\norm{\theta_{\ell}^{\frac{r+1-\beta}{2}}}_{L^2W^{1,2}} \le C(\beta)
	\end{equation}
	for all $\beta \in (0,\beta_0)$ (which can be however easily extended  via \eqref{EST}  to the~validity for all $\beta\in (0,1)$).
	Furthermore, from the~interpolation inequality 
	\begin{equation}\label{interp3}
		\|\theta_{\ell}\|^{r+1+\f2d-\beta}_{r+1+\f2d-\beta}\leq  C\|\theta_{\ell}\|_{1}^{\frac{2}{d}} \|\theta_{\ell}^{\frac{r+1-\beta}{2}}\|_{1,2}^{2},
	\end{equation}
	\eqref{th} and \eqref{kaniec}, we conclude
	\begin{equation}\label{interp4}
		\norm{\theta_{\ell}}_{L^{r+1+\f2d-\beta}L^{r+1+\f2d-\beta}}\le C(\beta).
	\end{equation}
	
	\subsection*{Summary of all uniform estimates}
	To summarize the~estimates proved up to this point, we recall \eqref{th}, \eqref{xiest}, \eqref{etaint} and  \eqref{EST} based on the~use of total energy and entropy estimates. Next, having \eqref{kaniec}, we can  choose $\beta\coloneqq\f{\beta_0}2$ and  go backward in the~computation in the~previous part and obtain further a~priori estimates. Namely, using \eqref{kaniec} and \eqref{interp1}, we see that the~right hand side of \eqref{BBest} is  uniformly  bounded.
	Then, using \eqref{BBest} in \eqref{integ} we deduce also the~a~priori bound for $\B_{\ell}$. 
	Thus,  we can conclude with the~following set of estimates
	\begin{align}
		\label{vrh}
		\norm{\ve_{\ell}}_{L^{\infty}L^2}+\norm{\Dv_{\ell}}_{L^2L^2}&\leq C,\\
		\label{integr2}
		\norm{\B_{\ell}}_{L^{\infty}L^{q}}+\norm{\B_{\ell}}_{L^{2q}L^{2q}}+\norm{\nn\B_{\ell}^{\frac{q}2}}_{L^2L^2}&\leq C,\\
		\label{tint}
		\norm{\theta_{\ell}}_{L^{\infty}L^1}+\norm{\nn\theta_{\ell}^{\frac{r+1-\eps}{2}}}_{L^2L^2}+\norm{\theta_{\ell}}_{L^{r+1+\f{2}d-\eps}L^{r+1+\f2d-\eps}}&\leq C(\eps) 
	\end{align}
	for all $\eps\in(0,1)$.  Next, in order to obtain estimates on $\nn\B_{\ell}$, we separate two cases. If $1<q<2$, we use \eqref{IJ3} and H\"older's inequality, \eqref{I8} and \eqref{integr2} to estimate
	\begin{equation*}
		\norm{\nn\B_{\ell}}_{\f{4q}{q+2}}\leq 2\norm{\B_{\ell}^{1-\f q2}}_{\f{4q}{2-q}}\norm{\nn\B_{\ell}^{\f q2}}_2\leq C.
	\end{equation*}
	On the~other hand, if $q\geq2$, the~optimal estimate on $\nn\B_{\ell}$ is obtained simply by testing \eqref{sys1B} with $\B_{\ell}$ (instead of $\B_{\ell}^{q-1}$). Indeed, using \eqref{Bsigg}, we eventually obtain \eqref{integ}, but with $q=2$. Combination of these two cases leads to
	\begin{equation}\label{naB}
		\norm{\nn\B_{\ell}}_{L^{m}L^{m}}\leq C.
	\end{equation}
	

	\subsection*{Uniform time derivatives estimates} We end this part by derivation of the~uniform estimates for the~time derivatives. To this end, we need to determine integrability of the~non-linear terms in~\eqref{sys1v}, \eqref{sys1B} and \eqref{ENTl}. It follows from an~interpolation inequality, Korn's inequality, \eqref{th} and \eqref{vrh} that
	\begin{equation}\label{vint}
		\norm{\ve_{\ell}}_{L^{2\f{d+2}{d}}L^{2\f{d+2}{d}}}\leq C\norm{\ve_{\ell}}_{L^{\infty}L^2}^{\f2{d+2}}\norm{\Dv_{\ell}}_{L^2L^2}^{\f d{d+2}}\leq C.
	\end{equation}
	Furthermore, the~H\"older inequality, \eqref{tint} and \eqref{integ} yield
	\begin{equation}\label{thB}
		\norm{\theta_{\ell}\B_{\ell}}_{L^{2}L^{2}}\leq C,
	\end{equation}
	Hence, as $d\geq2$, we read from \eqref{sys1v} that
	\begin{equation}\label{Dv}
		\norm{\partial_t\ve_{\ell}}_{L^{\f{d+2}{d}}W^{-1,\f{d+2}{d}}_{\n,\di}}\leq C.
	\end{equation}	
	
	Next, we focus on the~non-linear terms in~\eqref{sys1B}.
	Using H\"older's inequality and \eqref{integ}, \eqref{vint}, we observe that
	\begin{equation}\label{x1}
		\norm{\B_{\ell}\otimes\ve_{\ell}}_{L^{s_1}L^{s_1}}\leq C,
	\end{equation}
	with
	\begin{equation}\label{s1g}
		s_1\coloneqq\Big(\f1{2q}+\frac{d}{2(d+2)}\Big)^{-1}>\Big(\f1{2q}+\f12\Big)^{-1}= \f{2q}{q+1} .
	\end{equation}
	Moreover, making use of \eqref{integr2} and \eqref{APbou}, we obtain
	\begin{equation}\label{gBint}
		\norm{\PP(\theta_{\ell},\B_{\ell})}_{L^{\f{2q}{q+1}}L^{\f{2q}{q+1}}}\leq C.
	\end{equation}
	Furthermore, using \eqref{integr2}, \eqref{vrh} and H\"older's inequality, we also get
	\begin{equation}\label{x}
		\norm{(a\Dv_{\ell}+\Wv_{\ell})\B_{\ell}}_{L^{\f{2q}{q+1}}L^{\f{2q}{q+1}}}\leq C.
	\end{equation}
Thus, we read from \eqref{sys1B} using  \eqref{naB} (where note that $m>\f{2q}{q+1}$), \eqref{x1}, \eqref{s1g} and \eqref{gBint}, \eqref{x} that
\begin{equation}\label{DB}
	\norm{\partial_t\B_{\ell}}_{L^{\f{2q}{q+1}}W^{-1,\f{2q}{q+1}}}\leq C.
\end{equation}

Next, we examine the~non-linearities related to~\eqref{ENTl}.
Since $\xi_{\ell}$ is controlled by~\eqref{xiest}, the~problematic terms could be only on the~left hand side. To get an~appropriate uniform control over the~convective term, we estimate
\begin{equation*}
	\eta_{\ell}\leq\eta_{\ell}+ f(\B_{\ell})=c_v\ln\theta_{\ell}\leq c_v(\theta_{\ell}-1).
\end{equation*}
This, together with \eqref{th} and \eqref{etaint}, yields
\begin{equation}\label{cln}
	\norm{\ln\theta_{\ell}}_{L^{\infty}L^1}\leq C.
\end{equation}
Then, since \eqref{EST} and \eqref{Akap} give
\begin{equation}\label{nnlnt}
	\norm{\nn\ln\theta_{\ell}}_{L^2L^2}\leq C,
\end{equation}
we can use Sobolev's inequality, Poincar\'e's inequality and an~interpolation to obtain
\begin{equation}\label{lnint}
	\norm{\ln\theta_{\ell}}_{L^{2+\f2d}L^{2+\f2d}}\leq C\norm{\ln\theta_{\ell}}_{L^{\infty}L^1}^{\f1{d+1}}\norm{\ln\theta_{\ell}}_{L^2W^{1,2}}^{\f d{d+1}}\leq C.
\end{equation}
Now we observe that a~similar reasoning applies also for the~quantity $\ln\det\B_{\ell}$. Indeed, using \eqref{etaint}, \eqref{cln}, \eqref{integ} and \eqref{etadef} in the~form
\begin{equation*}
	\ln\det\B_{\ell}=\f1{\mu}(\eta_{\ell}-c_v\ln\theta_{\ell})+\tr\B_{\ell}-d,
\end{equation*}
it is clear that
\begin{equation}\label{fb}
	\norm{\ln\det\B_{\ell}}_{L^{\infty}L^1}\leq C.
\end{equation}
Further, the~estimate of its derivative follows from a~version of Jacobi's formula (see Lemma~\ref{PDalg} below) and \eqref{EST} as
\begin{equation}\label{nntr}
	\norm{\nn\ln\det\B_{\ell}}_{L^2L^2}=\norm{\tr(\B^{-\f12}_{\ell}\nn\B_{\ell}\B^{-\f12}_{\ell})}_{L^2L^2}\leq C.
\end{equation}
Hence, using again the~Sobolev, the~Poincar\'e and interpolation inequalities, we get
\begin{equation}\label{ld}
	\norm{\ln\det\B_{\ell}}_{L^{2+\f2d}L^{2+\f2d}}\leq C.
\end{equation}
From \eqref{lnint}, \eqref{ld}, \eqref{integ} and \eqref{etadef}, we deduce
\begin{equation}\label{etag}
	\norm{\eta_{\ell}}_{L^{s_2}L^{s_2}}\leq C,\quad\text{where}\quad s_2\coloneqq\min\{2+\tfrac2d,2q\}>2,
\end{equation}
and thus
\begin{equation}\label{etcon}
	\norm{\ve_{\ell}\eta_{\ell}}_{L^{s_3}L^{s_3}}\leq C,\quad\text{where}\quad s_3\coloneqq\Big(\f d{2(d+2)}+\f1{s_2}\Big)^{-1}>1.
\end{equation}
We remark that, since
\begin{equation*}
	\nn\eta_{\ell}=c_v\nn\ln\theta_{\ell}-\mu(\tr\nn\B_{\ell}-\tr(\B_{\ell}^{-\f12}\nn\B_{\ell}\B^{-\f12}_{\ell})),
\end{equation*}
we also have, using \eqref{nntr}, \eqref{nnlnt}, \eqref{naB}, \eqref{EST} and Poincar\'e's inequality that
\begin{equation}\label{nneta}
	\norm{\eta_{\ell}}_{L^{m}W^{1,m}}\leq C.
\end{equation}
Looking at \eqref{ENTl}, we still need to verify that the~flux terms are controlled. For the~term $\kappa(\theta_{\ell})\nn\ln\theta_{\ell}$, we  first use \eqref{Akap} and \eqref{tint} to estimate
\begin{equation*}
	\norm{\sqrt{\kappa(\theta_{\ell})}}_{L^{\f{2d(r+1)+4}{dr}}(\Q;\R)}\leq C\norm{1+\theta}_{L^{r+1+\f2d}(\Q;\R)}^{\f{r}2}\leq C
\end{equation*}
and then, by H\"older's inequality and \eqref{EST}, we get
\begin{equation}\begin{aligned}\label{enttes}
		&\norm{\kappa(\theta_{\ell})\nn\ln\theta_{\ell}}_{L^{\f{2d(r+1)+4}{2d(r+1)+2-d}}(\Q;\R^d)}\\
		&\qquad\leq\norm{\sqrt{\kappa(\theta_{\ell})}}_{L^{\f{2d(r+1)+4}{dr}}(\Q;\R)}\norm{\sqrt{\kappa(\theta_{\ell})}\nn\ln\theta_{\ell}}_{L^2(\Q;\R^d)}\leq C.
\end{aligned}\end{equation}
Further, let us derive an~estimate on $\omega|\nn\theta_{\ell}|^r\nn\ln\theta_{\ell}$, from which it follows that this term vanishes as $\omega\to0+$.  The~H\"older inequality, \eqref{EST} and \eqref{tint} yield 
\begin{equation}\label{tdrt}
	\begin{split}
		\omega\int_{\Q}|\nn\theta_{\ell}|^r|\nn\ln\theta_{\ell}|&=  \omega^{\frac{1}{r+2}}\int_{\Q}\big(\omega\theta_{\ell}^{-2}|\nabla \theta_{\ell}|^{ r+2}\big)^{\frac{r+1}{r+2}}\theta_{\ell}^{2\frac{r+1}{r+2}-1}\\
		&\le C\omega^{\frac{1}{r+2}}\Big(\int_{\Q}\theta_{\ell}^{r}\Big)^{\frac{1}{r+2}}\le C\omega^{\frac{1}{r+2}}.
	\end{split}
\end{equation}
From this and from \eqref{enttes}, \eqref{nnlnt}, \eqref{nntr}, \eqref{etcon}, \eqref{EST}, \eqref{ENTl}, we see, using the~definition of a~weak time derivative, that
\begin{equation}\label{dteta}
	\norm{\partial_t\eta_{\ell}}_{L^1W^{-M,2}}\leq C,
\end{equation}
where $M$ is so large that $W^{M,2}(\Omega;\R)\embl W^{1,\infty}(\Omega;\R)$.

Finally, we focus on terms appearing in the~temperature equation~\eqref{sys1t2}. First, we note that it is a~consequence of assumption \eqref{A0}, a~priori estimates \eqref{vrh}--\eqref{tint} and the~H\"{o}lder inequality, that
\begin{equation}\label{tukos}
	\int_{\Q}|\theta_{\ell}\ve_{\ell}|+\int_Q|\Sb_{\ell}^{\omega}\cdot\Dv_{\ell}|\leq C.
\end{equation}
In the~terms involving temperature gradient, we use \eqref{Akap}, \eqref{EST}, \eqref{tint} and the~inequality $\max\{2,r+1+\eps\}< r+1+\f2d-\eps$ for $\eps$ small (recall \eqref{A0}) to estimate
\begin{equation}\begin{aligned}\label{QQ1}
		\int_{\Q}|\kappa(\theta_{\ell})\nn\theta_{\ell}|&\leq C\int_{\Q}(\theta_{\ell}|\nn\ln\theta_{\ell}|+\theta_{\ell}^{\f{r+1+\eps}2}|\nn\theta^{\f{r+1-\eps}2}|)\\
		&\leq C\int_{\Q}(\theta_{\ell}^2+\theta_{\ell}^{r+1+\eps})\leq C.
\end{aligned}\end{equation}
Proceeding similarly as in \eqref{tdrt}, but using now \eqref{BBest} instead of \eqref{EST}, we also find
\begin{equation}\label{QQ2}
	\begin{split}
		\int_{\Q}\omega |\nabla \theta_{\ell}|^{r+1}&= \omega^{\frac{1}{r+2}}\int_{\Q}\big(\omega\theta_{\ell}^{-1-\beta}|\nabla \theta_{\ell}|^{r+2} \big)^{\frac{r+1}{r+2}}  \theta_{\ell}^{\frac{(\beta+1)(r+1)}{r+2}}\\
		&\le C \omega^{\frac{1}{r+2}}\Big(\int_{\Q} \theta_{\ell}^{(\beta+1)(r+1)}\Big)^{\frac{1}{r+2}}\le C \omega^{\frac{1}{r+2}},
	\end{split}
\end{equation}
where $\beta>0$ is chosen so small that $(\beta+1)(r+1)<r+1+\f2d$.
Using the~above estimates  in \eqref{sys1t}, we  deduce
\begin{equation}\label{dteplota3}
	\norm{\partial_t\theta_{\ell}}_{L^1W^{-M,2}}\leq C,
\end{equation}
for sufficiently large $M$. Very similarly,  choosing $\theta_{\ell}^{-\f12}\phi$ in \eqref{sys1t} and  repeating the~method for estimating $\partial_t \eta_{\ell}$,  we can find that
\begin{equation}
	\label{dteplota}
	\norm{\partial_t\theta_{\ell}^{\frac12}}_{L^1W^{-M,2}}\leq C.
\end{equation}
Finally, returning to \eqref{2test} with \eqref{BIde} and using the~uniform estimates proved so far, it is easy to see that also
\begin{equation}\label{dfB}
	\norm{\partial_t f(\B_{\ell})}_{L^1W^{-M,2}}\leq C.
\end{equation}
The last two properties will be useful in the~initial condition identification. 

\subsection*{Limits \texorpdfstring{$\omega\to0$}{w}, \texorpdfstring{$\ell\to\infty$}{l}.}

Let us note that the~estimates above are independent not only of $\ell$, but also of $\omega$. Hence, we can set $\omega\coloneqq\ell^{-1}$ and hereby, it remains to take the~limit $\ell\to\infty$ only.

By collecting the~estimates \eqref{th}, \eqref{tint}--\eqref{integr2}, \eqref{naB}, \eqref{DB}, \eqref{Dv},  \eqref{etag}, \eqref{nneta}, \eqref{dteta}, \eqref{QQ1}, \eqref{dteplota3}, \eqref{dteplota} and using the~Aubin--Lions lemma and Vitali's convergence theorem, we get the~following results:
{\allowdisplaybreaks\begin{align}
		\ve_{\ell}&\wc \ve&&\text{weakly in }L^{2}(0,T;W^{1,2}_{\n,\di}),\label{ccvw}\\
		\ve_{\ell}&\to\ve&&\text{strongly in }L^{2\f{d+2}d-\eps}(\Q;\R^d)\text{ and a.e.\ in }\Q,\label{ccvs}\\
		\partial_t\ve_{\ell}&\wc\partial_t\ve&&\text{weakly in }L^{\f{d+2}{d}}(0,T;W^{-1,\f{d+2}{d}}_{\n,\di}),\label{ccvdt}\\
		\B_{\ell}&\wc\B&&\text{weakly in }L^{m}(0,T;W^{1,m}(\Omega;\Sym)),\label{ccBw}\\
		\B_{\ell}&\to\B&&\text{strongly in }L^{2q-\eps}(\Q;\Sym)\text{ and a.e.\ in }\Q,\label{ccBs}\\
		\partial_t\B_{\ell}&\wc\partial_t\B&&\text{weakly in }L^{\f{2q}{q+1}}(0,T;W^{-1,\f{2q}{q+1}}(\Omega;\Sym)),\label{ccBdt}\\
		\eta_{\ell}&\wc\eta&&\text{weakly in }L^{m}(0,T;W^{1,m}(\Omega;\R)),\label{ccetw}\\
		\eta_{\ell}&\to\eta&&\text{strongly in }L^{s_2-\eps}(\Q;\R)\text{ and a.e.\ in }\Q,\label{ccets}\\
		\eta_{\ell}&\wcs\eta&&\text{weakly$^*$ in }BV(0,T; W^{-M,2}(\Omega;\mathbb{R})),\label{UUccets}\\
		\theta^{\frac{r+1-\eps}{2}}_{\ell}&\wc\theta^{\frac{r+1-\eps}{2}}&&\text{weakly in }L^{2}(0,T;W^{1,2}(\Omega;\R)),\label{ccetwop}\\
		\theta_{\ell}&\to\theta&&\text{strongly in }L^{r+1+\f2d-\eps}(\Q;\R),\label{cctw}\\
		\theta_{\ell}^{\frac12}&\wcs\theta^{\frac12}&&\text{weakly$^*$ in }BV(0,T; W^{-M,2}(\Omega;\mathbb{R})),\label{cctw2}\\ 
		\theta_{\ell}&\wcs\theta&&\text{weakly$^*$ in }BV(0,T; W^{-M,2}(\Omega;\mathbb{R}))\label{cctw22}
\end{align}}
for any $\eps\in(0,1)$. Using these properties, we shall now explain how to take the~limit in equations \eqref{sys1v}, \eqref{sys1B}, \eqref{ENTl}, \eqref{Ener0} and  \eqref{sys1t}. 

First, we focus on taking the~limit in the~function $g_{\f1{\ell}}$. From \eqref{pd}, \eqref{pt} and \eqref{ccvw}, \eqref{ccBw} (or \eqref{ccvs}, \eqref{ccBs}), we obtain
\begin{equation}\label{neup}
	\B\x\cdot\x\geq0 \quad\text{for all }\x\in\R^d\quad\text{and}\quad\theta\geq0\quad\text{a.e.\ in }\Q,
\end{equation}
however, we need these properties with strict inequalities. To this end, we use the~Fatou lemma, \eqref{ccBs} and \eqref{fb} to get
\begin{equation*}
	\ii|\ln\det\B|\leq\liminf_{\ell\to\infty}\ii|\ln\det\B_{\ell}|\leq C\quad\text{a.e.\ in }(0,T).
\end{equation*}
Thus, by taking the~essential supremum over $(0,T)$, we obtain
\begin{equation}\label{pdd}
	\norm{\ln\det\B}_{L^{\infty}L^1}<\infty,
\end{equation}
which, together with \eqref{neup} implies
\begin{equation}\label{pdB}
	\B\x\cdot\x>0 \quad\text{for all }\x\in\R^d\quad\text{a.e.\ in }\Q.
\end{equation}
An analogous argument, using now \eqref{cctw} and \eqref{cln}, shows that
\begin{equation}\label{tp}
	\theta>0\quad\text{a.e.\ in }\Q.
\end{equation}
With this in hand, note that the~property \eqref{s4} follows from \eqref{etadef} and the~pointwise a.e.~convergence of $\eta_{\ell}$, $\theta_{\ell}$ and $\B_{\ell}$. Also, 
from \eqref{pdB}, \eqref{tp} and the~pointwise convergence we deduce that, at almost every point $(t,x)\in \Q$, we can find $M_{t,x}\in\N$ such that for all $\ell>M_{t,x}$ we have
\begin{equation*}
\Lambda(\B_{\ell}(t,x))>\f12\Lambda(\B(t,x))>\f1{\ell}\quad\text{and}\quad\theta_{\ell}(t,x)>\f12\theta(t,x)>\f1{\ell}.
\end{equation*}
Then, looking at the~definition of $g_{\la}$, we see that at almost every point $(t,x)\in \Q$ and for $\ell> M_{t,x}$, the~positive parts $\max\{0,\cdot\}$ can be removed and thus, it is clear that $g_{\f1{\ell}}(\B_{\ell},\theta_{\ell})$ converges pointwise a.e.\ in $\Q$ to $1$. Hence, the~Vitali theorem and $0\leq g_{\f1{\ell}}<1$, imply
\begin{align}\label{ccg}
g_{\f1{\ell}}(\B_{\ell},\theta_{\ell})&\to 1\quad\text{strongly in}\quad L^{p}(\Q;\R)\quad\text{for any}\quad1\leq p<\infty.
\end{align}
Therefore, regarding the~first two equations \eqref{sys1v} and \eqref{sys1B}, we can take the~limit in the~same way as we did in the~limit $n\to\infty$. Indeed, the~integrability of the~resulting non-linear limits was already verified when estimating $\partial_t\ve_{\ell}$ and $\partial_t\B_{\ell}$ (\eqref{vint}--\eqref{x}). This way, taking \eqref{ccg} into account, using the~density of $\spa\{\we_i\}_{i=1}^{\infty}$ in $W^{1,\f{d}2+1}_{\n,\di}$ and extending the~functional $\partial_t\B$ to the~space stated in \eqref{DRdtB} using \eqref{DB}, we obtain precisely \eqref{Qv} and \eqref{QB}.

Next, we show how to take the~limit in~\eqref{ENTl}.  Regarding the~initial condition, 
using \eqref{entO} and \eqref{entOO}, we estimate
\begin{equation*}
|\eta_0^{\omega}|\leq c_v|\ln\theta_0^{\omega}|+\mu(|\tr\B_0^{\omega}|+d+|\ln\det\B_0^{\omega}|)\leq C(|\ln\theta_0|+|\B_0|+|\ln\det\B_0|+1),
\end{equation*}
where the~right hand side is integrable by assumptions \eqref{init} and \eqref{initln}. Moreover, the~function $\eta_0^{1/\ell}$ converges point-wise a.e.\ in $\Omega$ due to~\eqref{B0con} and \eqref{t0con}.  Thus, by the~dominated convergence theorem, the~function $\eta_0^{\f1{\ell}}$ converges to $\eta_0$ in $L^1(\Omega;\R)$.  In order to take the~limit in the~convective term, we use \eqref{ccvs}, \eqref{ccets} and \eqref{etcon}. Next,  the~properties  \eqref{lnint}, \eqref{ld}, \eqref{cctw}, \eqref{ccBs} and \eqref{nntr}, \eqref{nnlnt} imply
\begin{align}
\ln\theta_{\ell}&\wc\ln\theta&&\text{weakly in }L^2(0,T;W^{1,2}(\Omega;\R)),\label{cclnt}\\
\ln\det\B_{\ell}&\wc\ln\det\B&&\text{weakly in }L^2(0,T;W^{1,2}(\Omega;\R))\label{cclnB}.
\end{align} 
Further, we use \eqref{Acont}, \eqref{Akap}, \eqref{tint} and Vitali's theorem to find that
\begin{align}
\sqrt{\kappa(\theta_{\ell})}&\wc\sqrt{\kappa(\theta)}\quad\text{strongly in } L^{\f{2(r+1)}{r}}(\Q;\R).\label{kaps}
\end{align}
As a~consequence of this, \eqref{cclnt} and \eqref{EST}, we get
\begin{equation}\label{sqrt}
\sqrt{\kappa(\theta_{\ell})}\nn\ln\theta_{\ell}\wc\sqrt{\kappa(\theta)}\nn\ln\theta\quad\text{weakly in } L^2(\Q;\R^d).
\end{equation}
Therefore, using again \eqref{kaps}, we obtain
\begin{equation*}
\kappa(\theta_{\ell})\nn\ln\theta_{\ell}\wc\kappa(\theta)\nn\ln\theta\quad\text{weakly in }L^1(\Q;\R^d).
\end{equation*}
Next, in the~term $\mu\la(\theta_{\ell})\nn\tr\B_{\ell}$, we use \eqref{Acont}, \eqref{Ala}, \eqref{cctw}, Vitali's theorem and \eqref{ccBw}. Analogously, we take the~limit in the~term $\mu\la(\theta_{\ell})\nn\ln\det\B_{\ell}$, only we use \eqref{cclnB} instead of \eqref{ccBw}. The~term containing $\omega|\nn\theta_{\ell}|^r\nn\ln\theta_{\ell}$ tends to zero by \eqref{tdrt}.

Now we take the~limit in the~terms on the~right hand side of \eqref{ENTl}, i.e., the~function $\xi_{\ell}$ defined in \eqref{xil}. Note that we just need to pass to the~limit with possible inequality sign (selecting non-negative test functions $\phi$, $\varphi$). To take the~limit in the~term $\PP(\theta_{\ell},\B_{\ell})\cdot(\I-\B_{\ell}^{-1})\phi \varphi\geq0$, we use \eqref{cctw}, \eqref{ccBs} and apply Fatou's lemma. Next, in term $\kappa(\theta_{\ell})|\nn\ln\theta_{\ell}|^2\phi \varphi$, we  use \eqref{sqrt} and the~weak lower semi-continuity.  Moreover, the~auxiliary term $\omega|\nn\theta_{\ell}|^r|\nn\ln\theta_{\ell}|^2\phi\varphi$ is simply estimated from below by zero. Thus, in order to let $\ell\to\infty$ in~\eqref{ENTl}, it remains to show that
\begin{equation*}
\begin{split} &\liminf_{\ell\to\infty}\int_{\Q}\left(\f{2\nu(\theta_{\ell})}{\theta_{\ell}}|\Dv_{\ell}|^2+\la(\theta_{\ell})\left|\B_{\ell}^{-\f12}\nn\B_{\ell}\B_{\ell}^{-\f12}\right|^2\right)\phi \varphi\\
	&\quad\geq\int_{\Q} \left(\f{2\nu(\theta)}{\theta}|\Dv|^2+\la(\theta)\left|\B^{-\f12}\nn\B\B^{-\f12}\right|^2\right)\phi\varphi.
\end{split}
\end{equation*}
The above inequality is however consequence of the~weak lower semicontinuity and the~following claim
\begin{equation}\label{claim}
\begin{aligned}
	\sqrt{\la(\theta_{\ell})}\B^{-\f12}_{\ell}\nn\B_{\ell}\B_{\ell}^{-\f12}&\wc \sqrt{\la(\theta)}\B^{-\f12}\nn\B\B^{-\f12}&&\text{weakly in }L^2(\Q;\R^d\times\Sym),\\
	\sqrt{\f{2\nu(\theta_{\ell})}{\theta_{\ell}}}\Dv_{\ell}&\wc \sqrt{\f{2\nu(\theta)}{\theta}}\Dv&&\text{weakly in }L^2(\Q;\Sym),
\end{aligned}
\end{equation}
which we need to obtain. To do so, we start with \eqref{EST} and therefore we have (for a~proper subsequence) that
\begin{align}\label{ww}
\sqrt{\la(\theta_{\ell})}\B^{-\f12}_{\ell}\nn\B_{\ell}\B_{\ell}^{-\f12}&\wc G &&\text{weakly in }L^2(\Q;\R^d\times\Sym),\\
\label{K1}
\sqrt{\f{2\nu(\theta_{\ell})}{\theta_{\ell}}}\Dv_{\ell}&\wc K &&\text{weakly in }L^2(\Q;\Sym).
\end{align}
Thus, it remains to show
\begin{equation}\label{imop}
\begin{aligned}
	\sqrt{\la(\theta)}\B^{-\f12}\nn\B\B^{-\f12}&=G, \qquad \sqrt{\f{2\nu(\theta)}{\theta}}\Dv=K.
\end{aligned}
\end{equation}
First, we use the~Egorov theorem and then it follows from  \eqref{fb}, \eqref{ccBs}, \eqref{pdd}, \eqref{tp} and \eqref{cctw} that for any $\varepsilon>0$ there exists measurable $Q_{\varepsilon}\subset \Q$ fulfilling $|\Q\setminus Q_{\varepsilon}|\le \varepsilon$ such that
\begin{equation*}
\begin{aligned}
	\B^{-\f12}_{\ell}\rightrightarrows \B^{-\f12}, \qquad \sqrt{\la(\theta_{\ell})}\rightrightarrows \sqrt{\la(\theta)}, \qquad \sqrt{\f{2\nu(\theta_{\ell})}{\theta_{\ell}}} \rightrightarrows \sqrt{\f{2\nu(\theta)}{\theta}}
\end{aligned}
\end{equation*}
uniformly in $Q_{\varepsilon}$. Combining the~above uniform convergence results with the~weak convergence results \eqref{ccvw} and  \eqref{ccBw}, we deduce
\begin{equation}\label{claime}
\begin{aligned}
	\sqrt{\la(\theta_{\ell})}\B^{-\f12}_{\ell}\nn\B_{\ell}\B_{\ell}^{-\f12}&\wc \sqrt{\la(\theta)}\B^{-\f12}\nn\B\B^{-\f12}&&\text{weakly in }L^1(Q_{\varepsilon};\R^d\times\Sym),\\
	\sqrt{\f{2\nu(\theta_{\ell})}{\theta_{\ell}}}\Dv_{\ell}&\wc \sqrt{\f{2\nu(\theta)}{\theta}}\Dv&&\text{weakly in }L^1(Q_{\varepsilon};\Sym).
\end{aligned}
\end{equation}
Thus, the~uniqueness of a~weak limit implies that \eqref{imop} is satisfied a.e.~in $Q_{\varepsilon}$. Since $\varepsilon>0$ was arbitrary, we can let $\varepsilon \to 0+$ and conclude that \eqref{imop} holds true a.e.~in $\Q$. Consequently, we deduced \eqref{claim} and therefore we proved \eqref{Qent}.			

In addition, in very similar manner we can let $\ell \to \infty$ in \eqref{sys1t} to obtain \eqref{Qthe}. Note that contrary to the~entropy inequality, we use here in addition the~estimates \eqref{tukos}, \eqref{QQ1} and \eqref{QQ2}. Otherwise, the~proof is almost identical.	

To take the~limit in~\eqref{Ener0}, we first note, using \eqref{vv} and \eqref{t0}, that it implies
\begin{equation}\label{eEeq}
-(E_{\ell},\partial_t\phi)_{\Q}+\alpha(|\ve_{\ell}|^2,\phi)_{\SigT}=(\tfrac12|P_{\ell}\ve_0|^2+c_v\theta_0^{\f1{\ell}},\phi(0))+(\fe,\ve_{\ell}\phi)_{\Q}
\end{equation}
for all $\phi\in\mathcal{C}^{1}([0,T];\R)$ with $\phi(T)=0$. Then, recalling \eqref{ccvs} and \eqref{cctw}, we see that $E_{\ell}=\tfrac12|\ve_{\ell}|^2+c_v\theta_{\ell}$ converges strongly to $E$ and thus, using also properties of $P_{\ell}$ and \eqref{t0con}, we can take the~limit in \eqref{eEeq} to conclude
\begin{equation}\label{Ener}
-(E,\partial_t\phi)_{\Q}+\alpha(|\ve|^2,\phi)_{\SigT}=(E_0,\phi(0))+(\fe,\ve\phi)_{\Q},
\end{equation}
where we set $E_0\coloneqq\frac12|\ve_0|^2+c_v\theta_0$. In particular, by choosing an appropriate sequence of test functions $\phi$, we obtain \eqref{QE}.

\subsection*{Attainment of initial conditions}

To finish the~existence~proof, it remains to identify the~initial conditions and show that they are attained strongly. Let us start by an~observation that $\ve$ and $\B$ are weakly continuous in time. Indeed, first of all, we recall that
\begin{equation}\label{kkl}	
\begin{aligned}
	&\ve\in L^{\infty}(0,T;L^2(\Omega;\R^d)),\quad\partial_t\ve\in L^{\f{d+2}{d}}(0,T;W^{-1,\f{d+2}{d}}_{\n,\di}(\Omega;\R^d)),\\
	&\B\in L^{\infty}(0,T;L^{q}(\Omega;\PD)),\quad\partial_t\B\in L^{\f{2q}{q+1}}(0,T;W^{-1,\f{2q}{q+1}}(\Omega;\Sym)),
\end{aligned}
\end{equation}
cf.\ \eqref{integr2} and \eqref{DB}. From this we obtain, by a~standard argument known from the~theory of Navier--Stokes equations (see e.g.\ \cite[Sect.~3.8.]{Malek2005}), that
\begin{equation}
\ve\in \CO_w([0,T];L^2(\Omega;\R^d))\quad\text{and}\quad
\B\in \CO_w([0,T];L^{q}(\Omega;\R^d)).\label{bcw}
\end{equation}
Then, to identify the~corresponding weak limits, we can use an~analogous idea as in the~part where the~limit $n\to\infty$ was taken together with \eqref{B0con}. This way, we obtain
\begin{equation}
\lim_{t\to0+}(\ve(t),\we)=(\ve_0,\we)\quad\text{for all }\we\in L^2(\Omega;\R^d)\label{wccv}
\end{equation}
and
\begin{equation}
\lim_{t\to0+}(\B(t),\Wb)=(\B_0,\Wb)\quad\text{for all }\Wb\in L^{q'}(\Omega;\Sym).\label{wccB}
\end{equation}

Next, we use a~similar procedure for entropy and temperature. Recalling \eqref{UUccets} and \eqref{cctw2}, we can define for all $t^0\in [0,T]$ the~values $\sqrt{\theta}(t^0_{\pm})$, $\eta(t^0_{\pm})$ such that
\begin{equation}\label{sssl}
\lim_{t\to t^0{\pm}}\big(\big\|\sqrt{\theta(t)}-\sqrt{\theta(t^0_{\pm})}\big\|_{W^{-M,2}(\Omega;\R)}+\norm{\eta(t)-\eta(t^0_{\pm})}_{W^{-M,2}(\Omega;\R)}\big)=0.
\end{equation}
Therefore, using the~density of $L^w(\Omega;\R)$ in $W^{-M,2}(\Omega;\R)$, which is valid for all $w\in (1,\infty)$ and $M$ sufficiently large, and recalling the~fact that $\theta \in L^{\infty}(0,T; L^1(\Omega;\R))$, we can deduce that there is non-negative $\thlim\in L^{1}(\Omega; \R)$ fulfilling
\begin{equation}
\lim_{t\to0+}(\sqrt{\theta(t)},\zeta)=(\sqrt{\thlim},\zeta)\quad\text{for all }\zeta\in L^{2}(\Omega;\R).\label{wcctep}
\end{equation}
Our aim is to show that $\thlim=\theta_0$ and that it is attained strongly. 	

Unlike in the~theory of Navier--Stokes(--Fourier) systems, we can not draw information about $\limsup_{t\to0+}\norm{\ve(t)}_2^2$ from the~(kinetic) energy estimate directly because of the~presence of $\theta\B$ in~\eqref{Qv}. Instead, we need first to combine the~total energy and entropy balances to obtain the~initial condition for $\theta$. In \eqref{Ener} we choose a~sequence of test functions $\phi$ approximating the~function $\chi_{[0,t)}$, $t\in(0,T)$. This way, after taking the~appropriate limit, we arrive at
\begin{equation}\label{Ener3}
\ii E(t)+\alpha\int_0^t \ii |\ve|^2=\ii E_0+\int_0^t\ii\fe\cdot\ve\quad\text{for a.a.\ }t\in(0,T).
\end{equation}
Next, we strengthen the~above relation to be valid for all $t\in (0,T)$ with possibly inequality sign. Due to the~weak continuity of $\ve$, see~\eqref{bcw}, we see that $\ve(\tau)$ is uniquely defined for all $\tau\in (0,T)$ and
\begin{equation}
\lim_{t\to \tau}(\ve(t),\we)=(\ve(\tau),\we)\quad\text{for all }\we\in L^2(\Omega;\R^d).\label{wccvMB}
\end{equation}
The same is however not true for $\theta$ since it is not weakly continuous w.r.t.~$t\in (0,T)$. Nevertheless, we can define one-side values for every $t\in (0,T)$ with the~help of~\eqref{cctw2}, i.e., using the~similar arguments as in \eqref{wcctep}, we have the~one-sided uniquely defined weak limit
\begin{equation}
\lim_{t\to \tau\pm}(\sqrt{\theta(t)},\zeta)=(\sqrt{\theta(\tau_{\pm})},\zeta)\quad\text{for all }\zeta\in L^{2}(\Omega;\R).\label{wcctepMB}
\end{equation}

Next, we use above weak convergence results in \eqref{Ener3}. Integrating it with respect to $t\in (\tau, \tau+\delta)$, we get
\begin{equation*}
\int_{\tau}^{\tau+\delta}\ii E(t)\dd{t}=\int_{\tau}^{\tau+\delta} \int_0^t\left(\ii\fe\cdot\ve-\alpha\ii |\ve|^2 \right) \dd{t} +\delta\ii E_0.
\end{equation*}	
Thus, dividing by $\delta$, letting first $\delta\to 0+$ and then $\tau\to 0+$, we get
\begin{equation}\label{Ener3MB}
\lim_{\tau\to 0+} \lim_{\delta\to 0+} \delta^{-1}	\int_{\tau}^{\tau+\delta}\ii E(t)\,\textrm{d} t=\ii E_0=	\ii(\tfrac12|\ve_0|^2+c_v\theta_0)
\end{equation}	
and in a~very similar manner, we obtain
\begin{equation}\label{Ener3MB2}
\lim_{\tau\to 0+} \lim_{\delta\to 0+} \delta^{-1}	\int_{\tau-\delta}^{\tau}\ii E(t)\,\textrm{d} t=\ii E_0=	\ii(\tfrac12|\ve_0|^2+c_v\theta_0).
\end{equation}
We focus on the~term on the~left hand side. Using the~convexity, we have
\begin{equation*}
\begin{split}
	\delta^{-1}	&\int_{\tau}^{\tau+\delta}\ii E(t)\,\textrm{d} t=\delta^{-1}	\int_{\tau}^{\tau+\delta}\ii(\tfrac12|\ve (t)|^2+c_v\theta(t))\,\textrm{d} t\\
	&\ge\delta^{-1}	\int_{\tau}^{\tau+\delta}\ii(\tfrac12|\ve (\tau)|^2+c_v\theta(\tau_+)) +\ve(\tau)\cdot (\ve(t)-\ve(\tau)) \\
	&\qquad + 2c_v\sqrt{\theta(\tau_+)}\big(\sqrt{\theta(t)}-\sqrt{\theta(\tau_+)}\big) \,\textrm{d} t \\
	&\ge \ii(\tfrac12|\ve(\tau)|^2+c_v\theta(\tau_+)) \\
	&\quad -\sup_{t\in (\tau, \tau+\delta)}\left|\ii\ve(\tau)\cdot (\ve(t)-\ve(\tau))+ 2c_v\sqrt{\theta(\tau_+)}\big(\sqrt{\theta(t)}-\sqrt{\theta(\tau_+)}\big)\right|.
\end{split}
\end{equation*}
Then, it follows from the~weak continuity results \eqref{wccvMB} and \eqref{wcctepMB} and from the~above inequality that
\begin{equation*}
\begin{split}
	\lim_{\delta \to 0+}\delta^{-1}	&\int_{\tau}^{\tau+\delta}\ii E(t)\,\textrm{d} t\ge  \ii(\tfrac12|\ve (\tau)|^2+c_v\theta(\tau_+)).
\end{split}
\end{equation*}
Repeating the~same procedure we also get
\begin{equation*}
\begin{split}
	\lim_{\delta \to 0+}\delta^{-1}	&\int_{\tau-\delta}^{\tau}\ii E(t)\,\textrm{d} t\ge  \ii(\tfrac12|\ve (\tau)|^2+c_v\theta(\tau_{-})).
\end{split}
\end{equation*}
Consequently, combining it with \eqref{Ener3MB} and \eqref{Ener3MB2}, and also with \eqref{init}, \eqref{wccv}, \eqref{wcctep} and weak lower semi-continuity, we get
\begin{align*}
\ii(\tfrac12|\ve_0|^2+c_v\theta_0)&\ge \limsup_{t\to0+}\ii(\tfrac12|\ve(t)|^2+c_v\theta(t_{\pm}))\\
&\geq\liminf_{t\to0+}\ii\tfrac12|\ve(t)|^2+\limsup_{t\to0+}\ii c_v\theta(t_{\pm})\\
&\geq\ii\tfrac12|\ve_0|^2+\limsup_{t\to0+}\ii c_v\theta(t_{\pm}),
\end{align*}
hence due to \eqref{wcctep} and the~convexity of the~second power, we have
\begin{equation}\label{eslims}
\ii\thlim \le 	\limsup_{t\to0+}\ii\theta(t_{\pm})\leq\ii\theta_0.
\end{equation}

In what follows we will not distinguish ``$\pm$" in $\theta(t_{\pm})$ and $\eta(t_{\pm})$ and simply write $\theta(t)$ and $\eta(t)$. To obtain also the~corresponding lower estimate, we need to extract the~available information from the~entropy inequality~\eqref{Qent}. To this end, we localize~\eqref{Qent} in time, using a~sequence of non-negative functions approximating $\chi_{[0,t)}$. This way, we eventually obtain
\begin{equation}\label{ii}
\ii\eta(t)\phi+\int_0^{t}\ii\je\cdot\nn\phi\geq \ii\eta_0\phi+\int_0^t\ii\xi\phi
\end{equation}
a.e.\ in $(0,T)$ and for all $\phi\in W^{M,2}(\Omega;\R_{\geq0})$, where
\begin{align*}
\je&\coloneqq-\ve\eta+\kappa(\theta)\nn\ln\theta-\mu\la(\theta)\nn(\tr\B-d-\ln\det\B)\in L^1(\Q;\R^d).
\end{align*}
Hence, using \eqref{sssl} and taking $\liminf_{t\to0+}$ of \eqref{ii} (which surely exists due to \eqref{sssl}), we deduce \eqref{Qicet}. Let us now fix $\varphi\in\CO^M(\Omega;\R_{\geq0})$ such that $\ii\varphi=1$. Since $ f$ is convex, we get from \eqref{Qicet} and \eqref{wccB} (or \eqref{QicB}) that
\begin{align*}
\ii c_v\ln\theta_0\,\varphi=\ii\eta_0\varphi+\ii f(\B_0)\varphi
&\leq\liminf_{t\to0+}\ii\eta(t)\varphi+\liminf_{t\to0+}\ii f(\B(t))\varphi\\
&\leq\liminf_{t\to0+}\ii c_v\ln\theta(t)\,\varphi.
\end{align*}
If we use this information together with Jensen's inequality and the~fact that the~function $s\mapsto\exp(s/2)$, is increasing and convex in $\R$, we are led to
\begin{align}\begin{aligned}\label{eqq}
	&\exp\left(\frac12\ii\ln\theta_0\varphi\right)
	\leq\exp\left(\frac12\liminf_{t\to0+}\ii \ln\theta(t)\varphi\right)\\
	&\qquad=\liminf_{t\to0+}\exp\left(\ii \ln \sqrt{\theta(t)}\,\varphi\right)\leq \liminf_{t\to0+}\ii \sqrt{\theta(t)}\varphi\\
	&\qquad=\ii \sqrt{\thlim}\varphi.
\end{aligned}
\end{align}
In every Lebesgue point $x_0\in\Omega$ of both $\ln\theta_0$ and $\thlim$, we can localize inequality~\eqref{eqq} in~$\Omega$ by choosing a~sequence of functions $\varphi$ that approximates the~Dirac delta distribution at $x_0\in\Omega$. Indeed, appealing to the~Lebesgue differentiation theorem, we get this way that
\begin{equation*}
\sqrt{\theta_0(x_0)}=\exp\left(\frac12\ln\theta(x_0)\right)\leq \sqrt{\thlim(x_0)},
\end{equation*}
hence $\theta_0\leq\thlim$ a.e.\ in $\Omega$, which together with \eqref{eslims} implies $\thlim=\theta_0$ a.e.~in $\Omega$. To show the~strong convergence, we use 
\eqref{wcctep} with $\zeta:= \sqrt{\theta_0}$ and also \eqref{eslims}, to deduce 
\begin{align*}
\limsup_{t\to0+}\big\|\sqrt{\theta(t)}-\sqrt{\theta_0}\big\|_2^2
&=\limsup_{t\to0+}\ii\theta(t)+\ii\theta_0 - 2\lim_{t\to0+}\ii\sqrt{\theta(t)}\sqrt{\theta_0}\leq0.
\end{align*}
Hence, the~above inequality implies that
\begin{equation*}
\sqrt{\theta(t)} \to \sqrt{\theta_0} \quad \textrm{ strongly in } L^{2}(\Omega; \R),
\end{equation*}
which implies \eqref{Qict}.

Using information above, we can now improve the~initial condition for $\ve$ as well. Indeed, from \eqref{Ener3}, \eqref{Qict} and \eqref{init}, we obtain
\begin{align*}
\limsup_{t\to0+}\ii\tfrac12|\ve(t)|^2&\leq\limsup_{t\to0+}\ii E(t)-\liminf_{t\to0+}\ii c_v\theta(t)\\
&\leq\ii E_0+\lim_{t\to0+}\int_0^t(\fe,\ve)-\ii c_v\theta_0=\ii\tfrac12|\ve_0|^2.
\end{align*}
Thus, using also \eqref{wccv}, we conclude that
\begin{align*}
\limsup_{t\to0+}\norm{\ve(t)-\ve_0}_2^2=\limsup_{t\to0+}\ii|\ve(t)|^2+\ii|\ve_0|^2-2\lim_{t\to0+}\ii\ve(t)\cdot\ve_0\leq0,
\end{align*}
which implies \eqref{Qicv}.

Finally, since $f$ is strictly convex on $\PD$ as
\begin{equation*}
f''(\B)\A\cdot\A=\mu\B^{-1}\A\B^{-1}\cdot\A=\mu|\B^{-\f12}\A\B^{-\f12}|^2,\quad\B\in\PD,\quad\A\in\R^{d\times d},
\end{equation*}
the strong attainment of the~initial condition for $\B$ \eqref{QicB} follows readily from \eqref{wccB}, the~classical result \cite[Theorem~3~(i)]{Visintin} and the~Vitali's theorem once we show the~property
\begin{equation}\label{supf}
\limsup_{t\to0+}\ii f(\B(t))\leq \ii f(\B_0).
\end{equation}
To this end, we make an observation that in \eqref{2test} (with \eqref{BIde} in place), we can choose $\phi=1$, drop the~non-negative terms, integrate over $(0,t)$ and then estimate the~right-hand side using H\"older inequality, \eqref{integr2} and \eqref{vrh} to obtain
\begin{equation}\label{fif}
\ii f(\B_{\ell}(t))-\ii f(\B^{\f1{\ell}}_0)\leq \int_0^t\ii 2a\mu g_{\f1{\ell}}(\B_{\ell},\theta_{\ell})\B_{\ell}\cdot\Dv_{\ell}\leq Ct^{\f1{2q'}}.
\end{equation}
Note that again we rely on \eqref{dfB} to give a~proper meaning to the~left-hand side of \eqref{fif} for all $t\in(0,T)$. Utilizing now the~convexity and continuity of $f$ on $\PD$ and \eqref{B0con}, taking the~limit $\ell\to\infty$ in \eqref{fif} leads to
\begin{equation*}
\ii f(\B(t))-\ii f(\B_0)\leq \liminf_{\ell\to\infty}\Big(\ii f(\B_{\ell}(t))-\ii f(\B^{\f1{\ell}}_0)\Big)\leq C t^{\f1{2q'}},
\end{equation*}
from which \eqref{supf} immediately follows.  

\section{Global energy equality for \texorpdfstring{$d\le 3$}{dle3}}\label{SS7}

To derive \eqref{QEloc} (which is a~weak version of \eqref{s10}), we need to construct the~pressure $\pr$ and ensure that every term appearing \eqref{QEloc} is integrable. To this end, we apply the~conditions \eqref{A0}. Moreover, we need to be able to test the~momentum equation with $\ve\phi$, where $\phi$ is some smooth function on $\Q$. Unfortunately, we can not do this operation in \eqref{Qv} nor at any stage of our approximation scheme. The~remedy is to truncate the~convection term in the~balance of momentum. However, then we are just mimicking the~existence proof that is done in \cite{Bulicek2009} for a~different non-linear fluid. Thus, let us only verify the~weak compactness of weak solutions $(\ve_{\delta},\pr_{\delta},\B_{\delta},\theta_{\delta},\eta_{\delta})$ to the~system $\di\ve_{\delta}=0$, \eqref{QB}, \eqref{Qent},
\begin{equation}\label{QvE}
\scal{\partial_t\ve_{\delta},\fit}-(T_{\delta}\ve_{\delta}\otimes\ve_{\delta},\nn\fit)_{\Q}+(\Sb_{\delta},\nn\fit)_{\Q}+\alpha(\ve_{\delta}\fit)_{\SigT}=(\pr_{\delta},\di\fit)_{\Q}+(\fe,\fit)_{\Q}
\end{equation}
for all $\fit\in L^{\infty}(0,T;W_{\n}^{1,\infty})$, with $T_{\delta}\ve_{\delta}=((\ve_{\delta}s_{\delta})*r_{\delta})_{\di}$, where $s_{\delta}$ is a~truncation near $\partial\Omega$, $r_{\delta}$ is a~standard mollifier and $(\cdot)_{\di}$ is a~Helmholtz projection onto divergence-free functions, and
\begin{align}\label{QEloc2}
&-(E_0,\phi)\varphi(0)-(E_{\delta},\phi\partial_t\varphi)_{\Q}+\alpha(|\ve_{\delta}|^2,\phi\varphi)_{\SigT}+(\kappa(\theta_{\delta})\nn\theta_{\delta},\nn\phi\varphi)_{\Q}\nonumber\\
&\qquad=(E_{\delta}\ve_{\delta}+\pr_{\delta}\ve_{\delta}-\Sb_{\delta}\ve_{\delta},\nn\phi\varphi)_{\Q}
\end{align}
for all $\varphi\in W^{1,\infty}((0,T);\R)$, $\varphi(T)=0$, and every $\phi\in W^{1,\infty}(\Omega;\R)$. The~existence of such solutions follows by combining the~approximation scheme from Section~\ref{SS1}  together with the~one in \cite{Bulicek2009}. In view of the~uniform estimates derived in Sections~\ref{SS1}--\ref{SS5}, we may suppose that the~sequence $\{(\ve_{\delta},\pr_{\delta},\B_{\delta},\theta_{\delta},\eta_{\delta})\}_{\delta>0}$ is uniformly bounded in the~spaces depicted in \eqref{DR1}--\eqref{DReta} and that we have the~same convergence results as in \eqref{ccvw}--\eqref{cctw} and so forth (with $\ell$ replaced by $\delta$). We may also suppose that, say $\pr_{\delta}\in L^2(\Q;\R)$ with $\int_{\Omega}\pr_{\delta}=0$. Then, since we have $\nu(\theta_{\delta})\Dv_{\delta},\theta_{\delta}\B_{\delta}\in L^{2}(\Q;\Sym)$ and the~convection term is truncated, equation~\eqref{QvE} is valid for all $\fit\in L^{2}(0,T;W^{1,2}_{\n})$, in fact. What is missing is the~uniform estimate of the~pressure. By localizing \eqref{QvE} in time, choosing $\fit=\nn u$ and using $\di\ve_{\delta}=0$, we obtain
\begin{equation*}
-(\pr_{\delta},\Delta u)=(T_{\delta}\ve_{\delta}\otimes\ve_{\delta}-\Sb_{\delta},\nn\nn u)-\alpha(\ve_{\delta},\nn u)_{\partial\Omega}+(\fe,\nn u)
\end{equation*}
a.e.\ in $(0,T)$. There the~convective term, if not truncated, is the~most irregular one (recall that $\norm{\ve_{\delta}\otimes\ve_{\delta}}_{L^{\f{d+2}{d}}L^{\f{d+2}{d}}}\leq C$). Thus, expecting $\pr_{\delta}$ to have the~same integrability, we may choose $u\in W^{2,(\f{d+2}{d})'}(\Omega;\R)$ to be the~solution to the~Neumann problem
\begin{align*}
-\Delta u&=|\pr_0|^{\f{d+2}{d}-2}\pr_0-\f1{|\Omega|}\ii|\pr_0|^{\f{d+2}{d}-2}\pr_0\quad\text{in }\Omega,\\
\nn u \cdot\n&=0\quad\text{on }\partial\Omega
\end{align*}
a.e.\ in $(0,T)$, where $\pr_0=\pr_{\delta}-\f1{|\Omega|}\ii \pr_{\delta}$. Since $\norm{u}_{2,(\f{d+2}{d})'}\leq C\norm{\pr_0}_{\f{d+2}{d}}$ by the~corresponding $L^q$-theory (here we used $\Omega\in\mathcal{C}^{1,1}$), the~test function $u$ eventually leads to
\begin{equation*}
\norm{\pr_{\delta}}_{L^{\f{d+2}{d}}L^{\f{d+2}{d}}}\leq C,
\end{equation*}
see \cite{Bulicek2009} for details.

Taking the~limit $\delta\to0+$ in \eqref{QvE}, \eqref{QB} and \eqref{Qent} can be done analogously as when we considered the~limit $\ell\to\infty$. Indeed, in the~additional term $\int_0^T(\pr_{\delta},\di\fit)$, we simply use the~fact that $\pr_{\delta}\wc \pr$ weakly in $L^{\f{d+2}{d}}(\Q;\R)$. It remains to take the~limit $\delta\to0+$ in \eqref{QEloc2}. Since $\ve_{\delta}$ converges strongly in $L^{2\f{d+2}{d}-\eps}(\Q;\R^d)$ and $d\le 3$, we deduce that the~terms $p_{\delta}\ve_{\delta}$ and $|\ve_{\delta}|^2\ve_{\delta}$ converge weakly to their limits. The~limits in other terms were already discussed and we omit it here. Thus, the~proof  of Theorem~\ref{LG}  is complete.

\appendix

\section{Auxiliary results}\label{CH4}

In this additional section, we prove those auxiliary results which were used above but are not completely standard in the~existing literature. On the~other hand, they are not new and serve only to clarify some arguments used in the~proof.

For the~purposes of this section, we replace the~interval $(0,T)$ (or $[0,T]$) by an~arbitrary bounded interval $I\subset \R$ and set $Q=I\times\Omega$. The~set $\Omega$ is always assumed to be a~bounded Lipschitz domain in $\R^d$, $d\in\N$.

\subsection*{Intersections of Sobolev-Bochner spaces}\label{Ssum}

If $X\overset{\rm dense}{\embl}H\overset{\rm dense}{\embl}X^*$ is a~Gelfand triple, it is well known that
\begin{equation}\label{emboo}
\mathcal{C}^1(I;X)\overset{\rm dense}{\embl}\mathcal{W}_X^p\embl\mathcal{C}(I;H),
\end{equation}
where
\begin{equation*}
\mathcal{W}_{X}^p\coloneqq \big(\{u\in L^p(I;X);\;\partial_t u\in (L^p(I;X))^*\},\norm{\cdot}_{L^pX}+\norm{\partial_t\cdot}_{L^{p'}X^*}\big),\quad 1<p<\infty.
\end{equation*}
The first embedding in \eqref{emboo} is useful to manipulate certain duality pairings involving time derivatives, while the~second embedding is important for the~identification of boundary values (i.e.\ initial conditions) and the~corresponding integration by parts formulas. We would like to generalize \eqref{emboo} for the~space
\begin{align*}
\mathcal{W}_{X,Y}^{p,q}&\coloneqq\big(\{u\in L^{p}(I;X)\cap L^{q}(I;Y);\;\partial_tu\in (L^{p}(I;X)\cap L^{q}(I;Y))^*\},\\
&\;\qquad\norm{\cdot}_{L^{p}X\cap L^{q}Y}+\norm{\partial_t \cdot}_{(L^{p}X\cap L^{q}Y)^*}\big),\quad 1<p,q<\infty,
\end{align*}
The primary application which we have in mind is the~case where $X=W^{1,2}(\Omega)$, $Y=L^{\omega}(\Omega)$ and $\omega>\f{2d}{d-2}$ (i.e., we know better integrability than what follows from the~Sobolev embedding, recall the~function $\B_{\ell}$). Thus, we may assume that both $X$ and $Y$ admit the~Gelfand triplet structure with a~common Hilbert space~$H$.

\begin{lemma}\label{Temb}
Let $1<p,q<\infty$ and suppose that $X$, $Y$ are separable reflexive Banach spaces and $H$ is separable Hilbert space forming Gelfand triples in the~sense that
\begin{equation}\label{gelf1}
X\overset{\rm dense}{\embl} H\overset{\rm dense}{\embl}X^*\quad\text{and}\quad Y\overset{\rm dense}{\embl}H\overset{\rm dense}{\embl}Y^*.
\end{equation}

Then, we have the~embeddings
\begin{equation}\label{coh}
\CO^1(I;X\cap Y)\overset{\rm dense}{\embl}\mathcal{W}_{X,Y}^{p,q}\embl \CO(I; H).
\end{equation}
Moreover, the~integration by parts formula
\begin{equation}\label{ibp}
(u(t_2),v(t_2))_H-(u(t_1),v(t_1))_H=\int_{t_1}^{t_2}\scal{\partial_tu,v}+\int_{t_1}^{t_2}\scal{\partial_tv,u}
\end{equation}
holds for any $u,v\in \mathcal{W}_{X,Y}^{p,q}$ and any $t_1,t_2\in I$.
\end{lemma}
\begin{proof}
The proof of the~first embedding in~\eqref{coh} can be done in a~standard way by extending $u$ outside $I$ evenly, taking the~convolution with a~smooth kernel and then estimating the~difference from $u$ and $\partial_t u$ in the~respective norms. See \cite{Gaevski1978} or \cite{Zeidler1990} for details.

If $u,v\in \CO^1(I;X\cap Y)\embl\CO(I;H)$, then $\partial_t u,\partial_t v\in \CO(I;X\cap Y)\embl \CO(I;H)$ and, using density of the~embeddings in~\eqref{gelf1}, the~duality in~\eqref{ibp} can be represented as
\begin{equation*}
\scal{\partial_tu,v}+\scal{\partial_tv,u}=(\partial_tu,v)_H+(\partial_tv,u)_H=\partial_t(u,v)_H\quad\text{a.e.\ in }I,
\end{equation*}
hence \eqref{ibp} is obvious in that case. Next, we can proceed as in~\cite[Lemma~7.3.]{Roubicek2013} to prove that
\begin{align}\begin{aligned}\label{OD}
\norm{u(t)}_H&\leq C(\norm{u}_{L^1H}+\norm{u}_{\mathcal{W}^{p,q}_{X,Y}})
\end{aligned}\end{align}
for all $t\in I$ and every $u\in\mathcal{C}^1(I;X\cap Y)$. Moreover, by \eqref{gelf1}, we have
\begin{equation*}
\mathcal{W}_{X,Y}^{p,q}\embl L^{p}(I;X)\cap L^{q}(I;Y)\embl L^1(I;X)\cap L^1(I;Y)\embl L^1(I;X + Y)\embl L^1(I;H),
\end{equation*}
and thus \eqref{OD} yields
\begin{equation}\label{vno}
\norm{u}_{\CO(I;H)}\leq C\norm{u}_{\mathcal{W}_{X,Y}^{p,q}}.
\end{equation}
Since $\CO^1(I;X\cap Y)$ is dense in $\mathcal{W}_{X,Y}^{p,q}$, the~estimate \eqref{vno} and identity~\eqref{ibp} remain valid for all $u\in \mathcal{W}_{X,Y}^{p,q}$. Moreover, if $u\in \mathcal{W}_{X,Y}^{p,q}$, then we can take $v=u$ and $t_2\to t_1$ in~\eqref{ibp} to deduce that $u\in \CO(I;H)$. Thus, the~embedding $\mathcal{W}_{X,Y}^{p,q}\embl \CO(I;H)$ holds and the~proof is finished.
\end{proof}

Since $\mathcal{W}_{X,X}^{p,p}=\mathcal{W}_X^p$, the~classical result \eqref{emboo} can be seen as an~obvious corollary.

\subsection*{Fundamental theorem of calculus in the~Sobolev-Bochner setting}

Let $H=L^2(\Omega)$. The~formula \eqref{ibp} can be used to identify that
\begin{equation}\label{s}
\scal{\partial_t u,u}=\f12\f{\dd{}}{\dd{t}}\ii u^2
\end{equation}
a.e.\ in $I$. However, in certain situations we would like to generalize \eqref{s} to
\begin{equation*}
\scal{\partial_t u,\psi(u)}=\f{\dd{}}{\dd{t}}\ii\int_{w}^{u}\psi(s)\dd{s}.
\end{equation*}
Whether this is possible depends on what kind of function $\psi$ is and also on the~choice of $X$. The~next lemma characterizes one such situation.

\begin{lemma}\label{Ldual}
Let $1<p,q<\infty$. Suppose that $\psi:\R\to\R$ is a~Lipschitz function. For $w\in\R$, we define
\begin{equation*}
\Psi(x)=\int_{w}^x\psi(s)\dd{s},\quad x\in\R.
\end{equation*}

Then, for any $u\in\mathcal{W}^p_{W^{1,q}(\Omega)}$, there holds
\begin{equation}\label{Psic}
\Psi(u)\in \CO(I;L^1(\Omega))
\end{equation}
and
\begin{equation}\label{vzra}
\int_{t_1}^{t_2}\scal{\partial_t u,\psi(u)}=\ii\Psi(u(t_2))-\ii\Psi(u(t_1))\quad\text{for all }t_1,t_2\in I.
\end{equation}

Moreover, if $\psi$ is bounded, then
\begin{equation*}
\Psi(u)\in\CO(I;L^2(\Omega)).
\end{equation*}
\end{lemma}
\begin{proof}
First of all, we remark that $\psi(u)\in W^{1,q}(\Omega)$ a.e.\ in $I$, by a~classical result (see e.g. \cite[Theorem~2.1.11.]{Ziemer1989}), and thus the~duality in~\eqref{vzra} is well defined. Next, we apply Theorem~\ref{Temb} to find $u_{\eps}\in\CO^1(I;W^{1,q}(\Omega))$ satisfying
\begin{equation}\label{stro}
\norm{u_{\eps}-u}_{L^pW^{1,q}}+\norm{\partial_tu_{\eps}-\partial_tu}_{L^{p'}W^{-1,q'}}\to0\quad\text{as }\eps\to0+.
\end{equation}
Then, using the~standard calculus, it is easy to see that the~identity
\begin{equation}\begin{aligned}\label{ideap}
\int_{t_1}^{t_2}\scal{\partial_t u_{\eps},\psi(u_{\eps})}&=\int_{t_1}^{t_2}\ii\psi(u_{\eps})\partial_tu_{\eps}\\
&=\int_{t_1}^{t_2}\ii\partial_t\Psi(u_{\eps})=\ii\Psi(u_{\eps}(t_2))-\ii\Psi(u_{\eps}(t_1))
\end{aligned}\end{equation}
holds for any $t_1,t_2\in I$. Denoting the~Lipschitz constant of $\psi$ by $L\geq0$, we estimate
\begin{equation*}
|\psi(u_{\eps})|\leq |\psi(u_{\eps})-\psi(0)|+|\psi(0)|\leq L|u_{\eps}|+|\psi(0)|
\end{equation*}
and
\begin{equation*}
|\nn\psi(u_{\eps})|\leq|\psi'(u_{\eps})||\nn u_{\eps}|\leq L|\nn u_{\eps}|.
\end{equation*}
Hence, the~sequence $\psi(u_{\eps})$ is bounded in $L^p(I;W^{1,q}(\Omega))$. As $1<p,q<\infty$, this is a~separable reflexive space, and thus, there exist a~subsequence and its limit $\overline{\psi(u)}\in L^p(I;W^{1,q}(\Omega))$ such that
\begin{equation}\label{wea}
\psi(u_{\eps})\wc\overline{\psi(u)}\quad\text{weakly in }L^p(I;W^{1,q}(\Omega)).
\end{equation}
Since $p>1$, a~subsequence of $u_{\eps}$ converges point-wise a.e.\ in $Q$ to $u$, and thus $\overline{\psi(u)}=\psi(u)$ using the~continuity of $\psi$. Hence, by \eqref{stro} and \eqref{wea}, we obtain
\begin{equation}\begin{aligned}\label{duali}
\int_{t_1}^{t_2}\scal{\partial_t u_{\eps},\psi(u_{\eps})}
&=\int_{t_1}^{t_2}\scal{\partial_t u_{\eps}-\partial_t u,\psi(u_{\eps})}+\int_{t_1}^{t_2}\scal{\partial_t u,\psi(u_{\eps})}\\
&\to\int_{t_1}^{t_2}\scal{\partial_tu,\psi(u)}
\end{aligned}\end{equation}
as $\eps\to0+$. Next, using the~embedding $\mathcal{W}^{p}_{W^{1,q}(\Omega)}\embl\CO(I;L^2(\Omega))$ and \eqref{stro}, we get, for any $t_0\in I$, that
\begin{equation}\label{stejno}
\norm{u(t)-u(t_0)}_2\to0\quad\text{as}\quad t\to t_0
\end{equation}
and
\begin{equation}\label{normo}
\norm{u_{\eps}(t_0)-u(t_0)}_{2}\to0\quad\text{as}\quad\eps\to{0+}.
\end{equation}
Then, the~Lipschitz continuity of $\psi$, H\"older's inequality and \eqref{stejno} yield
\begin{align}\label{esssa}
&\ii|\Psi(u(t))-\Psi(u(t_0))|=\ii\left|\int_{u(t_0)}^{u(t)}\psi(s)\dd{s}\right|\leq\ii \int_{u(t_0)}^{u(t)}(|\psi(0)|+L|s|)\dd{s}\nonumber\\[-.1cm]
&\quad\leq\ii\int_{u(t_0)}^{u(t)}C(1+|u(t_0)|+|u(t)|)\leq C\ii(1+|u(t_0)|+|u(t)|)|u(t)-u(t_0)|\nonumber\\
&\quad\leq C\norm{1+|u(t_0)|+|u(t)|}_2\norm{u(t)-u(t_0)}_2\leq C\norm{u(t)-u(t_0)}_2\to0
\end{align}
as $t\to t_0$, which proves \eqref{Psic} (and thus, the~values $\Phi(u(t))$, $t\in I$, are well defined). By an~analogous estimate, using \eqref{normo} instead of \eqref{stejno}, we can prove that
\begin{equation*}
\ii|\Psi(u_{\eps}(t_0))-\Psi(u(t_0))|\to0\quad \text{as }\eps\to{0+}
\end{equation*}
for any $t\in I$. This and \eqref{duali} used in~\eqref{ideap} to take the~limit $\eps\to0+$ proves \eqref{vzra}.

If $\psi$ is bounded, we replace \eqref{esssa} by
\begin{equation*}
\ii|\Psi(u(t))-\Psi(u(t_0))|^2=\ii\left|\int_{u(t_0)}^{u(t)}\psi(s)\dd{s}\right|^2\leq C\ii|u(t)-u(t_0)|^2
\end{equation*}
and the~rest of the~proof remains the~same.
\end{proof}

Clearly, we can also replace $\psi$ by $\psi\phi$, where $\phi\in W^{1,\infty}(\Omega;\R)$, leading to
\begin{equation}\label{vzra2}
\int_0^t\scal{\partial_t u,\psi(u)\phi}=\ii\int_{w}^{u(t)}\psi(s)\dd{s}\,\phi-\ii\int_{w}^{u(0)}\psi(s)\dd{s}\,\phi\quad\text{for all }t\in I.
\end{equation}
Then, since $\phi$ is a~Lipschitz (time independent) function, the~proof is basically the~same as the~one presented above.

\subsection*{Calculus for positive definite matrices}

We recall that the~operations ``$\cdot$'' and $|\cdot|$ on matrices are defined by
\begin{equation*}
\A_1\cdot\A_2=\sum_{i=1}^d\sum_{j=1}^d(\A_1)_{ij}(\A_2)_{ij}\quad\text{and}\quad|\A|=\sqrt{\A\cdot\A},
\end{equation*}
respectively. Then, the~object $|\A|$ coincides, in fact, with the~Frobenius matrix norm of $\A$.

The next lemma is formulated for a~function $\A:Q\to\PD$ and for simplicity, we shall assume that $\A$ is continuously differentiable with respect to all variables, i.e., $\A\in\mathcal{C}^{1}(Q;\PD)$. In particular situations, this assumption can be of course removed by an appropriate approximation (convolution smoothing) and the~assertions of the~following lemma hereby extend to the~setting of weakly differentiable functions. Let us also denote any of the~space-time derivatives by a~generic symbol $\partial$.

\begin{lemma}\label{PDalg}
Let $\A\in\mathcal{C}^1(Q;\PD)$. Then
\begin{align}
&{\rm(i)} &0&\leq\tr\A-d-\ln\det\A,&&\label{II}\\
&{\rm(ii)} &|\A|&\leq\tr\A\leq\sqrt{d}|\A|,&&\label{I7}\\
&{\rm(iii)} &\min\{1,d^{\f{1-\alpha}2}\}|\A|^{\alpha}&\leq|\A^{\alpha}|\leq\max\{1,d^{\f{1-\alpha}2}\}|\A|^{\alpha}\quad\text{for any}\quad\alpha\geq0,&&\label{I8}\\
&{\rm(iv)} &\partial\A\cdot\A^{\alpha}&=\Big\{\begin{matrix}
\f1{\alpha+1}\partial\tr\A^{\alpha+1}&\text{if}\quad\alpha\neq-1;\\
\partial\ln\det\A=\partial\tr\log\A&\text{if}\quad\alpha=-1,
\end{matrix}&&\label{IJ1}\\
&{\rm(v)} &({\rm sign}\,\alpha)\partial\A\cdot\partial\A^{\alpha}&\geq\Big\{\begin{matrix}
\f{4|\alpha|}{(\alpha+1)^2}|\partial\A^{\f{\alpha+1}2}|^2&\text{if}\quad\alpha\neq-1;\\
|\partial\log\A|^2&\text{if}\quad\alpha=-1,
\end{matrix}&&\label{IJ2}\\
&{\rm(vi)} &|\partial\A|&\leq2|\A^{1-\alpha}\partial\A^{\alpha}|\quad\text{for all}\quad\alpha\in[\tfrac12,1).&&\label{IJ3}
\end{align}
\end{lemma}
\begin{proof}
Property (i) follows by passing to the~spectral decomposition of $\A$ and from the~fact that $x\mapsto x-1-\ln x$ attains its minimum at $x=1$. Estimate (ii) is a~consequence of the~Cauchy--Schwarz inequality since
\begin{align*}
|\A|&=|(\A^{\f12})^T\A^{\f12}|\leq|\A^{\f12}|^2=\tr\A=\I\cdot\A\leq|\I||\A|=\sqrt{d}|\A|.
\end{align*}
For (iii), we refer to \cite[Proposition~1]{BathoryOdh} and for (iv), (v) to \cite[Theorem~1]{BathoryOdh}. The~relation (iv) with $\alpha=-1$ is also known as the~Jacobi identity.

Finally, property \eqref{IJ3} can be shown using the~idea from the~proof of \cite[Theorem~3]{BathoryOdh}, which we briefly sketch here. For any natural numbers $p,q$ satisfying $q-p\leq p<q$ (so that $\alpha=\f pq\in[\f12,1)$), we may use the~Young inequality to write
\begin{equation}
|\partial\B^q|^2=|\partial\B^{q-p}\B^p+\B^{q-p}\partial\B^p|^2\leq2|\partial\B^{q-p}\B^p|^2+2|\B^{q-p}\partial\B^p|^2=:2A+2B.\label{AB}
\end{equation}
Now we simply expand the~derivative and rearrange the~terms to get
\begin{align}
A&=\Big|\sum_{i=0}^{q-p-1}\B^i\partial\B\B^{q-1-i}\Big|^2=\sum_{i=0}^{q-p-1}\sum_{j=0}^{q-p-1}\big|\B^{\f{i+j}2}\partial\B\B^{q-1-\f{i+j}2}\big|^2\nonumber\\
&=\sum_{s=0}^{2(q-p-1)}(1+\min\{s,2(q-p-1)-s\})\big|\B^{\f s2}\partial\B\B^{q-1-\f{s}2}\big|^2\label{Aest},
\end{align}
whereas for $B$, a~completely analogous computation yields
\begin{equation*}
B=\sum_{s=0}^{2(p-1)}(1+\min\{s,2(p-1)-s\})\big|\B^{\f s2}\partial\B\B^{q-1-\f{s}2}\big|^2.
\end{equation*}
Then, using $q-p\leq p$ first inside the~minimum in \eqref{Aest} and then in the~number of terms of the~sum (relying on the~non-negativity of each term), we see that $A\leq B$. Returning with this information to \eqref{AB} and setting $\B=\A^{\f1q}$, we easily conclude \eqref{IJ3} for rational powers $\alpha$. The~general case follows by a~density argument (the continuity of the~mapping $\alpha\mapsto\A^{1-\alpha}$ follows immediately from the~spectral decomposition, while continuity of $\alpha\mapsto\partial\A^{\alpha}$ is a~consequence of the~integral representation formula for $\partial\exp\mathbb{X}$, see \cite{Wilcox_1967} or \cite{BathoryOdh} for more details).

\end{proof}


\begin{thebibliography}{10}
	
	\bibitem{bio4}
	{\sc N.~Arada and A.~Sequeira}, {\em Strong {S}teady {S}olutions for a
		{G}eneralized {O}ldroyd-{B} {M}odel with {S}hear-{D}ependent {V}iscosity in a
		{B}ounded {D}omain}, Mathematical Models and Methods in Applied Sciences, 13
	(2003), pp.~1303--1323.
	
	\bibitem{BARRETT2011}
	{\sc J.~W. Barrett and S.~Boyaval}, {\em {E}xistence and approximation of a
		(regularized) {O}ldroyd-{B} model}, Mathematical Models and Methods in
	Applied Sciences, 21 (2011), pp.~1783--1837.
	
	\bibitem{BathoryOdh}
	{\sc M.~Bathory}, {\em Sharp nonlinear estimates for multiplying derivatives of
		positive definite tensor fields}, Mathematical Inequalities {\&}
	Applications, 25 (2022), pp.~751--769.
	
	\bibitem{Bathory_2020}
	{\sc M.~Bathory, M.~Bul{\'{\i}}{\v{c}}ek, and J.~M{\'{a}}lek}, {\em Large data
		existence theory for three-dimensional unsteady flows of rate-type
		viscoelastic fluids with stress diffusion}, Advances in Nonlinear Analysis,
	10 (2020), pp.~501--521.
	
	\bibitem{Blechta2019}
	{\sc J.~Blechta, J.~M\'{a}lek, and K.~R. Rajagopal}, {\em On the
		{C}lassification of {I}ncompressible {F}luids and a {M}athematical {A}nalysis
		of the {E}quations {T}hat {G}overn {T}heir {M}otion}, SIAM J. Math. Anal., 52
	(2020), pp.~1232--1289.
	
	\bibitem{Bulicek2019}
	{\sc M.~Bul{\'{i}}{\v{c}}ek, E.~Feireisl, and J.~M{\'{a}}lek}, {\em On a class
		of compressible viscoelastic rate-type fluids with stress-diffusion},
	Nonlinearity, 32 (2019), pp.~4665--4681.
	
	\bibitem{Bulicek2009}
	{\sc M.~Bul{\'{i}}{\v{c}}ek, J.~M{\'{a}}lek, and K.~R. Rajagopal}, {\em
		Mathematical analysis of unsteady flows of fluids with pressure, shear-rate,
		and temperature dependent material moduli that slip at solid boundaries},
	{SIAM} Journal on Mathematical Analysis, 41 (2009), pp.~665--707.
	
	\bibitem{BuFei2009}
	{\sc M.~Bul\'{\i}\v{c}ek, E.~Feireisl, and J.~M\'{a}lek}, {\em A
		{N}avier-{S}tokes-{F}ourier system for incompressible fluids with temperature
		dependent material coefficients}, Nonlinear Anal. Real World Appl., 10
	(2009), pp.~992--1015.
	
	\bibitem{Havrda}
	{\sc M.~Bul\'{\i}\v{c}ek and J.~Havrda}, {\em On existence of weak solution to
		a model describing incompressible mixtures with thermal diffusion cross
		effects}, ZAMM Z. Angew. Math. Mech., 95 (2015), pp.~589--619.
	
	\bibitem{los1}
	{\sc M.~Bul\'{\i}\v{c}ek, T.~Los, Y.~Lu, and J.~M\'{a}lek}, {\em On planar
		flows of viscoelastic fluids of {G}iesekus type}, Nonlinearity, 35 (2022),
	pp.~6557--6604.
	
	\bibitem{stick1}
	{\sc M.~Bul\'{i}\v{c}ek and J.~M\'{a}lek}, {\em Internal flows of
		incompressible fluids subject to stick-slip boundary conditions}, Vietnam J.
	Math., 45 (2017), pp.~207--220.
	
	\bibitem{stick2}
	\leavevmode\vrule height 2pt depth -1.6pt width 23pt, {\em Large data analysis
		for {K}olmogorov's two-equation model of turbulence}, Nonlinear Anal. Real
	World Appl., 50 (2019), pp.~104--143.
	
	\bibitem{Bulicek2018}
	{\sc M.~Bul\'{i}\v{c}ek, J.~M\'{a}lek, V.~Pr\r{u}\v{s}a, and E.~S\"{u}li}, {\em
		P{DE} analysis of a class of thermodynamically compatible viscoelastic
		rate-type fluids with stress-diffusion}, in Mathematical analysis in fluid
	mechanics---selected recent results, vol.~710 of Contemp. Math., Amer. Math.
	Soc., Providence, RI, 2018, pp.~25--51.
	
	\bibitem{BulPreprint2020}
	{\sc M.~Bul\'{\i}\v{c}ek, J.~M\'{a}lek, V.~Pr\r{u}\v{s}a, and E.~S\"{u}li},
	{\em On incompressible heat-conducting viscoelastic rate-type fluids with
		stress-diffusion and purely spherical elastic response}, SIAM J. Math. Anal.,
	53 (2021), pp.~3985--4030.
	
	\bibitem{BulicekIndiana}
	{\sc M.~Bul\'{\i}\v{c}ek, J.~M\'{a}lek, and K.~R. Rajagopal}, {\em Navier's
		slip and evolutionary {N}avier-{S}tokes-like systems with pressure and
		shear-rate dependent viscosity}, Indiana University Mathematics Journal, 56
	(2007), pp.~51--85.
	
	\bibitem{casey22}
	{\sc M.~Bul\'{\i}\v{c}ek, J.~M\'{a}lek, and C.~Rodriguez}, {\em Global
		well-posedness for two-dimensional flows of viscoelastic rate-type fluids
		with stress diffusion}, J. Math. Fluid Mech., 24 (2022), pp.~Paper No. 61,
	19.
	
	\bibitem{Callen}
	{\sc H.~B. Callen}, {\em Thermodynamics and an introduction to
		thermostatistics}, Wiley, New York, second~ed., 1985.
	
	\bibitem{Chupin2018}
	{\sc L.~Chupin}, {\em Global {S}trong {S}olutions for {S}ome {D}ifferential
		{V}iscoelastic {M}odels}, {SIAM} Journal on Applied Mathematics, 78 (2018),
	pp.~2919--2949.
	
	\bibitem{Coddington1955}
	{\sc E.~A. Coddington and N.~Levinson}, {\em Theory of ordinary differential
		equations}, McGraw-Hill Book Company, Inc., New York-Toronto-London, 1955.
	
	\bibitem{Constantin2012}
	{\sc P.~Constantin and M.~Kliegl}, {\em Note on {G}lobal {R}egularity for
		{T}wo-{D}imensional {O}ldroyd-{B} {F}luids with {D}iffusive {S}tress},
	Archive for Rational Mechanics and Analysis, 206 (2012), pp.~725--740.
	
	\bibitem{Dostalik2019}
	{\sc M.~Dostal\'\i{}k, V.~Pr{\r{u}}\v{s}a, and T.~Sk\v{r}ivan}, {\em On
		diffusive variants of some classical viscoelastic rate-type models}, AIP
	Conference Proceedings, 2107 (2019).
	
	\bibitem{Dressler1999}
	{\sc M.~Dressler, B.~J. Edwards, and H.~C. Öttinger}, {\em Macroscopic
		thermodynamics of flowing polymeric liquids}, Rheologica Acta, 38 (1999),
	p.~117 – 136.
	
	\bibitem{Feireisl2012}
	{\sc E.~Feireisl}, {\em Relative entropies in thermodynamics of complete fluid
		systems}, Discrete \& Continuous Dynamical Systems - A, 32 (2012),
	pp.~3059--3080.
	
	\bibitem{Fremond_2012}
	{\sc E.~Feireisl, M.~Fr{\'{e}}mond, E.~Rocca, and G.~Schimperna}, {\em A new
		approach to non-isothermal models for nematic liquid crystals}, Archive for
	Rational Mechanics and Analysis, 205 (2012), pp.~651--672.
	
	\bibitem{Feireisl06}
	{\sc E.~Feireisl and J.~M\'{a}lek}, {\em On the {N}avier-{S}tokes equations
		with temperature-dependent transport coefficients}, Differ. Equ. Nonlinear
	Mech.,  (2006), pp.~Art. ID 90616, 14.
	
	\bibitem{Feireisl2016}
	{\sc E.~Feireisl, A.~Novotný, and Y.~Sun}, {\em On the motion of viscous,
		compressible, and heat-conducting liquids}, Journal of Mathematical Physics,
	57 (2016), p.~083101.
	
	\bibitem{Gaevski1978}
	{\sc H.~Gajewski, K.~Gr\"{o}ger, and K.~Zacharias}, {\em Nichtlineare
		{O}peratorgleichungen und {O}peratordifferentialgleichungen},
	Akademie-Verlag, Berlin, 1974.
	\newblock Mathematische Lehrb\"{u}cher und Monographien, II. Abteilung,
	Mathematische Monographien, Band 38.
	
	\bibitem{Gordon1975}
	{\sc R.~J. Gordon and W.~R. Schowalter}, {\em Anisotropic fluid theory: A
		different approach to the dumbbell theory of dilute polymer solutions},
	Transactions of the Society of Rheology, 16 (1972), pp.~79--97.
	
	\bibitem{Saut1990}
	{\sc C.~Guillop\'{e} and J.-C. Saut}, {\em Existence results for the flow of
		viscoelastic fluids with a differential constitutive law}, Nonlinear
	Analysis. Theory, Methods \& Applications. An International Multidisciplinary
	Journal, 15 (1990), pp.~849--869.
	
	\bibitem{HePaRe19}
	{\sc M.~Heida, R.~I.~A. Patterson, and D.~R.~M. Renger}, {\em Topologies and
		measures on the space of functions of bounded variation taking values in a
		{B}anach or metric space}, J. Evol. Equ., 19 (2019), pp.~111--152.
	
	\bibitem{Hron2017}
	{\sc J.~Hron, V.~Milo{\v{s}}, V.~Pr{\r{u}}{\v{s}}a, O.~Sou{\v{c}}ek, and
		K.~T{\r{u}}ma}, {\em On thermodynamics of incompressible viscoelastic rate
		type fluids with temperature dependent material coefficients}, International
	Journal of Non-Linear Mechanics, 95 (2017), pp.~193--208.
	
	\bibitem{Ireka2013}
	{\sc I.~Ireka and T.~Chinyoka}, {\em Non-isothermal flow of a johnson-segalman
		liquid in a lubricated pipe with wall slip}, Journal of Non-Newtonian Fluid
	Mechanics, 192 (2013), p.~20 – 28.
	
	\bibitem{Johnson1977}
	{\sc M.~Johnson and D.~Segalman}, {\em A model for viscoelastic fluid behavior
		which allows non-affine deformation}, Journal of Non-Newtonian Fluid
	Mechanics, 2 (1977), pp.~255--270.
	
	\bibitem{Kreml}
	{\sc O.~Kreml}, {\em Mathematical {A}nalysis of {M}odels for {V}iscoelastic
		{F}luids}, {P}h.{D}. thesis, Charles University, 2010.
	
	\bibitem{Pokorny}
	{\sc O.~Kreml, M.~Pokorn\'{y}, and P.~\v{S}alom}, {\em On the global existence
		for a regularized model of viscoelastic non-{N}ewtonian fluid}, Colloquium
	Mathematicum, 139 (2015), pp.~149--163.
	
	\bibitem{Leonov1976}
	{\sc A.~Leonov}, {\em Nonequilibrium thermodynamics and rheology of
		viscoelastic polymer media}, Rheologica Acta, 15 (1976), pp.~85--98.
	
	\bibitem{Lions2000}
	{\sc P.~L. Lions and N.~Masmoudi}, {\em Global solutions for some oldroyd
		models of non-newtonian flows}, Chinese Annals of Mathematics, 21 (2000),
	pp.~131--146.
	
	\bibitem{Lu_2020}
	{\sc Y.~Lu and M.~Pokorn\'{y}}, {\em Global existence of large data weak
		solutions for a simplified compressible {O}ldroyd-{B} model without stress
		diffusion}, Analysis in Theory and Applications, 36 (2020), pp.~348--372.
	
	\bibitem{Lukacova-Medvidova2017}
	{\sc M.~Luk{\'{a}}{\v{c}}ov{\'{a}}-Medvi{\v{d}}ov{\'{a}}, H.~Mizerov{\'{a}},
		{\v{S}}.~Ne{\v{c}}asov{\'{a}}, and M.~Renardy}, {\em Global existence result
		for the generalized {P}eterlin viscoelastic model}, {SIAM} Journal on
	Mathematical Analysis, 49 (2017), pp.~2950--2964.
	
	\bibitem{Malek1996}
	{\sc J.~M{\'{a}}lek, J.~Ne{\v{c}}as, M.~Rokyta, and
		M.~R{\r{u}}{\v{z}}i{\v{c}}ka}, {\em {W}eak and {M}easure-valued {S}olutions
		to {E}volutionary {PDE}s}, Chapman \& Hall, 1996.
	
	\bibitem{MaPr2018}
	{\sc J.~M\'{a}lek and V.~Pr\r{u}\v{s}a}, {\em Derivation of equations for
		continuum mechanics and thermodynamics of fluids}, in Handbook of
	mathematical analysis in mechanics of viscous fluids, Springer, Cham, 2018,
	pp.~3--72.
	
	\bibitem{Malek2018}
	{\sc J.~M{\'{a}}lek, V.~Pr{\r{u}}{\v{s}}a, T.~Sk{\v{r}}ivan, and E.~S\"{u}li},
	{\em Thermodynamics of viscoelastic rate-type fluids with stress diffusion},
	Physics of Fluids, 30 (2018).
	
	\bibitem{Malek2005}
	{\sc J.~M\'{a}lek and K.~R. Rajagopal}, {\em Mathematical issues concerning the
		{N}avier-{S}tokes equations and some of its generalizations}, in Evolutionary
	equations. {V}ol. {II}, Handb. Differ. Equ., Elsevier/North-Holland,
	Amsterdam, 2005, pp.~371--459.
	
	\bibitem{MaRaTu15}
	{\sc J.~M\'{a}lek, K.~R. Rajagopal, and K.~T{\r{u}}ma}, {\em On a variant of
		the {M}axwell and {O}ldroyd-{B} models within the context of a thermodynamic
		basis}, International Journal of Non-Linear Mechanics, 76 (2015), pp.~42 --
	47.
	
	\bibitem{MaRaTu18}
	\leavevmode\vrule height 2pt depth -1.6pt width 23pt, {\em Derivation of the
		variants of the {B}urgers model using a thermodynamic approach and appealing
		to the concept of evolving natural configurations}, Fluids, 3 (2018).
	
	\bibitem{Masmoudi2011}
	{\sc N.~Masmoudi}, {\em Global existence of weak solutions to macroscopic
		models of polymeric flows}, Journal de Math{\'{e}}matiques Pures et
	Appliqu{\'{e}}es, 96 (2011), pp.~502--520.
	
	\bibitem{Necas}
	{\sc J.~Ne\v{c}as}, {\em Direct methods in the theory of elliptic equations},
	Springer Monographs in Mathematics, Springer, Heidelberg, 2012.
	\newblock Translated from the 1967 French original by Gerard Tronel and Alois
	Kufner, Editorial coordination and preface by \v{S}\'{a}rka Ne\v{c}asov\'{a}
	and a contribution by Christian G. Simader.
	
	\bibitem{Olmsted_2000}
	{\sc P.~D. Olmsted, O.~Radulescu, and C.-Y.~D. Lu}, {\em
		Johnson{\textendash}segalman model with a diffusion term in cylindrical
		couette flow}, Journal of Rheology, 44 (2000), pp.~257--275.
	
	\bibitem{Pivokonsky_2015}
	{\sc R.~Pivokonsk\'{y}, P.~Filip, and J.~Zelenkov\'{a}}, {\em The role of the
		gordon{\textendash}schowalter derivative term in the constitutive
		models{\textemdash}improved flexibility of the modified {XPP} model}, Colloid
	and Polymer Science, 293 (2015), pp.~1227--1236.
	
	\bibitem{RS_2000}
	{\sc K.~R. Rajagopal and A.~R. Srinivasa}, {\em A thermodynamic frame work for
		rate type fluid models}, J. Non-Newton. Fluid Mech., 88 (2000), pp.~207--227.
	
	\bibitem{RS_2004}
	\leavevmode\vrule height 2pt depth -1.6pt width 23pt, {\em On thermomechanical
		restrictions of continua}, Proc. R. Soc. Lond., Ser. A, Math. Phys. Eng.
	Sci., 460 (2004), pp.~631--651.
	
	\bibitem{Rao2002}
	{\sc I.~J. Rao and K.~R. Rajagopal}, {\em A thermodynamic framework for the
		study of crystallization in polymers}, Zeitschrift für angewandte Mathematik
	und Physik ZAMP, 53 (2002), pp.~365--406.
	
	\bibitem{Roubicek2013}
	{\sc T.~Roub\'{\i}\v{c}ek}, {\em Nonlinear partial differential equations with
		applications}, vol.~153 of International Series of Numerical Mathematics,
	Birkh\"{a}user/Springer Basel AG, Basel, second~ed., 2013.
	
	\bibitem{Tanner2009}
	{\sc R.~I. Tanner}, {\em The changing face of rheology}, Journal of
	Non-Newtonian Fluid Mechanics, 157 (2009), p.~141 – 144.
	
	\bibitem{Visintin}
	{\sc A.~Visintin}, {\em Strong convergence results related to strict
		convexity}, Communications in Partial Differential Equations, 9 (1984),
	pp.~439--466.
	
	\bibitem{Wapperom1998}
	{\sc P.~Wapperom and M.~Hulsen}, {\em Thermodynamics of viscoelastic fluids:
		the temperature equation}, Journal of Rheology, 42 (1998), pp.~999--1019.
	
	\bibitem{Wilcox_1967}
	{\sc R.~M. Wilcox}, {\em Exponential operators and parameter differentiation in
		quantum physics}, Journal of Mathematical Physics, 8 (1967), pp.~962--982.
	
	\bibitem{Wouk_1965}
	{\sc A.~Wouk}, {\em Integral representation of the logarithm of matrices and
		operators}, Journal of Mathematical Analysis and Applications, 11 (1965),
	pp.~131--138.
	
	\bibitem{Zaremba1903}
	{\sc S.~Zaremba}, {\em Sur une forme perfectionee de la theorie de la
		relaxation}, Bull. Int. Acad. Sci. Cracovie,  (1903), pp.~594--614.
	
	\bibitem{Zeidler1990}
	{\sc E.~Zeidler}, {\em Nonlinear functional analysis and its applications.
		{II}/{A}}, Springer-Verlag, New York, 1990.
	\newblock Linear monotone operators, Translated from the German by the author
	and Leo F. Boron.
	
	\bibitem{Ziemer1989}
	{\sc W.~P. Ziemer}, {\em Weakly differentiable functions}, vol.~120 of Graduate
	Texts in Mathematics, Springer-Verlag, New York, 1989.
	\newblock Sobolev spaces and functions of bounded variation.
	
\end{thebibliography}
\end{document}